\documentclass[ reqno, 11pt]{amsart}
\usepackage{setspace}
\usepackage{multicol}
\usepackage[lite]{amsrefs} 
\usepackage{amssymb} 
\usepackage{graphicx} 
\usepackage{euscript}
\usepackage{color}
\usepackage{array}
\usepackage{picture}
\usepackage{epic}
\usepackage{tikz}
 \usetikzlibrary{backgrounds,cd}
\usepackage{comment}
\usepackage{enumerate}
\usepackage{hyperref}
\usepackage[margin=1.3in]{geometry}
\usepackage{MnSymbol}

\newcommand{\bracenom}{\genfrac{\lbrace}{\rbrace}{0pt}{}}

\newcommand{\plucker}{\genfrac{\{}{\}}{0pt}{}}

\newcommand{\isom}{\cong}
\newcommand{\set}[1]{\left\{#1\right\}} 

\newcolumntype{H}{>{\setbox0=\hbox\bgroup}c<{\egroup}@{}}

\DeclareMathOperator{\Spec}{Spec}
\DeclareMathOperator{\Proj}{Proj}
\DeclareMathOperator{\lcm}{lcm}
\DeclareMathOperator{\cl}{cl}

\DeclareMathOperator{\Cl}{cl}

\DeclareMathOperator{\Sp}{Sp}

\newcommand{\fb}{{\mathfrak b}}

\newcommand{\fg}{{\mathfrak g}}
\newcommand{\fh}{{\mathfrak h}}

\newcommand{\fl}{{\mathfrak l}}

\newcommand{\fn}{{\mathfrak n}}

\newcommand{\fu}{{\mathfrak u}}

\newcommand{\ga}{\alpha}
\newcommand{\gb}{\beta}
\newcommand{\gd}{\delta}

\renewcommand{\ge}{\varepsilon}
\renewcommand{\gg}{\gamma}

\newcommand{\gl}{\lambda}

\newcommand{\cF}{{\mathcal F}}
\newcommand{\cL}{{\mathcal L}}

\newcommand{\cS}{\mathcal S}

\newcommand{\bC}{\mathrm{\bf C}}           
\newcommand{\C}{\bC}
\newcommand{\bD}{\mathrm{\bf D}}

\newcommand{\bZ}{\mathrm{\bf Z}}           
\newcommand{\Z}{\bZ}
\newcommand{\bF}{\mathrm{\bf F}}           
\newcommand{\F}{\bF}
\newcommand{\bQ}{\mathrm{\bf Q}}           

\newcommand{\bH}{\mathrm{\bf H}}
         
\newcommand{\bP}{\mathrm{\bf P}}

\renewcommand{\H}{\mathbb{H}}
\newcommand{\T}{\mathbb{T}}

\numberwithin{equation}{section}

\theoremstyle{plain} 
\newtheorem{theorem}{Theorem}[section]
\newtheorem{corollary}[theorem]{Corollary}
\newtheorem{lemma}[theorem]{Lemma}
\newtheorem{proposition}[theorem]{Proposition}
\newtheorem*{proposition*}{Proposition}

\theoremstyle{definition}

\theoremstyle{remark}
\newtheorem{remark}[theorem]{Remark}

\newtheorem{example}[theorem]{Example}

\hypersetup{bookmarksdepth=2}

\makeatletter
\def\l@subsection{\@tocline{2}{0pt}{2.5pc}{5pc}{}}
\makeatother

\makeatletter
\def\l@subsubsection{\@tocline{3}{
0pt}{5pc}{5pc}{}}
\makeatother

\usepackage{amsmath,amscd}

\begin{document}
\parskip=4pt
\baselineskip=14pt

\title[Positivity in Weighted Flag Varieties]{Positivity in Weighted Flag Varieties}
\author{William Graham}
\address{
Department of Mathematics,
University of Georgia,
Athens, GA 30602
}
\email{wag@uga.edu}

\author{Scott Joseph Larson}
\address{
Department of Mathematics,
University of Georgia,
Athens, GA 30602
}
\email{Scott.Larson@uga.edu}
\maketitle

\begin{abstract}
We study the torus-equivariant cohomology of weighted flag varieties, and prove a positivity property in the equivariant cohomology 
and Chow groups of weighted flag varieties, analogous to the non-weighted positivity proved in \cite{Graham2001}.
Our result strengthens and generalizes the positivity proved for weighted Grassmannians by 
Abe-Matsumura \cite{AbeMatsumura2015}.  The positivity property is expressed in terms of weighted roots, which are used to describe weights
of torus equivariant curves in weighted flag varieties.  This provides a geometric interpretation of the parameters used in \cite{AbeMatsumura2015}.
We approach weighted flag varieties from a uniform Lie-theoretic point of view, providing a more general definition than has appeared previously,
and prove other general results about weighted flag varieties in this setting, including
a Borel presentation of the equivariant cohomology.  In addition, we generalize some results obtained for weighted Grassmannians or more
generally type $A$ (\cite{AbeMatsumura2015}, \cite{AzamNazirQureshi2020}); in particular, we obtain a weighted Chevalley formula, descriptions
of restrictions to fixed points, the GKM description of the cohomology, and
a weighted Chevalley formula.
\end{abstract}

\thispagestyle{empty}
\setcounter{tocdepth}{1}
\tableofcontents

\section{Introduction}\label{section: introduction}
Weighted flag varieties are a generalization of flag varieties and weighted projective spaces, introduced by Grojnowski, Corti and Reid
\cite{CortiReid2002}.
Although they are not generally homogeneous varieties, they admit an action of a torus with isolated fixed points, and as with flag varieties,
their equivariant cohomology rings admit a Schubert basis.  
Abe and Matsumura  
\cite{AbeMatsumura2015} proved that the equivariant cohomology
of weighted Grassmannians
has a positivity property analogous to the positivity for flag varieties proved in \cite{Graham2001}.  This positivity is expressed using certain parameters, and the
authors write that, in contrast to the situation, for non-weighted flag varieties,  ``we do not have the geometric or representation-theoretic interpretation of those
parameters".
The main result of this paper extends the positivity statement
to all weighted flag varieties.  The parameters in the statement 
are weighted roots, which describe the weights
of torus-invariant curves through a torus fixed point in a weighted
flag variety.  This formulation gives a geometric meaning to the
parameters in the positivity statement of \cite{AbeMatsumura2015}: in our
language, their parameters are negative roots at the longest element of
the Weyl group.  

The approach taken in this paper is to study weighted flag varieties from a uniform Lie-theoretic point of view.
Besides
our main theorem, we prove some other results which we hope will be useful
in the future study of weighted flag varieties, including Chevalley-type formulas, and descriptions of restrictions to fixed points,
extending work of \cite{AbeMatsumura2015} for weighted Grassmannians,
and \cite{AzamNazirQureshi2020} for more general type $A$ flag varieties.  We also provide a Borel
presentation of the equivariant cohomology of weighted flag varieties, which appears to be new.
We define the
weighted Schubert classes as Poincar\'e dual classes to closures of weighted
Schubert cells.  This provides a direct geometric interpretation
of the weighted Schubert classes.  Our classes are positive integer multiples of the classes
considered in \cite{AbeMatsumura2015} and \cite{AzamNazirQureshi2020}, and
we determine these integers, which depend on the orders of certain stabilizer
groups.  Our results are valid in the setting of equivariant Chow groups as well as equivariant
cohomology.  

To state our main result we need some notation.  
Let \(G\) be a connected reductive complex algebraic group.
In the definition of weighted flag varieties given
in \cite{CortiReid2002}, $G$
is taken to be of the form $G' \times \bC^\times$, but it
is convenient to work more generally (the relation of the definition
in \cite{CortiReid2002} to the definition in this paper is discussed
in Remark \ref{r:corti-reid}).
Let \(H\) be a Cartan subgroup of $G$, and \(B=HN\) a Borel subgroup containing \(H\)
with unipotent radical $N$.   Let $\fg$, $\fb$, $\fh$ and $\fn$ denote the Lie algebras of $G, B, H$ and $N$, respectively.
We denote by \(\Lambda \cong \Z^{m+1} \) and  \(\Lambda^\vee \cong \Z^{m+1}\)
the character and cocharacter groups of \(H\), respectively; let $\Lambda^+ \subset \Lambda$
denote the set of dominant weights.  Given a nonzero element
$\chi = (n_1, \ldots, n_{m+1}) \in \Lambda^\vee$, we write $\gcd(\chi)$ for the greatest
common divisor of the $n_i$.  Equivalently, $\gcd(\chi)$ is the largest integer $k$
such that $\frac{1}{k} \chi$ is in $\Lambda^\vee$.
Fix a field
\(\F\) of characteristic zero (this will serve as the coefficient field for homology and cohomology).
Let \(\mathbb{H} = \Lambda^\vee \otimes_{\Z} \F \), and let \(\mathbb{H}^\ast\) be the dual \(\F\)-module; if
$\F = \C$, then $\mathbb{H} = \fh$ and $\mathbb{H}^\ast = \fh^\ast$.
We use similar notation for other tori.
We write the natural
pairing between \(\mathbb{H}\) and \(\mathbb{H}^\ast\) as a dot product.  If $\lambda \in \Lambda$,
let $\H_{\lambda} \subset \H$ denote the kernel of $\lambda$, with dual space
$\H_{\lambda}^\ast$.
If $V$ is a (complex) representation of $H$,
let $\Phi(V) \subset \Lambda$ denote the weights of $V$.

Let
$\Phi = \Phi(\fg)$ denote the set of roots of $\fh$ in \(\fg \), with
positive system $\Phi^+ = \Phi(\fn)$, and let
\(\Delta\) denote the corresponding set of simple roots.
Let $W$ denote the Weyl group of $G$, equipped with the Bruhat order.
Fix a dominant weight
\(\lambda\in\Lambda\), and let \(P=LU\) be the parabolic subgroup containing \(B\), such
that the Levi factor $L$ contains $H$, and such that the coroots of the Lie algebra $\fl$ of $L$
are orthogonal to $\lambda$.  The Weyl group of $L$ is denoted $W_P$.
Let $W^P $ denote the set
of maximal length coset representatives of $W/W_P$.
Let $\chi: \bC^\times \to H$ be a cocharacter
of $H$, and let $S$ denote the subgroup $\chi(\bC^\times)$ of $H$.  Let
$\mathbb{T}^\ast$ denote the orthogonal complement of \(\chi\) in \(\mathbb{H}^\ast\).

For each $w \in W^P$, define $a_w:=\gcd(\chi)^{-1}w\lambda\cdot\chi$.
We assume that each $a_w>0$ (cf.\ \cite[Remark 3.1]{CortiReid2002}).  This assumption implies that
$G$ is not semisimple (see Proposition \ref{prop:not semisimple}). 

The character $\lambda$ of $H$ extends to a character of $P$ which is trivial
on $U$; let
$\bC_\lambda$ be the corresponding $1$-dimensional representation.
Let $L_\lambda$ be the kernel in $L$ of $\lambda$; then $Q = L_\lambda U
\subset P$ is the kernel in $P$ of $\lambda$. Let $Z = G/Q$.  Since $P/Q \cong \C^\times$, we can identify
$Z$ with
the mixed space $G \times^P \bC_\lambda^\times$.

We define the weighted flag variety as $X = S\backslash Z$.  The relation with the definition in
\cite{CortiReid2002} 
is explained in Remark \ref{r:corti-reid}.  Our assumptions
on \(\lambda\) and $a_w$ imply that $X$ is a projective algebraic variety with at worst
orbifold singularities, and with
an action of $T = S \backslash H$ (see Section \ref{section: geometry}).  For every $w \in W^P$, there
is a weighted Schubert variety $X_w \subset X$.  The fundamental
classes $[X_w]$ form a basis over $H^*_T$ of the equivariant
Borel-Moore homology $H^T_*(X)$.  There is a corresponding Poincar\'e dual basis
$\{ \delta_{X_w}^T \}$  of $H_T^\ast(X)$ (q.v.\ \eqref{equation: PD}).  The elements of this basis
are integral multiples of the weighted Schubert classes considered in \cite{AbeMatsumura2015} and
\cite{AzamNazirQureshi2020}.  We use these classes because of their
direct geometric connection to $X$.

For each $w \in W^P$, we define a set $\Phi(w)$ of characters
of $T$ which we call weighted roots at $w$.
Each $w \in W^P$ corresponds
to a $T$-fixed point $p_w \in X$, and the $T$-weights of the $T$-invariant
curves in $X$ through $p_w$ form a subset of $\Phi(w)$. 
An explicit formula for the weighted roots at $w$ is given in Proposition \ref{proposition: weighted}.  

Given $u,v, w \in W^P$, let
$\cS(u,v;w) = [w, u] \cap [w,v]$, where the notation denotes Bruhat intervals.
In other words, 
$$
\cS(u,v;w) = \{ x \in W^P \mid w \leq x \leq u \mbox{ and } w \leq x \leq v \}.
$$
Our main result about positivity is the following theorem.

\begin{theorem}\label{theorem: main theorem}
Suppose that  \(\lambda\) is dominant and \(\chi\) is antidominant.
For every \(u,v\in W^P\), write
\begin{equation} \label{e.main}
\delta_{X_u}^T\delta_{X_v}^T=\sum_{w\in W^P}c_{uv}^w\delta_{X_w}^T,
\end{equation}
where \(c_{uv}^w\in H_T^\ast=S(\mathbb{T}^\ast)\).
Then each \(c_{uv}^w\) is a nonnegative linear combination of monomials of the form
$\nu_1(x_1) \cdots \nu_k(x_k)$, where the $\nu_1, \ldots, \nu_k$ are
distinct negative roots, and each $x_i \in \cS(u,v;w)$.
\end{theorem}

The method of proof of Theorem \ref{theorem: main theorem} also yields a positivity result about classes
of subvarieties of weighted flag varieties (Theorem \ref{t:expansion}), which extends the corresponding result in the non-weighted case
(see \cite[Theorem 19.3.1]{AndersonFulton2024}, \cite[Theorem 3.2]{Graham2001}).

Theorem \ref{theorem: main theorem} implies that the coefficients \(c_{uv}^w\) are square-free
in a certain sense.  But because weighted roots at different $x$ may occur in the same monomial, this
is not the same as the usual meaning of square-free, which means with respect to a fixed set of generators.  
Since a weighted root at $x \in W$ can be expressed 
as linear combinations of weighted roots
at any other $y \in W$, it is possible to expand $c_{uv}^w$ as a sum of monomials in negative roots
at a single fixed $y$.  However, such an expansion need not be square-free (see Remark \ref{r:non-squarefree}).

We say that $c_{uv}^w$ is nonnegative at $y \in W^P$ if $c_{uv}^w$ is a nonnegative linear combination
of monomials in negative roots
at $y$.  If $y \geq x$, then by  Lemma \ref{lemma: positive},
a negative simple root at $x$ is a nonnegative linear combination
of negative simple roots at $y$. 
Therefore, as a consequence of Theorem \ref{theorem: main theorem}, we obtain:

\begin{corollary} \label{corollary: w0}
With the hypotheses of Theorem \ref{theorem: main theorem}, each \(c_{uv}^w\)
is nonnegative at any $y$ such that $y \geq x$ for all $x \in \cS(u,v;w)$.  In particular, this holds
if $y \geq u$ and $y \geq v$.
\end{corollary}

In the case of the weighted Grassmannian, this implies
the positivity result of Abe and Matsumura, since their parameters
are positive scalar multiples of the negative simple roots at the maximal element
$w_0$ of $W$.

It is natural to ask whether Theorem \ref{theorem: main theorem} can be strengthened to say that
$c_{uv}^w$ is nonnegative at $w$.
Example \ref{e:nonexpress} shows that this is false for arbitrary antidominant $\chi$, but can hold for
certain antidominant $\chi$. In Example \ref{e:nonexpress},
this is true provided the entries of $\chi$ satisfy certain inequalities.  We refer to this example for further discussion.

The proof of Theorem \ref{theorem: main theorem} relies on a striking product formula
(Lemma \ref{lemma: lambdamultiply})
proved using the 
the non-weighted
Chevalley formula, together with the non-weighted positivity
result of \cite{Graham2001}.   The assumption
that \(\lambda\) is dominant (or antidominant) is used to show that the weighted
flag variety $X$ exists as a separated scheme, and for positivity.
However, many of the results of the paper do not depend on $X$;
they can be formulated and proved in terms
of the space $Z$, without assuming either that $\lambda$
is dominant or $\chi$ is antidominant.  
As observed in \cite{AbeMatsumura2015}, the assumption $\chi$ is antidominant in the positivity statement can be removed by
using a Schubert basis with respect to a Borel subgroup with respect to which $\chi$ is antidominant
(see Remark \ref{r:antidominant}).

\medskip

{\em Acknowledgments:} We would like to thank Arik Wilbert for contributing to the discussions in the early stages of this project,
and Allen Knutson for asking if the weighted structure constants are square-free.
This material is based upon work supported by a grant from the Institute for Advanced Study School of Mathematics.

\section{Equivariant cohomology and Chow groups of partial quotients}\label{section: Background}
We work with schemes over $\bC$ which can be equivariantly embedded in a smooth scheme.  Throughout this paper, we work with (equivariant)
cohomology and Borel-Moore homology groups with coefficients in a field \(\F\) of characteristic zero; the notation $H_*(X)$ and $H_*^G(X)$
denotes Borel-Moore homology groups (rather than singular homology
groups) with coefficients in $\F$.  As discussed in Section \ref{ss:Chow}, our results are also valid in the
setting of equivariant Chow groups (again with coefficients in $\F$).  Background on equivariant cohomology and Borel-Moore
homology can be found in \cite{AndersonFulton2024};
some basic properties of these theories are
also summarized in \cite{Graham2001}.  References
for the non-equivariant versions of these
theories are 
\cite[Appendix B]{Fulton1997} and
the appendix to Chapter 13 of \cite{CoxLittleSchenck2011}.
Basic definitions and results for equivariant Chow groups can be found in \cite{EdidinGraham1998}.
In this section, the notations $G$, $H$, $Z$, etc. are specific to this section; in particular we do not restrict to the case of weighted
flag varieties.
In subsequent sections, the notation has the meaning given in the introduction.

In this section, $Z$ will denote a variety with an action of an algebraic torus $H$ such that the restriction
of this action to a subtorus
$S$ is proper.   Let $T= H/S$, and let
$X = S \backslash Z$.  We refer to $X$ as a partial quotient, since it is a quotient by a subgroup; it has an action
of the quotient group $T$.  Abe and Matsumura showed that there is a pullback isomorphism between
the equivariant cohomology groups $H^*_T(X)$ and $H^*_H(Z)$ (see \cite[Lemma 3.3]{AbeMatsumura2015}).
In this paper we work with classes which are Poincar\'e dual to classes of fundamental classes of subvarieties.
For weighted flag varieties, these are scalar multiples of the classes
studied in \cite{AbeMatsumura2015}.  

Section \ref{ss:Poincare} contains some general results we use to identify these classes.
Proposition \ref{prop:Poincare} relates pullbacks of Poincar\'e dual classes to orders of stabilizer groups;
the first part of this proposition gives an alternative proof of \cite[Lemma 3.3]{AbeMatsumura2015}.  In addition,
we relate these results to equivariant Borel-Moore homology, and show that the Poincar\'e duals of the closures of cells
of an orbifold paving form a basis for equivariant cohomology (Proposition \ref{p:free-coh} and Corollary \ref{corollary: paving}).

Section \ref{ss:Chow} proves versions of these results for equivariant Chow groups.  Proposition \ref{prop:Chowpullback} 
extends \cite[Theorem 3]{EdidinGraham1998} to partial quotients of torus actions.  The proof of this proposition does not rely on 
\cite[Theorem 3]{EdidinGraham1998}, and in fact provides an alternative proof of that theorem in the special case of
torus actions.   Proposition \ref{p:Chowcommute} extends \cite[Theorem 4]{EdidinGraham1998} to partial quotients
in the case of torus actions, and in addition gives a more precise relation between
the pullbacks in equivariant Chow cohomology and homology.  The results of this section imply
that our results about weighted flag varieties hold in the setting of equivariant Chow groups.

\subsection{Equivariant cohomology, partial quotients, and Poincar\'e duality} \label{ss:Poincare}
We recall the definition of the Poincar\'e dual class of
a subvariety.
If $X$ is irreducible and rationally smooth of dimension $n$,
there is a cap product isomorphism
$\cap [X]: H^i(X) \to H_{2n -i}(X)$ (\cite[Prop.~13.A.4]{CoxLittleSchenck2011}); for completeness,
a proof
is given in the appendix to this paper.
If in addition $X$ has an action of a linear algebraic
group $G$, 
then the map
$\cap [X]: H^i_G(X) \to H_{2n -i}^G(X)$ is an isomorphism as well.
Given a closed subvariety $Y \subset X$, the Poincar\'e dual
class to $Y$ is the class $\delta_Y \in H^*(X)$ mapping
to $[Y] \in H_*(X)$ under the cap product isomorphism.  If $Y$ is $G$-invariant, $\delta_Y^G$ is the class
in $H^*_G(X)$ mapping to $[Y] \in H_*^G(X)$. More generally, given any 
$C \in H_*(X)$ (resp.~$H^G_*(X)$), we can define $\delta_C \in H^*(X)$ (resp.~$\delta^G_C$) in the same way.

\begin{lemma} \label{lem.finitepullback}
Suppose that \(\eta\,\colon Z \to Z'\) is a finite 
morphism of rationally smooth irreducible algebraic varieties.
Suppose further that \(\eta_\ast\,\colon H_\ast(Z)\to H_\ast(Z')\) is an isomorphism of vector spaces.
Then \(\eta^\ast\,\colon H^\ast(Z')\to H^\ast(Z)\) satisfies
\begin{equation}
\eta^\ast(\delta_C)=\dfrac{d_Z}{d_C}\delta_Y,
\end{equation}
where \(Y'\subset Z'\) is a subvariety, $C \in H_\ast(Z)$ satisfies $\eta_*C = d_C [Y']$, and $\eta_*[Z] = d_Z [Z']$.

\begin{proof}
We have
\begin{equation}
\eta_\ast(\eta^\ast(\delta_{Y'})\cap[Z]) = \delta_{Y'}\cap\eta_\ast[Z] = 
\delta_{Y'}\cap d_Z[Z'] = d_Z[Y'] = \eta_* (\frac{d_Z}{d_C} C).
\end{equation}
Therefore, since $\eta_*$ is an isomorphism,
\begin{equation}
 \eta^\ast(\delta_{Y'} )\cap [Z]  = \dfrac{d_Z}{d_C} C = \dfrac{d_Z}{d_C}\delta_C \cap [Z],
\end{equation}
so $\eta^\ast (\delta_{Y'}) = \frac{d_Z}{d_C}\delta_C$, as claimed.
\end{proof}
\end{lemma}

\begin{remark} \label{rem.finitepullback}
In the setting of Lemma \ref{lem.finitepullback}, suppose that $C = [Y]$ is the fundamental cycle of the scheme-theoretic inverse image \(Y=\eta^{-1}(Y')\).
Then \(Y\) is equidimensional, by \cite[Theorem 5.11]{AtiyahMacdonald1969}.  Each component of $Y$ maps surjectively onto $Y'$, so \(\eta_\ast[Y]=d_Y[Y']\) for some integer $d_Y$.
In this case, we will write $d_Y$ for $d_C$.  
\end{remark}

\begin{lemma} \label{lem.quotientdegree}
Let $Z$ be a variety with an action of a finite abelian group $F$, and
let $F^Z$ denote the stabilizer
in $F$ of a general point of $Z$.  The degree of the map
$Z \to Z' = F \backslash Z$ is $|F/F^Z|$.
\end{lemma}

\begin{proof}
The group $F^Z$ acts trivially on $Z$, and $Z' = (F/F^Z) \backslash Z$.
Since there is an open subset of
$Z$ where the action of the quotient group $F/F^Z$ is free, the
quotient map has degree $|F/F^Z|$.
\end{proof}

Suppose that $H$ is an algebraic torus with character group $\Lambda$.
The equivariant cohomology of a point is denoted by $H^*_H$, and is isomorphic
to the symmetric algebra $S(\H^*)$, where $\H = \Lambda \otimes \F$.
Suppose that $H$ is a torus and $S$ is a 
subtorus of $H$.  The projection $\phi: H \to T = H/S$ induces a pullback map
$H^*_T \to H^*_H$.

As above, $Z$ denotes a variety with an action of an algebraic torus $H$.  Suppose that the restriction
of this action to a subtorus
$S$ is proper, that is, the map $S \times Z \to Z \times Z$ taking
$(s,z)$ to $(z, sz)$ is proper.  This implies that a quotient $X = S \backslash Z$ exists in general
as a separated algebraic space (\cite{Kollar1997}, \cite{KeelMori1997}); under some additional hypotheses
(for example if $Z$ is quasiprojective
with a linearized $S$-action), $X$ is a scheme.
Since $S$ acts properly, the stabilizer group $S^z$ of any
point in $Z$ is finite.  Given a subvariety $Y$ of $Z$, there is an open subset of
$Y$ such that the stabilizer group $S^y$ is the same for any point in this subset.
Let $S^Y = S^y$ for such a point $y$; in other words, $S^Y$ is the stabilizer
group of a general point in $Y$.  Let $e_Y = |S^Y|$ denote the number of elements in $S^Y$.

Let $\pi: Z \to X$ and $\pi: H \to T$ denote the quotient maps (although we have used the
same notation $\pi$ for both maps, the context will indicate which is meant).
There is a map $\pi^*: H^*_T(X) \to H^*_H(Z)$,
as observed in \cite{AbeMatsumura2015}.  Here we use the definition
of this map given in \cite{AndersonFulton2024}, which uses
the fact that $\pi$ is $\pi$-equivariant
(that is, $\pi(h x) = \pi(h) \pi(x)$).  
We recall briefly
the definition of this map.
Let $E_H$ and $E_T$ be models for the universal spaces for
$H$ and $T$ such that there is a $\pi$-equivariant map $E_H \to E_T$.
We obtain a map
\begin{equation} \label{equation: pimap}
 E_H \times^H Z \longrightarrow E_T \times^T X.   
\end{equation}
The spaces $ E_H \times^H Z $ and $E_T \times^T X$ are examples of ``mixed spaces";
we often denote them by $Z_H$ and $X_T$.
By definition, $H^*_H(Z) = H^*(Z_H)$ and $H^*_T(X) = H^*(X_T)$, and
the map $\pi^*: H^i_T(X) \to H^i_H(Z)$ is by definition
the pullback $H^i(X_T) \to H^i(Z_H)$.  This
map is independent of the choices of $E_H$ and $E_T$ and the
map $E_H \to E_T$ (see \cite[Exercise 2.1]{AndersonFulton2024}).
The map $\pi^*$ is $H^*_T$-linear.  Thus, $H^*_H(Z)$ has an $H^*_T$-module structure obtained
from the $H^*_H$-module structure
via the map $H^*_T \to H^*_H$.

\begin{remark} \label{r:torusproduct}
We can find a subtorus $T_1$ of $H$ such that the composition $T_1 \to H \stackrel{\pi}{\rightarrow} T = H/S$ is an isomorphism.
Then $H \cong S \times T_1$.  By identifying $T_1$ with $T$ we may assume $H = S \times T$.  Then
$E_H = E_S \times E_T$ and the map $E_H \to E_T$ is projection on the second factor.  If $H$ acts on $Z$, the
$T$-action on $Z$ obtained by the identification of $T$ with $T_1$ depends on the choice of $T_1$, but the
$T$-action on $X = S \backslash Z$ does not.
\end{remark}

\begin{remark} \label{r: changegroups}
The $T = H/S$-action on $X = S \backslash Z$ induces an $H$-action via the map $H \to T$.  The map $Z \to X$ is
$H$-equivariant, so we obtain a pullback $H^*_H(X) \to H^*_H(Z)$.  To avoid confusion with the map $\pi^*$ defined above, we denote this
map by $\pi^*_H$.  The maps $\pi^*$ and $\pi^*_H$ are closely related.  The map \eqref{equation: pimap} factors as
\begin{equation} \label{equation: pimap2}
 E_H \times^H Z \longrightarrow  E_H \times^H X \stackrel{p}{\longrightarrow} E_T \times^T X. 
\end{equation}
Therefore, $\pi^* = \pi_H^* \circ p^*$.  We can identify $H$ with $S \times T$, and 
then $H^*_H(X) = H^*_S \otimes H^*_T(X)$.  Under this identification, 
$p^*(c) = 1 \otimes c$, so $\pi^*(c) = \pi^*_H(1 \otimes c)$.
\end{remark}

\begin{lemma} \label{l.quotient}
With notation as above, let $F$ be the subgroup of $S$
generated by the stabilizer groups of all $z \in Z$.  Then $F$ is finite.  Let $S' = S/F$ and $Z' = F \backslash Z$.  
Then $S$ acts on $Z'$ with constant stabilizer group $F$.  The $S'$-action on $Z'$ is free, and 
$$
S' \backslash Z'  \cong S \backslash Z' \cong S \backslash Z = X.
$$
\end{lemma}

\begin{proof}
There are only finitely
many subgroups of $S$ which occur as the stabilizer group of 
a point in $Z$.  Each is a finite group, so the subgroup $F$ they generate is finite.

We claim that $S$ acts on $Z'$ with constant stabilizer group $F$.
Indeed, suppose $z' =F z$ is an element of $Z' = F \backslash Z$.  The inclusion
$F \subset S^{z'}$ is immediate, so we must prove the reverse inclusion.
Suppose $s \in S^{z'}$.
Then $sz = f z$ for some $f \in F$.  Hence $f^{-1} s \in S^z \subset F$,
so $s \in F$.  We conclude that $S^{z'} \subset F$, proving the claim.
This implies that the $S$-action on $Z'$ yields an action of $S/F = S'$ on $Z'$, and
then $S' \backslash Z' \cong S \backslash Z' \cong X$.  The $S'$-action on $Z'$ is 
proper (since the $S$-action on $Z$ is proper) and has trivial stabilizers, so it is free
(\cite[Lemma 8]{EdidinGraham1998}).
\end{proof}

We remark that if $X$ is a scheme, then so is $Z'$; this can be deduced from the fact that the quotient map $\pi: Z \to X$ is affine.

The first part of the following proposition is an analogue of \cite[Lemma 3.3]{AbeMatsumura2015}, with a different proof.
The technique used in the proof of taking the quotient by a finite subgroup $F$ to reduce to a free action of a quotient group  is not new; it appears for example in \cite{DuistermaatHeckman1982}.

\begin{proposition} \label{prop:Poincare}
With notation as above, we have:

(1) $\pi^*: H^i_T(X) \to H^i_H(Z)$ is an isomorphism.

(2) Suppose that $Z$ is rationally smooth and that $Y$ is an
$S$-invariant subvariety of $Z$.  Let $\overline{Y} = S \backslash Y
\subset X$.  Then 
\begin{equation}\label{equation: pullback}
\pi^\ast(\delta_{\overline Y}^{ T})=\frac{e_Y}{e_Z}\delta_Y^H\in H^\ast_H(Z).
\end{equation}
\end{proposition}

\begin{proof}
Let $H' = H/F$, and as in the previous lemma, let $Z' = F \backslash Z$, $S' = S/F$.  Factor the map \eqref{equation: pimap} as
\begin{equation} \label{e: pimap-factor}
E_H \times^H Z \stackrel{f}{\longrightarrow} E_H \times^H Z'  \stackrel{g}{\longrightarrow} E_{H'} \times^{H'} Z'   \stackrel{h}{\longrightarrow} E_T \times^T X.   
\end{equation}
These maps induce pullbacks
\begin{equation} \label{e: pimap-pullbacks}
H^*_T(X) \stackrel{h^*}{\longrightarrow} H^*_{H'}(Z') \stackrel{g^*}{\longrightarrow} H^*_H(Z')  \stackrel{f^*}{\longrightarrow} H^*_H(Z).
\end{equation}
We will prove that each of $f^*$, $g^*$, and $h^*$ are isomorphisms, so their composition
$\pi^* = f^*  \circ g^* \circ h^* $ is as well.

{\em Step 1.}  The map $f$ is the quotient map
$$
E_H \times^H Z \longrightarrow E_H \times^H (F \backslash Z) \cong F \backslash (E_H \times^H Z),
$$
where the $F$-action on $E_H \times^H Z$ is defined as the action on the second factor $Z$.
Therefore, the pullback map gives an isomorphism $f^*: H^*_{H'}(Z') \to H^*_H(Z)^F$.
The induced action of $F$ on $H^*(E_H \times^H Z) = H^*_H(Z)$
is trivial.  Indeed, the $F$-action is the
restriction of an action of the connected group $S$ and the action
of a connected group induces a trivial action on cohomology
(as the map induced by any element of the group is homotopic to the identity map). 
We conclude that $f^*: H^*_{H'}(Z') \to H^*_H(Z)$ is an isomorphism.

{\em Step 2.}  We can factor the map $g$ as follows.  Let $E$ be a model
for $E_H$, and $E'$ a model for $E_{H'}$,  Then $H$ acts on $E'$ via the map
$H \to H'$, and $E \times E'$, with the product action of $H$, is also a model for $E_H$.
Then $g$ can be written as the composition
\begin{equation} \label{e: g-factor}
(E \times E') \times^H Z' \stackrel{p}{\longrightarrow}  (E/F \times E') \times^H Z' = (E/F \times E') \times^{H'} Z' \stackrel{q }{\longrightarrow} E' \times^{H'} Z',
\end{equation}
where the middle equality is because the subgroup $F$ acts trivially on $(E/F \times E')$ and on $Z'$.  The map $q$ is a fiber bundle
with fibers $E/F$.  Since the cohomology of $E/F$ vanishes in positive degrees, $q^*$ is an isomorphism in cohomology.  The map $p$ is a quotient
map by $F$, and arguing as in Step 1 shows that $p^*$ is an isomorphism.  Hence $g^*: H^*_{H'}(Z') \to H^*_H(Z)$ is an isomorphism.

{\em Step 3.}  As in Remark \ref{r:torusproduct}, we may assume $H' = S' \times T$ and $E_{H'} = E_{S'} \times E_T$.  Then
$$
E_{H'} \times^{H'} Z' = (E_{S'} \times E_T) \times^{S' \times T} Z
\cong E_{S'} \times^{S'} (E_T \times^T Z').
$$
Also,
$$
E_T \times^T X \cong E_T \times^T (Z'/S') = (E_T \times^T Z')/S'.
$$
Therefore, $h$ can be identified with the map
$$
E_{S'} \times^{S'} (E_T \times^T Z') \longrightarrow (E_T \times^T Z')/S'.
$$
So $h^*$ can be identified with the pullback $h^*: H^*((E_T \times^T Z')/S') \to H^*_{S'}(E_T \times^T Z')$.
Since $S'$ acts freely on $Z'$, it acts freely on $E_T \times^T Z'$, and therefore the pullback is an isomorphism.

Combining the results of Steps 1-3, we see that $\pi^* = f^*  \circ g^* \circ h^* $ is an isomorphism.  This proves (1).

We now prove (2).  Let $Y' = F\backslash Y$, so $\overline{Y} = S \backslash Y = S' \backslash Y'$.  We have
$\pi^*(\delta_{\overline Y}^{ T}) = f^*  \circ g^* \circ h^* (\delta_{\overline Y}^{ T}) $.  
Since $S'$ acts freely on $E_T \times^T Z'$, the map $h$ is an $S'$-principal bundle map, so it is smooth.
Under a smooth
map, the pullback of the Poincar\'e dual of a variety is the
Poincar\'e dual of the inverse image.  Hence,
$$
h^* \gd_{E_T \times^T {\overline Y}} = \gd_{E_{S'} \times^{S'} (E_T \times^T Y') } = \gd_{E_{H'} \times^{H'} Y'}.
$$
In the language of equivariant cohomology, this states that $h^* \gd^T_{\overline{Y}} = \gd^{H'}_{Y'}$.

Next, using the factorization \eqref{e: g-factor} of the map $g$, we see that
$$
g^* \gd_{E_{H'} \times^{H'} Y'} = p^* q^* \gd_{E' \times^{H'} Y'} = p^* \gd_{(E/F \times E' ) \times^H Y'} = \gd_{(E \times E') \times^H Y'}.
$$
Here, the second equality is because $q$ is a smooth map.  The third equality follows from Lemma \ref{lem.finitepullback}, because $p$ is a quotient map where the group
$F$ acts with trivial stabilizers.  In equivariant language, we see that $g^*  \gd^{H'}_{Y'} =  \gd^H_{Y'}$.

Finally, the degree of the map $f: E_H \times^H Z \to E_H \times^H (F \backslash Z)$ is the same as the degree of the map $Z \to F \backslash Z$, which
by Lemma \ref{lem.quotientdegree} is $|F/F^Z|$.  Similarly, the degree of $E_H \times^H Y \to E_H \times^H (F \backslash Y)$ is $|F/F^Y|$.  Hence, by Lemma
\ref{lem.finitepullback},
$$
f^* \gd_{E_H \times^H Y'} = \frac{|F|/|F^Z|}{|F|/|F^Y|} \delta_{E_H \times^H Y}
= \frac{|F^Y|}{|F^Z|} \delta_{E_H \times^H Y} = \frac{e_Y}{e_Z} \delta_{E_H \times^H Y},
$$
where the last equality is because $|F^Y| = |S^Y| = e_Y$ and $|F^Z| = |S^Z| = e_Z$.  In equivariant language, this says that
$f^*  \gd^H_{Y'} =   \frac{e_Y}{e_Z} \gd^H_Y$.

Combining what we have proved, we see that 
$$
\pi^* (\delta_{\overline Y}^{ T})= f^*g^*h^* (\delta_{\overline Y}^{ T}) =  f^*g^*( \gd^{H'}_{Y'} )=  f^*(\gd^H_{Y'} )= \frac{e_Y}{e_Z} \gd^H_Y,
$$
proving (2).
\end{proof}

With notation as above, assume that $Z$ (and therefore $X$) are rationally
smooth.  Then there are isomorphisms $\cap [X]_T: H^*_T(X) \to H^T_*(X)$
and $\cap [Z]_H: H^*_H(Z) \to H^H_*(Z)$. Define the homology
pullback $\pi^*: H^T_*(X) \to H^H_*(Z)$ so that the following diagram commutes:
\begin{equation} \label{e:homology-commute}
\begin{CD}
H^*_H(Z) @>{\cap [Z]_H}>> H_*^H(Z) \\
@A{\pi^*}AA     @AA{\pi^*}A \\
H^*_T(X) @>{\cap \frac{1}{|S_Z|} [X]_T}>> H_*^T(X).
\end{CD}
\end{equation}
The factor of $\frac{1}{|S_Z|}$ appears for compatibility with equivariant Chow groups (see Section \ref{ss:Chow}).
It means that
\begin{equation} \label{e.BMpullback}
(\pi^* \alpha) \cap [Z]_H =  \frac{1}{e_Z} \pi^*(\alpha \cap [X]_T) .
\end{equation}
Taking $\ga = \gd^T_{\overline{Y}}$, the left side equals $\frac{e_Y}{e_Z} [Y]_H$, while the right side is $\frac{1}{e_Z} \pi^* [\overline{Y}]_T$.  Hence 
\begin{equation} \label{e.BMpullback2}
\pi^*[\overline{Y}]_T = e_Y [Y]_H,
\end{equation}
which is consistent with the pullback \eqref{e.Chowpullback} for equivariant Chow groups.

A variety $X$ is said to be paved by subvarieties $U_{pq}$ if $X$
has a filtration $X_0 \subset X_1 \subset \cdots \subset X_r = X$ by
closed subvarieties $X_p$ such that $X_p \setminus X_{p-1}$ is a disjoint
union of the subvarieties $U_{pq}$.  We call this an orbifold paving
if each $U_{pq}$ is a quotient of the affine
space $\mathbb{A}^{k_{pq}}$ by a finite group $F_{pq}$.  If in addition $T$ acts on
$X$ preserving the paving, and the $T$-action on $U_{pq}$
is induced by a linear action of $T$ on $\mathbb{A}^{k_{pq}}$ which commutes with the $F_{pq}$ action,
we call this a $T$-invariant orbifold paving.  The $U_{pq}$ are called cells of the paving.

\begin{remark} \label{remark:orbifold}
In our applications to weighted flag varieties, the sets $U_{pq}$ will arise as $S \backslash (C \times \C^*_{\mu})$.  Here $C$ is isomorphic to affine space
with a linear $H$-action and $\mu: H \to \C^*$ is a character of $H$ which is nontrivial on $S$.  Then $H$ acts on $C \times \C^*_{\mu}$, so $T$ acts on the quotient $S \backslash (C \times \C^*_{\mu})$.
Observe that $\C^*_{\mu} \cong H/H_1 = S/F$, where $H_1$ is the kernel of $\mu$, and $F = S \cap H_1$ is finite.
Therefore 
$$
S \backslash (C \times \C^*_{\mu}) \cong S \backslash (C \times S/F) \cong F \backslash C,
$$
where the second isomorphism is given by the map $S(c, sF) \mapsto F s^{-1} c$.  This exhibits $S \backslash (C \times \C^*_{\mu})$ as a
quotient of affine space by a finite group.
\end{remark}

If \(X\) is a \(G\)-variety such that \(H_*^G(X)\) is the free $H^*_G$-module generated by $[X]_G$,  then we call \(X\) a \emph{\(G\)-acyclic} variety.  

\begin{proposition}\label{p:free-coh}
Suppose \(X\) is a \(T\)-variety admitting a $T$-invariant paving by \(T\)-acyclic subvarieties \(U_{pq}\).  Let $V_{pq}$ denote the closure of $U_{pq}$ in $X$.
Then \(H_*^T(X)\) is a free \(H_T^*\)-module with basis given by the fundamental classes \([V_{pq}]\).  
If \(X\) is rationally smooth, then \(H^*_T(X)\) is a free \(H_T^*\)-module with basis given by \(\delta_{V_{pq}}^T\).
\end{proposition}

\begin{proof}
The first statement follows from the long exact sequence in equivariant Borel-Moore homology (cf.~\cite[Prop.~2.1(a)]{Graham2001}).
If in addition \(X\) is rationally smooth, applying equivariant Poincar\'e duality to the homology result yields the cohomology result.
\end{proof}

The next proposition proves acyclicity of the varieties discussed in Remark \ref{remark:orbifold}.

\begin{proposition}\label{p:acyclic}
Let \(S\subset H\) be a subtorus acting transitively and properly on \(H/H_1\), where \(H_1\subset H\) is a closed subgroup.
Let \(C\) be a smooth \(H_1\)-acyclic \(H\)-variety.
Then \(S\backslash(C\times H/H_1)\) is \(T\)-acyclic, where \(T=H/S\).
\end{proposition}

\begin{proof}
By Poincar\'e duality, it suffices to show that $H^*_T(S\backslash(C\times H/H_1)$ is a free $H^*_T$-module generated by
$\gd^T_M$, where $M = S\backslash(C\times H/H_1)$.

Since $S$ acts transitively on $H/H_1$, the map
$S \times H_1 \to H$ is surjective.  It is also finite, since the properness of the action implies that 
the intersection $F = S \cap H_1$ is finite.  Therefore, the composition
$H_1 \hookrightarrow H \rightarrow T = H/S$ is a finite surjective map.  This composition induces maps
$$
H^*_T \stackrel{\pi^*}{\longrightarrow} H^*_H \longrightarrow H^*_{H_1};
$$
the composition is an isomorphism. 

Since $C$ is smooth, via Poincar\'e duality, we can identify $H^*_{H_1}(C)$ with $H_*^{H_1}(C)$.  Since
$C$ is $H_1$-acyclic, $H^*_{H_1}(C)$ is the free $H^*_{H_1}$-module generated by $\gd^{H_1}_C$.
Since $H/H_1 \cong S/F$, we see that $M = S\backslash(C\times H/H_1) \cong F \backslash C$ is rationally 
smooth, so we can identify $H^*_T(M)$ with $H^T_*(M)$.
It suffices to show that $H^*_T(M)$ is the free $H^*_T$-module generated
by $\gd^T_{M}$.

We have
\begin{equation} \label{e:acyclic}
H^*_T(M) = H^*_T(S\backslash(C\times H/H_1)) \cong H^*_H(C\times H/H_1) \cong H^*_{H_1}(C),
\end{equation}
where the first isomorphism is by Proposition \ref{prop:Poincare}, and the second is a change of groups isomorphism.
The composition of these isomorphisms takes $\gd^T_M$ to $\gd^{H_1}_C$, and moreover, is
compatible with the $H^*_T$-module structures (the $H^*_T$-module structure
on $H^*_{H_1}(C)$ is obtained from the 
isomorphism $H^*_T \to H^*_{H_1}$).  
The result follows.
\end{proof}

The following is a direct consequence of Propositions \ref{p:free-coh} and \ref{p:acyclic}.

\begin{corollary} \label{corollary: paving}
Suppose \(Z\) is an \(H\)-variety admitting a paving by subvarieties of the form \(C_i\times H/H_i\), where \(C_i\) is a smooth \(H_i\)-acyclic \(H\)-variety.  Suppose
$S$ is a subtorus of $H$ acting properly on \(Z\) and transitively on each \(H/H_i\).
Let \(X=S\backslash Z\) and \(T=H/S\).
Then $H^T_*(X)$ is a free \(H^*_T\)-module with basis given by the fundamental classes of \(X_i = \overline{S\backslash(C_i\times H/H_i)}\).
 If in addition $X$ is rationally smooth, then $H^*_T(X)$ is free with basis given by the classes $\gd^T_{X_i}$.
\end{corollary}

We will see below that weighted flag varieties admit pavings of the form described in
Corollary \ref{corollary: paving}; as discussed in Remark \ref{remark:orbifold}, these are orbifold affine pavings.

\subsection{Equivariant Chow groups} \label{ss:Chow}
In this section, we work with Chow groups with coefficients in $\F$, and
write $A_*(X)$ to mean the Chow groups of $X$ tensored with $\F$.
If $G$ is an algebraic
group acting on $M$, the equivariant Chow groups
are defined by $A_i^G(M) = A_{i + r}(E_G \times^G M)$.  Here,
$E_G$ denotes a finite-dimensional model of the universal
space for $G$, chosen depending on the degree $i$, and $r = \dim E_G - \dim G$.
The space $E_G \times^G M$ is referred to as a mixed space and the action of $G$
on $E_G \times M$ as the mixing action.
Similarly, there are operational equivariant Chow groups $A^i_G(M)$, defined
in analogy to the operational Chow groups of \cite{Fulton1997}. 
There
is a cycle map $\Cl:A_i^G(M) \to H_{2i}^G(M)$.  We refer
to \cite{EdidinGraham1998} for details.  To simplify the exposition,
we will simply write $M_G = E_G \times^G M$ and $A^G_*(M) = A_*(M_G)$, although 
to be
precise, we would need to choose a particular model of $E_G$, depending
on the degree of the Chow groups.  We define $B_G = E_G/G$, and 
then $A^*_G = A^*(B_G)$.  Since $B_G$ is nonsingular, there is a cap product
isomorphism $\cap [B_G]: A^*(B_G) \to A_*(B_G)$.  Since we are taking coefficients
in a field $\F$ of characteristic $0$, there is a natural isomorphism $A^*_G \cong H^*_G$.
The Chow groups $A^G_*(M)$ are $A^*_G$-modules.

If $\alpha$ is a character of $G$, then $\alpha$ defines a line bundle $\cL^G_{\alpha} = E_G \times^G \C_{\alpha}$.
This line bundle pulls back to a line bundle $\cL_{\alpha,M}$ on the mixed space $M_G$, and the
action of $\alpha$ on $A^G_*(M)$ is by cap product with the Chern class $c_1(\cL^G_{\alpha,M})$.

\begin{lemma} \label{lemma: Chowaction}
Suppose $\beta \in A^*_G = A^*(B_G)$ satisfies $\beta \cap [B_G] = \sum a_i [Q_i/G]$,
where $a_i \in \F$, and $Q_i$ are closed $G$-invariant subvarieties of $E_G$.
Let $M$ be a $G$-scheme and $N \subset M$ a closed $G$-invariant subscheme.
Then 
$$
\beta \cap [N]_G = \sum a_i [Q_i \times^G N] \in A_*(E_G \times^G M)
= A^G_*(M).
$$
\end{lemma}

\begin{proof}
By the functorial properties of Chow groups, we may assume $N = M$.
The projection $p: E_G \times^G M \to B_G$ 
is flat.  By definition,
\begin{align*}
\beta \cap [M]_G & = (p^* \beta) \cap [E_G \times^G M] 
= (p^* \beta) \cap p^*(B_G) \\
& = p^*(\beta \cap [B_G]) =  p^*( \sum a_i  [Q_i/G]) = \sum a_i [Q_i \times^G M];
\end{align*}
here we have used the compatibility between flat pullback and pullbacks of operational Chow groups.
\end{proof}

As in the previous subsection, $H$ denotes a torus acting on a variety $Z$, $S$ is a 
subtorus of $H$ such
that the action of $S$ on $Z$ is proper, and $X = S \backslash Z$.  The map $H \to T \cong S \backslash H$ yields
an injective map $A^*_T \to A^*_H$.  Using this injection, we view $A^*_T$ as a subset
of $A^*_H$, so we obtain an $A^*_T$-module structure on $H$-equivariant Chow groups.  

The next proposition
is an equivariant Chow homology analogue of Proposition \ref{prop:Poincare}.  If $S = H$ is trivial, it reduces
to the case of \cite[Theorem 3]{EdidinGraham1998} where the group $S$ is a torus.  The proposition
could be proved by 
applying \cite[Theorem 3]{EdidinGraham1998} to appropriate mixed spaces.
However, we instead give an argument along the lines of the proof of Proposition \ref{prop:Poincare}.   

\begin{proposition} \label{prop:Chowpullback}
There is an $A^*_T$-linear isomorphism
$\pi^*:A^T_*(X) \to A^H_*(Z)$ such that if $Y$ is an $H$-invariant subvariety of $Z$, and 
$\overline{Y} = S \backslash Y$, then
\begin{equation} \label{e.Chowpullback}
\pi^*[\overline{Y}]_T = e_Y [Y]_H.
\end{equation}
If $Z$ is rationally smooth, then this pullback is compatible under the cycle
map with the pullback $\pi^*: H^T_*(X) \to H^H_*(Z)$ defined in \eqref{e.BMpullback}.
\end{proposition}

\begin{proof}
We adopt the notation of Proposition \ref{prop:Poincare}, and 
again use the following sequence of maps of mixed spaces:
\begin{equation} \label{e: pimap-factor2}
E_H \times^H Z \stackrel{f}{\longrightarrow} E_H \times^H Z'  \stackrel{g}{\longrightarrow} E_{H'} \times^{H'} Z'   \stackrel{h}{\longrightarrow} E_T \times^T X.   
\end{equation}
We will define maps $f^*$, $g^*$, and $h^*$, fitting into a sequence 
\begin{equation} \label{e: pimap-pullbacks-Chow}
A^T_*(X) \stackrel{h^*}{\longrightarrow} A^{H'}_*(Z') \stackrel{g^*}{\longrightarrow} A^H_*(Z')  \stackrel{f^*}{\longrightarrow} A^H_*(Z),
\end{equation}
and define $\pi^* = h^* \circ g^* \circ f^*$.
In contrast to cohomology, 
arbitrary maps of schemes do
not induce pullback maps on Chow groups, so we need to define the pullback maps.

{\em Step 1.}  Since $f$ is a quotient map by a finite group, \cite[Ex.~1.7.6]{Fulton1997} provides an isomorphism
$$
f^*: A^*(E_H \times^H Z') \to A^*(E_H \times^H Z)^F = A^*(E_H \times^H Z).
$$
The last equality is because the action of $F$ on $A^*(E_H \times^H Z)$ is trivial.  The reason is that
the action of $F$ on $E_H \times^H Z$ is the restriction of an action of the torus $S$, and the action of any $s \in S$ on 
the Chow groups is trivial, since $s$ can be connected by a sequence of rational curves to the identity element of $S$.  

{\em Step 2.}  Since we are working with coefficients in $\F$, the map $H \to H' = H/F$ induces a pullback
identification of $A^*_{H'}$ with $A^*_H$.  Brion (\cite[Theorem 2.1]{Brion1997} gave a presentation of $A^*_H(Z')$ as
an $A^*_H$-module: it is generated by the fundamental classes $[Y]_H$, where $Y$ is an $H$-invariant subvariety of $Z'$,
subject to the relations $\mbox{div}_Y(r) - \gamma \cdot [Y]_{H'}$, where $r$ is a rational function on $Y$ which is an $H$-eigenvector
of weight $\gamma$.  Note that $\gamma$ is an element of $A^*_H = A^*_{H'}$.  The presentation of $A^*_{H'}(Z')$ is almost the same, with $[Y]_{H'}$ in place of $[Y]_H$.
Since the action of $H$ on $Z'$ is induced from the $H'$-action by the projection $H \to H'$, the invariant subvarieties and eigenvectors
for $H$ and $H'$ are the same.  We define $g^*: A^{H'}_*(Z') \to A^H_*(Z')$ by $g^*(\sum \gamma_i [Y_i]_{H'}) = \sum \gamma_i [Y_i]_{H})$.
It follows from Brion's result that $g^*$ is an isomorphism.

{\em Step 3.}  As noted in the proof of Proposition \ref{prop:Poincare}, the map $h$ can be identified with the map
$$
E_{S'} \times^{S'} (E_T \times^T Z') \longrightarrow (E_T \times^T Z')/S'.
$$
This map is an $S'$-principal bundle map, 
so flat pullback induces an isomorphism $h^*: A_*(E_T \times^T Z')/S) \to A_*^S(E_T \times^T Z')$ (see \cite[Prop.8]{EdidinGraham1998}).
This map takes $[(E_T \times^T Y)/S] = [E_T \times^T \overline{Y}] $ to $[E_T \times^T Y]_S$.  In other words, we obtain an isomorphism
 $h^*: A^T_*(X) \stackrel{h^*}{\longrightarrow} A^{H'}_*(Z')$ satisfying $h^*([\overline{Y}]_T) = [Y]_{H'}$.

The equation \eqref{e.Chowpullback} holds because
\begin{equation} \label{e.Chowpullback-proof}
\pi^* [Y]_T = f^* g^* h^* ([\overline{Y}]_T)  = f^* g^*  ([Y]_{H'}) = f^* ( [Y]_H) =  e_Y [Y]_H.
\end{equation}
Here, the first equality is by definition, the second was proved in the preceding paragraph, and the third follows
from the definition of $g^*$.  It remains to prove the last equality.  By assumption, $Y$ is an $S$-invariant subvariety of
$Z$, and $Y' = F \backslash Y \subset Z'$.  Under the map $Z \to Z'$, the inverse image of $Y'$ is $Y$.  The reason is that
$F$ transitively permutes the irreducible components of the inverse image, but $Y$ is stable under $S$, so it is stable under $F$;
thus $Y$ is the entire inverse image.  The subgroup of $F$ acting trivially on $E_H \times Y$ is the subgroup $F^Y$, and similarly with
$Z$ in place of $Y$.  Therefore, applying the definition of $f^*$ from \cite[Ex.~1.7.6]{Fulton1997}, we have
$$
f^* [E_H \times^H Y'] = | F^Y |  [E_H \times^H Y'].
$$
As noted in the proof of Proposition \ref{prop:Poincare}, $ | F^Y |=  | S^Y | = e_Y$.  Using this equality, and translating the displayed
equation into equivariant language, we see that $f^* [Y'_H] = e_Y [Y]_H$.  This proves the last equality of \eqref{e.Chowpullback-proof}.

We now check that $\pi^*$ is $A^*_T$-linear.  As above, we view $A^*_T$ as a subset of $A^*_H \cong A^*_{H'} $. 
We will show that the map $h^*$ is $A^*_T$-linear, and that $g^*$ and $f^*$ are $A^*_H$-linear, hence (as $A^*_T \subseteq A^*_H$)
$A^*_T$-linear.  This suffices, since $\pi^* = h^* \circ g^* \circ f^*$.

It suffices to check that each of the maps $f^*$, $g^*$, and $h^*$ is $A^*_T$-linear.  
The map $h^*$ is $A^*_T$-linear because $h^*$ is defined as flat pullback, which is compatible with the action of
Chern classes on Chow groups, and $h^* \cL^T_{\alpha,X} = \cL^{H'}_{\alpha, Z'}$.  Thus,
$h^*$ is compatible with the action of a character $\ga$ of $T$; since these characters generate
$A^*_T$, we see that $h^*$ is compatible with the $A^*_T$-action.  The $A^*_H$-linearity of $g^*$ is immediate
from the definition of this map.  Finally, we check $A^*_H$-linearity of $f$.  Since $A^*_H(Z')$ is generated as an $A^*_H$-module by the fundamental
classes of $H$-invariant subvarieties of $Z'$,
Suppose that $\beta \in A^*_H$ and 
$W'$ is a closed $H$-invariant subvariety of $Z'$.  Arguing as above shows that the inverse image of $W'$ in $Z$ is an irreducible
variety $W$.  It suffices to show that 
\begin{equation} \label{e:f-pullback}
f^*(\beta \cap [W']_H) = \beta \cap f^*([W']_H) = e_W \beta \cap [W]_H.
\end{equation}
By Lemma \ref{lemma: Chowaction}, if $\beta \cap [B_H]$ satisfies $\sum a_i [Q_i/H]$, then
$\beta \cap  [W']_H = \sum a_i [Q_i \times^H W']$.  Let $e$ be the order of the subgroup of $F$ fixing a general
point of $Q_i \times^H W$.  By definition, $f^*[Q_i \times^H W'] = e [Q_i \times^H W] $.
Since the action of $F$ on $Q_i \times^H W$ is on the second factor, $e = |F^W| =e_W$.  Therefore
$$
f^*(\beta \cap [W']_H) = \sum a_i f^*( [Q_i \times^H W']) = e_W \sum a_i [Q_i \times^H W] = e_W \beta \cap [W]_H,
$$
where the last equality is by Lemma  \ref{lemma: Chowaction}.  This proves \eqref{e:f-pullback}.  We conclude that $\pi^*$ is $A^*_T$-linear.

 Finally, $\pi^*$ is compatible with the cycle map
because equations \eqref{e.BMpullback} and \eqref{e.Chowpullback} are the same,
except that the first equation is interpreted as an equality in equivariant Borel-Moore
homology and the second as an equality in equivariant Chow groups.
\end{proof}

\begin{remark}
The map $\pi^*$ could be defined by starting with the formula \eqref{e.Chowpullback} and extending
by $A^*_T$-linearity.  Brion's presentation of torus-equivariant Chow groups implies that this is a well-defined 
map of equivariant Chow groups.  However, it would remain to show that $\pi^*$ is an isomorphism, which is why
we have proceeded differently.
\end{remark}

In the setting of cohomology and Borel-Moore homology, we first deduced the isomorphism on cohomology and constructed
the map on Borel-Moore homology from that (under the assumption that $Z$ is smooth).
In the Chow setting, we reverse the order: having shown the isomorphism on the Chow homology, we now show that there is a compatible
isomorphism of equivariant operational Chow groups, which correspond to equivariant cohomology.  

The projection $\pi: Z_H \to X_T$ induces a map
$\pi^*: A^*(X_T) \to A^*(Z_H)$.  Since $A^*(X_T) \cong A^*_T(X)$ and $A^*(Z_H) \cong A^*_H(Z)$ (see \cite[Theorem 2]{EdidinGraham1998}),
we obtain a map $\pi^*: A^*_T(X) \to A^*_H(Z)$.

The following proposition is a Chow version of \eqref{e:homology-commute}.

\begin{proposition} \label{p:Chowcommute}
The following diagram commutes.
\begin{equation} \label{e:Chowcommute}
\begin{CD}
A^*_H(Z) @>{\cap [Z]_H}>> A^H_*(Z) \\
@A{\pi^*}AA  @AA{\pi^*}A \\
A^*_T(X) @>{\cap \frac{1}{e_Z} [X]_T}>> A^T_*(X).
\end{CD}
\end{equation}
The vertical maps are isomorphisms, and if $Z$ is smooth, the horizontal maps are isomorphisms as well.
\end{proposition}

\begin{proof}
We deduce this result by applying the results of \cite{EdidinGraham1998} to appropriate mixed spaces.
We may assume $H = S \times T$,
so $E_H = E_S \times E_T$.  Then we can identify $Z_H$ with $(Z_T)_S$, and $X_T$ with $(Z_T)/S$.  With these identifications,
the diagram \eqref{e:Chowcommute} becomes
\begin{equation} \label{e:Chowcommute2}
\begin{CD}
A^*_S(Z_T) @>{\cap [Z_T]_S}>> A^S_*(Z_T) \\
@A{\pi^*}AA  @AA{\pi^*}A \\
A^*((Z_T)/S) @>{\cap \frac{1}{e_Z} [(Z_T)/S]}>> A_*((Z_T)/S).
\end{CD}
\end{equation}
By \cite{EdidinGraham1998}, the vertical maps are isomorphisms, and if $Z$ is smooth, the horizontal maps are isomorphisms as
well.  We now verify that the diagram commutes.  If $c \in A^*((Z_T)/S)$, the class $\pi^*c \cap [Z_T]_S$ is represented by the
class $\pi^* c \cap [(Z_T)_S]$.  We need to show that 
\begin{equation} \label{e:Chowcommute3}
\frac{1}{e_Z}  \pi^*(c \cap [(Z_T)/S]) = \pi^* c \cap [(Z_T)_S].
\end{equation}
We have
$$
\pi^* (c \cap [(Z_T)/S]) =   \pi^* c \cap \pi^*[(Z_T)/S] = e_{Z_T}|  \pi^* c \cap [Z_T]_S =  e_Z  \pi^* c \cap [Z_T]_S,
$$
Here, the first equality holds by \cite[Lemma 6]{EdidinGraham1998}, and the second holds by the definition of $\pi^*$.
We conclude that \eqref{e:Chowcommute3} holds, as desired.
\end{proof}

If $Z$ is smooth, the cycle
map may be interpreted as a map $\cl: A^*_H(Z) \to H^*_H(Z)$ by using the Poincar\'e duality identifications of cohomology with homology.
With this interpretation, \cite[Corollary 19.2]{Fulton1998},
applied to the mixed space $Z_H$, implies that $\cl$ is a ring homomorphism.  The pullback homomorphisms
$\pi^*: H^*_T(X) \to H^*_H(Z)$ and $\pi^*: A^*_T(X) \to A^*_H(Z)$ are  pullbacks on appropriate mixed spaces,
so standard properties of cohomology imply that they are ring homomorphisms as well.  

For weighted flag varieties,
the next proposition implies that the cycle map is an isomorphism.  This proposition
is a Chow homology analogue of Proposition \ref{p:free-coh}.  We use the notation of that proposition.

\begin{proposition} \label{p:free-Chow}
Assume that the hypotheses of Proposition  \ref{p:free-coh} hold, and assume in addition that $A_*^T(U_{pq})$ is the free $A^*_T$-module
generated by $[U_{pq}]_T$.  Then \(A_*^T(X)\) is a free \(A_T^*\)-module with basis given by the fundamental classes \([V_{pq}]\),
and the cycle map \(\cl\colon A^T_*(X)\to H_*^T(X)\) is an isomorphism.  This holds if the $U_{pq}$ are the cells of a $T$-invariant orbifold paving.
\end{proposition}

\begin{proof}
The first statement follows from the arguments of \cite[Example 19.1.11]{Fulton1998}.  Now suppose
that the $U_{pq}$ are the cells of a $T$-invariant orbifold paving.   Using mixed spaces and
\cite[Ex.~1.7.6]{Fulton1997}, one can see $A_*^T(U_{pq}) \cong A^T_*({\mathbb A}^{k_{pq}})^{F_{pq}}$.  The equivariant
Chow groups $A^T_*({\mathbb A}^{k_{pq}})$ are a free $A^*_T$-module generated by $[{\mathbb A}^{k_{pq}}]_T$.
The group $F_{pq}$ acts trivially on 
$A^T_*({\mathbb A}^{k_{pq}})$, since the action of $F_{pq}$ commutes with the action of $T$ and preserves the fundamental class $[{\mathbb A}^{k_{pq}}]_T$.
Hence $ A^T_*({\mathbb A}^{k_{pq}})^{F_{pq}} = A^T_*({\mathbb A}^{k_{pq}})$.
\end{proof}

In light of the results of this section, our results about weighted flag varieties are valid in the setting
of equivariant Chow groups.

\subsection{Appendix: Poincar\'e duality}\label{subsection: poincare duality}
With respect to rational cohomology, rationally smooth varieties behave much like smooth varieties, and in particular
they satisfy Poincar\'e duality (see for example, \cite[Section 12.4]{CoxLittleSchenck2011}).
For reference, we provide a brief justification of Poincar\'e duality in this setting.
We recall facts about sheaves from \cite{GoreskyMacPherson1983} in the following discussion.
Given a complex of sheaves \(\cF\), let \(\bD(\cF)\) denote the Borel-Moore dual complex of \(\cF\).
For brevity, for any algebraic variety \(Q\), we set \(\bD_Q=\bD(\bF_Q)\) the dual complex of the constant sheaf.
Hypercohomology gives \(\bH^k(\bF_Q)=H^k(Q)\) singular cohomology and \(\bH^{-k}(\bD_Q)= H_k(Q)\) Borel-Moore homology.
A \(\bF\)-orientation for \(Q\) is a chosen quasi-isomorphism 
\begin{equation}\label{equation: orientation}
\bD_U\to {\bF}_U[2n] 
\end{equation}
where \(U\subset Q\) is the smooth locus, and \(n\) is the complex dimension of \(Q\).
By \cite[\S5.1]{GoreskyMacPherson1983} there are unique morphisms 
\begin{equation}\label{equation: factor}
\bF_Q[2n]\to IC_Q\to \bD_Q
\end{equation}
such that restrictions to \(U\) are isomorphisms -- the composition giving an inverse of \eqref{equation: orientation}.
The second morphism in \eqref{equation: factor} is (up to shift) \(\bD\) of the first morphism, where Verdier duality gives \(IC_Q=\bD(IC_Q)[2n]\) by \cite[\S6.1]{GoreskyMacPherson1983}.
Hypercohomology of \eqref{equation: factor} gives rise to homomorphisms
\begin{equation}\label{equation: cap product}
H^k(Q)\to IH_{2n-k}(Q)\to  H_{2n-k}(Q),
\end{equation} 
for every \(0\leq 2n-k\leq 2n\).
The cap product with fundamental class of Borel-Moore homology factors through \eqref{equation: cap product}.

Assume that \(Q\) is rationally smooth; i.e., \({\bF}_Q[2n]=IC_Q\) (cf.\ \cite[\S1.4]{BorhoMacPherson1983}).
It follows that \(IC_Q=\bD_Q\), hence \eqref{equation: factor} and \eqref{equation: cap product} are isomorphisms.
Given any rationally smooth variety \(Q\) and subvariety \(Y\subset Q\), let \(\delta_Y\in H^\ast(Q)\) be the unique class (possibly zero) such that
\begin{equation}\label{equation: PD}
\delta_Y\cap[Q]=[Y],
\end{equation}
where \(\cap\) is the action of \(H^\ast(Q)\) on Borel-Moore homology \(H_\ast(Q)\).

\section{Weighted flag varieties}\label{section: geometry}
The purpose of this section is to collect some basic results about the
geometry of weighted flag varieties.  Most of the results of this section are not essentially
new; the basic geometry of weighted flag varieties was presented in \cite{CortiReid2002}, and, 
with some additional results, in \cite{AbeMatsumura2015} and  \cite{AzamNazirQureshi2020}.
Unlike this previous work, we do not assume that the group we are using is the product of
$\C^\times$ with a reductive group.  This leads to a slightly larger class of varieties (see Remark
\ref{r:corti-reid}).  Also,
it allows for a more uniform Lie-theoretic presentation, which turns out to be helpful in understanding
positivity and related results.  The new results in this section are mainly contained in Section \ref{ss:weighted}.  In particular,
the definition of weighted roots and the identification of the weights of $T$-fixed curves in weighted flag varieties
seem to be new, although they are closely related to the GKM descriptions of cohomology given
in \cite{AbeMatsumura2015} and  \cite{AzamNazirQureshi2020} (see Remark \ref{r:curve-GKM}).

\subsection{Definitions and notation} \label{ss: definitions}
We refer to the introduction
for the definition of weighted flag varieties and basic notation.
We now make some additional definitions which will be used in the remainder
of the paper.  

Recall that $\Phi$, $ \Phi^+=\Phi(\fn)$, and
$\Delta$
denote (respectively)
the set of roots of $H$ in $G$, a choice of positive system,
and the corresponding set of simple roots.  As in the introduction, $m+1 = \dim H$.  Let
$n$ denote the semisimple rank of $G$, and let \(\set{\alpha_1,\ldots,\alpha_n}\) be an ordering of the simple roots.  
We will see below that $G$ cannot be semisimple, so $m \geq n$.
Let $\Phi^\vee$ denote the
set of coroots.  If $\alpha \in \Phi$, we denote by $r_{\alpha}$ the corresponding reflection in $W$.  We will abuse
notation and use the same letter to denote both an element of $W = N_G(H)/H$ and a representative in $N_G(H)$.
The longest element of $W$ is denoted $w_0$.

Given \(\mu\in\Lambda\), set \(\Phi^\mu=\set{\alpha\in\Phi\mid \mu\cdot\alpha^\vee>0}\).
Recall that we have fixed a dominant character \(\lambda\in\Lambda^+\),  and a cocharacter \(\chi\in\Lambda^\vee\).  By assumption, $L$ is a Levi subgroup such that the coroots of $L$ are orthogonal to $\lambda$.  Let $P = LU$ denote the corresponding standard parabolic subgroup; then
$\Phi^{\lambda} \subseteq \Phi(\fu)$, 
with equality if the coroots of $\fl$
are exactly the coroots orthogonal to $\lambda$.  By definition, $W^P$ consists of maximal length representatives of left $W_P$ cosets.  If $w \in W^P$, then
the minimal length element in $w W_P$ is $w_{min} = w w_0^P$, where $w_0^P$ is the longest element in $W_P$.
Given $w \in W^P$, let \(\Phi_w^P= \Phi^+ \cap w_{min} \Phi^- = \set{ \alpha \in \Phi \mid \alpha>0, w_{min}^{-1}\alpha<0}\).
Note that \(\Phi_w^P= \Phi^+ \cap w \Phi(\fu^-) \). 

We will use the following notation for the covering relations in the Bruhat order.  Given $v, w \in W$, 
write $w \lessdot v$ if $\ell(v) = \ell(w) + 1$ and $v = w r_{\gg} = r_{\gb} w$ for some positive root $\gg$; in this case $\gb = w(\gg)$ is also positive.
If $v$ and $w$ are both assumed to be in $W^P$, we write $w {\lessdot}_P v = w r_{\gg}$. 
The poset $W^P$ is connected by covering relations (this is an immediate consequence of the analogous result of Deodhar for minimal coset
representatives (\cite[Cor.~3.8]{Deodhar1977};  see \cite[Prop.~2.6]{BoeGraham2003}).

\subsection{Projectivity of weighted flag varieties} \label{ss: projective}
In \cite{CortiReid2002}, weighted flag varieties are constructed as closed subvarieties of weighted projective space,
so they are projective.  Because our definition is slightly more general than the one used in \cite{CortiReid2002},
we give the proof of this result in our setting.

For any weight $\gamma$, define $a_{\gamma} = \gcd(\chi)^{-1}\gamma \cdot \chi$.
For $w \in W$, we set  \(a_w=a_{w\lambda}\).  Observe that if \(\lambda\) is dominant and \(\chi\) is antidominant, then \(w \mapsto a_w\) is an increasing function on the Weyl group.
As in the introduction, we assume that for every \(w\in W\), we have 
\begin{equation}\label{equation: a_w}
a_w>0.
\end{equation}
This assumption implies that $G$ is not semisimple.  More generally,
we have:

\begin{proposition} \label{prop:not semisimple}
If there exists $\mu \in \Lambda$ such that $a_{w \mu} >0$ for each \(w\in W\), then 
 \(G\) is not semisimple (that is, \((m > n)\)).  Moreover, $\mu$ is not in the
 span of the roots.
\end{proposition}

\begin{proof}
We argue by contradiction.  Suppose \(G\) is semisimple, and identify $\mathfrak h$ with
$\mathfrak h^*$ by a nondegenerate $W$-invariant symmetric form.  
Choose $w_1$ and $w_2$ so that
$w_1 \mu$ is dominant and $w_2 \chi$ is antidominant.  If 
$\lambda_1, \ldots, \lambda_n$ are the fundamental
dominant weights, then $w_1 \mu = \sum c_i \lambda_i$ and $w_2 \chi = \sum d_i \lambda_i$,
where each $c_i \geq 0$ and $d_i \leq 0$.
If $w = w_2^{-1} w_1$, then
$\gcd(\chi) a_{w \mu} = w_2^{-1} w_1 \mu \cdot \chi = w_1 \mu \cdot w_2 \chi \leq 0$, since each $\lambda_i \cdot \lambda_j \geq 0$ (as follows from
\cite[Ex.~ 13.8]{Humphreys1978}).  This contradicts our assumption
that $a_{w \mu}>0$.  We conclude that \(G\) is not semisimple.

Let $\fh^* = \fh_0^* \oplus \fh_1^*$ be the $W$-isotypic decomposition of $\fh^*$,
where $\fh_1^*$ is the span of the roots, and let $\fh = \fh_0 \oplus \fh_1$
be the corresponding decomposition of $\fh$.  Write $\chi = \chi_0 + \chi_1$,
where $\chi_i \in \fh_i$.  Observe that $\fh_1$ is a Cartan subalgebra
of the semisimple Lie algebra $[\fg, \fg]$.  If $\mu \in \fh_1$, then
we would have $w \mu \cdot \chi = w \mu \cdot \chi_1 >0$ for all $w \in W$,
but this is impossible by the argument of the preceding paragraph.  Therefore
$\mu \not\in \fh^*_1$.
\end{proof}

As in the introduction, we write $L_{\lambda}$ for the kernel of $\lambda$ on $L$, $Q = L_{\lambda} U$,
and $Z = Z=G\times^P\bC_\lambda^\times \cong G/Q$.  The weighted flag variety is
\(X=S\backslash Z\).  Write $\pi: Z \to X$ for the quotient map. 

\begin{proposition} \label{prop:projective}
\(X=S\backslash Z\) is a geometric quotient in the category of schemes.
The quotient is a projective algebraic variety.
\end{proposition}

\begin{proof}
First assume that the coroots of $L$ are exactly the coroots of $G$ which are orthogonal to $\lambda$.
Let $V_\lambda$ denote the irreducible representation of
$G$ with highest weight $\lambda$.  The subgroup $S=\chi(\bC^\times)$ of $H$
acts on $V_\lambda$.  Since the weights of $V_\lambda$ are in the convex hull of the weights
$w \lambda$ for $w \in W$, and $a_w = w \lambda \cdot \chi >0$, the weights of $S$ on
$V_{\lambda}$ are strictly 
positive.  Therefore, the quotient $\bP = S\backslash (V_\lambda\smallsetminus\set{0})$
is a weighted projective space.  This is a geometric quotient of
$V_\lambda\smallsetminus\set{0}$ by $S$.  Our assumption on $L$ implies that $Q$ is the stabilizer of a
a highest weight vector $v_{\lambda}$ in $V_{\lambda}$.  
The orbit $G \cdot v_{\lambda} \cong G/Q = Z$ is an $S$-invariant closed subvariety of $V_\lambda\smallsetminus\set{0}$.  Therefore $S \backslash Z$ is a closed
subvariety of the weighted projective space $\bP$, and \(X=S\backslash Z\) is a geometric quotient of $Z$ by $S$.

More generally, $P$ is contained in a parabolic subgroup $P'$ satisfying the hypothesis of the
first sentence of the proof.  The map $\pi: Z = G \times^P \C_{\lambda} \to Z' = G \times^{P'} \C_{\lambda}$ is projective (with fibers isomorphic
to $P'/P$) and $S$-equivariant.  Because the quotient $X' = S \backslash Z'$ is projective,
it follows from \cite[Prop.~2.18]{FogartyKirwanMumford1994} that $Z$ is the stable locus of some $S$-equivariant ample line bundle,
and therefore a quasi-projective geometric quotient $X = S \backslash Z$ exists in the category of schemes (\cite[Ch.~1.4]{FogartyKirwanMumford1994}).
We claim that the map $\rho: X \to X'$ is proper.  This suffices, for then $X$ is complete (as $X'$ is), 
so $X$ is projective.  We now prove the claim.  By the results of Section \ref{subsection: stabilizers},
we can cover $Z'$ by $S$-invariant open sets of the form $C \times \C^*_{w \lambda} \isom C \times S/F$, where $F = \ker(w \lambda)|_S$.
The inverse image in $Z$ is isomorphic to $C \times w P'/P \times S/F$.  The map 
$C \times w P'/P \times S/F \to C \times S/F$ induces a map on the quotients by $S$,
which is isomorphic to the map $F \backslash (C \times w P'/P) \to F \backslash C$.  Since $P'/P$ is projective and $F$ is finite,
this map is proper.  Thus, we can cover $X$ by open sets of the form $F \backslash C$ such that
the map $\rho^{-1}(F \backslash C) \to F \backslash C$ is proper.  Hence, $\rho$ is proper.
\end{proof}

\begin{remark}\label{remark: proj}
If the coroots of $L$ are exactly the coroots of $G$ which are orthogonal
to $\lambda$, then the weighted flag variety \(X\) could be defined by the Proj construction as follows.
For every \(\mu\in \Lambda^+\), let \(V_\mu\) be the irreducible representation of \(G\) of highest weight \(\mu\), and let \(\mu^\ast\) be the highest weight of \(V_\mu^\ast\).
Recall that
\begin{equation}\label{equation: functions on G}
\bC[G]=\bigoplus_{\mu\geq0}V_\mu\otimes V_{\mu^\ast}
\end{equation}
by, e.g., \cite{Jantzen2003}.
The \(Q\)-invariant functions on \(G\) are
\begin{equation}\label{equation: functions on Z}
A = \bC[G]^Q =\bigoplus_{k\geq0} V_{k\lambda^\ast}\otimes\bC_{k\lambda}
\end{equation}
Observe that \(A\) is a \(\bZ_{\geq0}\)-graded \(\bC\)-algebra via the action of \(S\) on \(Z\).
We have \(G\times^P\bC_\lambda\to\overline{Z}\), where \(\overline{Z}=Z\cup\set{0}\subset V_\lambda\).
By \cite[\S8.13]{Jantzen2004}, $A = \C[G\times^P\bC_\lambda]
= \C[\overline{Z}]$ as left \(G\)-modules.
Write \(0\) for the point of \(\Spec(A)\) corresponding to the irrelevant maximal ideal.
Then \(\Proj(A)\) is the geometric quotient of \(\Spec(A)\smallsetminus\set{0}\) by \(S\), since \(A_0=\bC\) (cf.\ \cite{Mukai2003}).
We have \(Z=\Spec(A)\smallsetminus\set{0}\), so \(X\) is a projective algebraic variety defined as a geometric quotient of \(Z\) by \(S\).
\end{remark}

\begin{remark}
The argument in Remark~\ref{remark: proj} shows that if $G \neq Q$, then \(Z=G/Q\) is quasiaffine if and only if \(\lambda\) is dominant or antidominant and the coroots of \(L\) are exactly the coroots of \(G\) which are orthogonal to \(\lambda\).  
Indeed, if $L$ does not satisfy this condition, then let $P' = L' U'$ be as in the proof of Proposition \ref{prop:projective} and let
$Q' = L'_{\lambda} U'$.  Then \(Q\subsetneq Q'\), and \(G/Q\to G/Q'\to \Spec(\bC[G]^{Q'})\) shows that \(G/Q\to\Spec(\bC[G]^Q)\) is not an immersion, 
since \(\bC[G]^Q=\bC[G]^{Q'}\) by \eqref{equation: functions on Z} (or more directly since \(G/Q\to G/Q'\) is proper).
If \(\lambda\) is not dominant or antidominant, then \(\bC[G]^Q=\bC\) is one dimensional by the calculation in \eqref{equation: functions on Z}; since the only term remaining is when \(k=0\).  In this case, if $G/Q \subset \Spec(\bC[G]^Q)$ then $G = Q$.
\end{remark}

\begin{remark}
Every \(G\)-equivariant line bundle on a flag variety has the same embedded principal \(\bC^\times\)-bundle
(defined as the complement of the $0$-section) as the dual line bundle, since \(G\times^P\bC_\lambda^\times=G\times^P\bC_{-\lambda}^\times\) are both equal to \(G/Q\). 
Hence weighted flag varieties are projective varieties when \(\lambda\) is dominant or antidominant.
\end{remark}

Generalized flag varieties are a special case of weighted flag varieties.
Indeed, by Proposition \ref{prop:not semisimple}, $\lambda$
is not in the span of the roots, so there exists a rank one subtorus \(S_0\) of the 
center of \(G\) such that \(\lambda\) restricts nontrivially to \(S_0\).
Define \(\chi_0\) to be an isomorphism of \(\bC^\times\) with \(S_0\) such that \(\chi_0\) pairs positively with the restriction of \(\lambda\) to \(S_0\).
If we consider \(\chi_0\) as a cocharacter of \(H\), then \(\chi_0\) satisfies \eqref{equation: a_w}, and \(S_0\backslash G/Q=G/P\) since \(P=S_0Q\).  Therefore the weighted flag variety $S_0\backslash G/Q$ is the generalized flag variety
$G/P$.
We refer to this as the non-weighted case.  We fix \(\chi_0\) from now on; we 
write $Y = S_0\backslash G/Q=G/P$ for the non-weighted generalized flag variety,
and $\pi_0: Z \to Y$ for the quotient map.
Set \(G_0=G/S_0\), \(P_0=P/S_0\), and \(T_0=H/S_0\).  Note that \(T_0\) is a maximal torus of \(G_0\), and
\(Y = G/P=G_0/P_0\). 

The maps $\pi: Z \to X$ and $\pi_0: Z \to Y$ induce isomorphisms
$\pi^*: H^*_T(X) \to H^*_H(Z)$ and $\pi_0^*: H^*_{T_0}(Y) \to H^*_H(Z)$, as follows
from Proposition
 \ref{prop:Poincare}.  This isomorphism is described more precisely in Corollary
 \ref{corollary: pullback}.

The weighted flag variety is covered by quotients of affine space by finite groups (see Section \ref{subsection: stabilizers}), so \(X\) is normal with rational singularities, \(\bQ\)-factorial \cite[\S5.1]{KollarMori1998}, rationally smooth, and has orbifold singularities \cite{BartoloMartinMoralesOrtigasGalindo2014}.
Although $X$ need not be homogeneous, 
it does admit an action of any subgroup of $G$ centralizing $S$.
In particular, $X$ admits an $H$-action; since the subgroup
$S$ acts trivially, this action induces an action of $T = H/S$.

\begin{remark} \label{r:corti-reid}
The construction of weighted flag varieties in \cite{CortiReid2002} can be described in our notation as follows.
Let $G = G' \times \C^*$, where $G'$ is reductive.  Then $\lambda$ is of the form $(\lambda', u)$, where
$\lambda'$ is dominant for $G'$, and $u \in \mathbb{N}$.  The cocharacter $\chi$ is of the form $(\chi', 1)$, where
$\chi'$ is a cocharacter of $G'$, and $1$ corresponds to the identity cocharacter of $\C^*$.  Let $V_{\lambda}$ denote the
representation of highest weight $\lambda$ for $G$.  The weights of $S$ on $V_{\lambda}$ are given by
$\mu' \cdot \chi' + u$.  These weights are assumed to be strictly positive, which can be ensured by taking $u$ large enough.
This corresponds to our assumption that $a_w > 0$ for all $w \in W$.
The $G$-orbit of a highest weight vector $v_{\lambda}$ is $G \cdot v_{\lambda} = G/Q$.
The weighted flag variety is $S \backslash G /Q$, which is a closed subvariety of the weighted projective space 
$\bP = S \backslash (V_{\lambda} \setminus \{0\})$.  The construction in this paper is slightly more general in that
$G$ is not assumed to be of the form $G' \times \C^{\lambda}$, and $Q$ is only assumed to be contained in
the stabilizer of $v_{\lambda}$.  
\end{remark}

The following example shows that without the positivity of the $a_w$,
the quotient $X = S \backslash Z$ need not be separated (hence not projective).

\begin{example}\label{example: not projective}
Let \(G=SL_2\), with $B$ and $H$ the subgroups of upper triangular and diagonal matrices, respectively.
Let \(\lambda=1=\chi\), where we identify $\Lambda$ and $\Lambda^{\vee}$ with $\Z$ so that the dominant weights correspond to nonnegative integers.
Write \(W=\set{1,s}\).
Then \(a_1=1\) and \(a_s=-1\); in particular, \eqref{equation: a_w} is not satisfied.
The action of \(S \cong \C^{\times}\) on \(Z \cong \C^2 \smallsetminus \{0\} \) is
with weights $1$ and $-1$, so the action is closed (that is, the $S$-orbits are closed).  By \cite[Proposition 1.9]{FogartyKirwanMumford1994}, a universal geometric quotient exists.
The quotient \(X=S\backslash Z\) is not separated since the orbits through the axes form a doubled origin in the quotient.  In this example, $Z \cong \C^2 \setminus \{ 0 \}$ is 
covered by the $S$-invariant open sets $U_i$ where the $i$-th coordinate is nonzero.  Although
$S$ acts properly on each $U_i$, the action of $S$ on $Z$ is not proper.
This example shows that locally proper actions need not glue to a globally proper action.  Note that if we fix \(\lambda=1\), then there is no choice of \(\chi\) satisfying \eqref{equation: a_w}.
\end{example}

\begin{example}\label{example: quotientproj}
Let \(G=SL_2\times\bC^\times\) and \(\lambda=(1,1)\).
Then \(\chi=(0,1)\) gives \(a_0=1=a_1\) and hence \eqref{equation: a_w} is satisfied.
We have \(X\isom \bP^1\), the flag variety for $SL_2$.  
\end{example}

\subsection{An open covering and stabilizers}\label{subsection: stabilizers} 
In this section we describe a covering of a weighted flag variety by open sets
isomorphic to affine space modulo a finite group, as is done in \cite{CortiReid2002}.  We show that all stabilizers of \(Z\) in \(S\) are contained in a finite subgroup \(\bZ_{a_w}\) of $S$, depending on the neighborhood \(U_w\subset Z\).  In the next section we describe generic stabilizers for each weighted Schubert cell.  

Recall that \(P=LU\) is the Levi decomposition of our parabolic subgroup such that \(H\subset L\). Let \(U^-\) be the unipotent radical of the opposite parabolic \(P^-\) with respect to \(H\).
Given \(w\in W\), let \({}^wU^-=w U^-w^{-1}\), and
let \(h^w=w^{-1}hw\).  The map
\begin{equation}\label{equation: charts}
\varphi\,\colon {}^wU^-\to G/P,\quad  wu w^{-1}\mapsto wuP/P,
\end{equation}
takes \({}^wU^-\) isomorphically onto its image, which we denote by $C_w$. 
Let \(\zeta\,\colon G\to G/P\) denote the projection.
Then \(\zeta^{-1}(C_w) = {}^wU^-\times wP \subset G\), where
\({}^wU^-\times wP\) is identified
with its image in $G$ under the multiplication map.  This image is invariant under left multiplication
by $H$, which corresponds to the action of $H$ on
\({}^wU^-\times wP\) given by 
\begin{equation} \label{e.action}
h( wu w^{-1}, wp)=( wh^wu(h^w)^{-1} w^{-1}, wh^wp).
\end{equation}

Let $U_w$ denote $\pi_0^{-1}({}^wU^-) \subset Z$.

\begin{lemma} \label{lemma: uw}
Suppose $w \in W^{P}$.  As \(H\)-varieties, we have
\begin{equation}
U_w\isom {}^wU^-\times\bC^\times_{w\lambda}.
\end{equation}
\end{lemma}

\begin{proof}
We have 
\begin{equation}\label{equation: U_w}
U_w=\pi_0^{-1}({}^wU^-)\isom ({}^wU^-\times wP)\times^P\bC_\lambda^\times
\end{equation}
The action of \(H\subset G\) on \(U_w\) corresponds by \eqref{e.action}
to the action of $H$ on $({}^wU^-\times wP)\times^P\bC_\lambda^\times$ given by
\begin{equation}
h[( wu w^{-1}, wp),v] = [( wh^wu(h^w)^{-1} w^{-1},w h^w p),v].
\end{equation}
The isomorphism $({}^wU^-\times wP)\times^P\bC_\lambda^\times \to {}^wU^-\times\bC^\times_{w\lambda}$ given by $[(wuw^{-1}, wp), v] \mapsto (wuw^{-1}, \lambda(p) v)$
is $H$-equivariant.  The lemma follows.
\end{proof}

Recall that
$a_\mu=\gcd(\chi)^{-1}\mu\cdot\chi$ for any weight \(\mu\in\Lambda\), and 
and \(a_w=a_{w\lambda}\)  for $w \in W$.  For $s \in S$, write $s^{a_{\mu}} = \mu(s)$, and
$s^{a_w} = s^{a_{w\lambda}}$.  Under the identification of  \(S\) with \(\bC^\times\) given by
\(\bC^\times \cong \bC^\times / \ker \chi \to S\), the map \(s \mapsto s^{a_\mu}\) is just the map raising \(s\) to the power \(a_\mu\).
Let \(\bZ_{a_w}\) denote the finite group \(\ker (w \lambda \vert_S)\).   Because the action of 
$S$ on $Z$ is proper, $S$ acts on $Z$ with finite stabilizers; this is also a consequence
of the following result, which gives more precise information.

\begin{corollary} \label{corollary: stab}
The stabilizer in $S$ of any $z \in U_w$ is contained in \(\bZ_{a_w}\).  Hence
$S$ acts on $Z$ with finite stabilizers.  Moreover, 
for \(w\in W^P\), the map \({}^wU^-\to S\backslash U_w\) defined by
\begin{equation}\label{equation: chart isoms}
u \mapsto S(u,1)
\end{equation}
induces an isomorphism
\begin{equation}\label{equation: chartsquotient}
\bZ_{a_w}\backslash {}^wU^-\isom S\backslash U_w.
\end{equation}
The \(S\backslash U_w\) form an open cover of \(X\) by quotients of affine space by finite groups.
Hence $X$ is rationally smooth.
\end{corollary}

\begin{proof}
Lemma \ref{lemma: uw} implies that  the stabilizer group 
$S^z$ (for $z \in U_w$) is contained in $\bZ_{a_w}$.  This lemma also
implies that \eqref{equation: chart isoms} induces the isomorphism \eqref{equation: chartsquotient}.
Since the $C_w$ form an open cover of $G/P$, their inverse images $U_w$ in $Z$ form an open cover of $Z$.   Each ${}^wU^-$ is isomorphic to affine space, and since $Z \to X = S \backslash Z$
is a geometric quotient, the $S\backslash U_w \isom \bZ_{a_w}\backslash {}^wU^-$ form
an open cover of \(X\) by quotients of affine space by finite groups. 
\end{proof}

\subsection{Weighted Schubert varieties} \label{subsection: Schubert}
 Recall that $\pi: Z \to X$ and $\pi_0: Z \to Y = G/P$ denote the natural projections.  Suppose
\(w\in W^P\).
By definition, the Schubert cell $Y^0_w$ corresponding to $w$ in $Y$ is the $B$-orbit 
$B \cdot wP$.  The map $u \mapsto uP$ gives an isomorphism
of ${}^wU^- \cap U$ onto this cell.  Define
$Z_w^0 = \pi_0^{-1}(Y^0_w) = ({}^wU^- \cap U) \cdot w P \times^P \bC^\times_{\lambda}$.
Lemma \ref{lemma: uw} implies that $Z_w^0 \isom ({}^wU^- \cap U) \times \bC^\times_{w\lambda}$
as $H$-varieties.
We define the weighted Schubert cell $X_w^0 \subset X$ by
$$
X_w^0 = S \backslash Z_w^0 \cong S \backslash (({}^wU^- \cap U) \times \bC^\times_{w\lambda}).
$$
The weighted Schubert variety $X_w$ is the closure of
$X_w^0$ in $X$.  

\begin{proposition}
    The weighted Schubert cells $X_w^0$ form a paving of $X$ by affine spaces modulo finite groups.  The equivariant cohomology $H^*_T(X)$ is a free $H^*_T$-module with basis
    $\delta_{X_w}^T$.
\end{proposition}

\begin{proof}
The Schubert cells form a paving of $Y$, so their inverse images $Z_w^0$ form a paving of $Z$.
As $\pi: Z \to X$ is a geometric quotient, the quotient varieties $S \backslash Z_w^0 = X_w^0$
form a paving of $X$.  Since $X_w^0 = S \backslash Z_w^0 \isom \bZ_{a_w}\backslash ({}^wU^- \cap U)$,
each $X_w^0$ is isomorphic to affine space modulo a finite group.  Each
$X^w_0$ is $T$-acyclic, by Proposition \ref{p:acyclic} (cf.~Remark \ref{remark:orbifold}).  Since $X$ is rationally
smooth, the statement about cohomology follows from Corollary~\ref{corollary: paving}.
\end{proof}

We now determine the generic stabilizers of $S$ on $Z_w^0$.
The \emph{support} of the \(S\)-action on \(Z\), denoted by $\sigma(S)$,
is the set of all $s \in S$ such that $s z = z$ for some $z \in Z$.
We have
\begin{equation}
\sigma(S)=\bigcup_{w\in W}\bZ_{a_w}.
\end{equation}
Indeed, $\sigma(S) \subset \bigcup_{w\in W}\bZ_{a_w}$ by Corollary \ref{corollary: stab};
the reverse inclusion holds since 
  \(S^{[w,1]}=\bZ_{a_w}\) for all $w \in W$.

Let \(w\in W^P\), and choose an ordering $\beta_1, \ldots, \beta_r$ of the roots in $\Phi_w^P =  \Phi^+ \cap w \Phi(\fu^-)$.
For each root $\beta$, the corresponding root subgroup $U_{\beta}$ is isomorphic
to $\C$.  We have isomorphisms of algebraic varieties
$\C^r \isom \prod_i U_{\beta_i} \isom {}^wU^- \cap U$, where the second isomorphism
is given by the product map.  If $u = \prod u_{\beta_i}$, we can view
the $u_{\beta_i}$ as complex numbers giving the coordinates
of $u \in {}^wU^- \cap U$ under the isomorphism
$\C^r \isom {}^wU^- \cap U$.  Let $({}^wU^- \cap U)^{gen}$ be the Zariski open
subset of ${}^wU^- \cap U$ consisting of the elements $u$ whose coordinates $u_{\beta_i}$ are
all nonzero.  Define the Zariski open subset $Z^{gen}_w$ of $Z^0_w$ by
$$ 
Z^{gen}_w = ({}^wU^- \cap U)^{gen{}} \cdot w P \times^P \bC^\times_{\lambda} \isom ({}^wU^- \cap U)^{gen} \times \bC^\times_{w\lambda}.
$$

For any subset \(\cS\subset\Lambda\), write \(a_\cS=\set{a_\mu\mid \mu\in \cS}\).
Write $\gcd(a_w,a_{\Phi_{w}^P})$ for the greatest common divisor of the elements
of  $\{a_w\} \cup a_{\Phi_{w}^P}$.  

\begin{proposition}\label{prop:stabs}
Let \(w\in W^P\).  For all $z$ in $Z^{gen}_w$, we have 
$S^z=\bZ_d$, where $d = \gcd(a_w,a_{\Phi_{w}^P})$.
\end{proposition}

\begin{proof}
    The point $z = (u,z) \in Z^{gen}_w$ is fixed by $s \in S$ if and only if
    $s^{a_{\beta_i}} = 1$ for all $i$, and $s^{a_w} = 1$.  This is equivalent to the
    statement that $s^d = 1$, where $d = \gcd(a_w,a_{\Phi_{w}^P})$, which is equivalent
    to the statement that $s \in \bZ_d$.
\end{proof}

The weight $\lambda$ is said to be \emph{minuscule}
if every \(\alpha\in\Phi\) satisfies \(|\lambda\cdot\alpha^\vee|\leq1\).
Let \(\leq_P\) denote the restriction of the Bruhat order on \(W\) to \(W^P\).  The next proposition provides 
an alternative description of the stabilizer groups if $\lambda$ is minuscule.

\begin{proposition}\label{proposition: equivalent stabs}
For every \(w\in W^P\), we have
\begin{equation}\label{equation: equivalent stabs}
\gcd(a_w,a_{\Phi_w^P})\ |\ \gcd(a_x\mid x\leq_P w).
\end{equation}
If \(\lambda\) is minuscule, then \eqref{equation: equivalent stabs} is an equality.
\end{proposition}

\begin{proof}
The proposition is true for the minimal element $w_0^P$ of $W^P$, since both sides
of \eqref{equation: equivalent stabs} are equal to $a_{w_0^P}$.
Since $W^P$ is connected by covering relations (see Section \ref{ss: definitions}),
it suffices to show that for every \(x\leq_P w\) such that \(\ell(x)=\ell(w)-1\), we have
\begin{equation}\label{equation: generic stabilizer in X_w}
\gcd(a_w,a_{\Phi_w^P})\ |\ \gcd(a_x,a_{\Phi_x^P}).
\end{equation}
Indeed, by induction on \(\ell(w)\), \(\gcd(a_x,a_{\Phi_x^P})\) will divide \(\gcd(a_y\mid y\leq_P x)\), so it will follow that \(\gcd(a_w,a_{\Phi_w^P})\) divides
\begin{equation*}
\gcd(\gcd(a_y\mid y\leq_P x_1),\ldots,\gcd(a_y\mid y\leq_P x_d),a_w)=\gcd(a_x\mid x\leq_P w),
\end{equation*}
where \(x_1,\ldots,x_d\) index all of the divisors in \(Y_w\).

Let \(\alpha>0\) 
such that \(w=r_\alpha x\), where $x,w \in W^P$.  We claim that \(\alpha\in \Phi_w^P\).  It suffices to show
that $w^{-1}(\alpha) \in \Phi(\fu^-)$.  Since 
$r_\alpha w < w$, we have $w^{-1} \alpha <0$.  Moreover, since $x \in W^P$,
if $w^{-1} \alpha \in \Phi(\fl)$, then $x w^{-1} \alpha$ and $w^{-1}\alpha$ would
have opposite signs, so $x w^{-1} \alpha = -\alpha$ would be positive, which
is a contradiction.  Therefore $w^{-1}(\alpha) \in \Phi(\fu^-)$, proving the claim.

We now prove \eqref{equation: generic stabilizer in X_w}.  
Suppose that $d$ divides $a_w$ and $a_{\gb}$ for each $\gb \in \Phi_w^P$.
It suffices to show that $d$ divides $a_x$ and $a_{\gg}$ for each $\gg \in \Phi_x^P$.
Since \(x=r_\alpha w\), we have \(a_x=a_w-(w\lambda\cdot\alpha^\vee)a_\alpha\).
Since $\alpha \in \Phi_w^P$, $d$ divides $a_x$.  The set of weights of 
\(H\) on \(T_xY_x\) is equal to \(\Phi_x^P\), and
by \cite[(10.4)]{Carrell1995}, the set of weights of 
\(H\) on \(T_xY_w\) is equal to \(r_\alpha\Phi_w^P\).
It follows that \(\Phi_x^P\subset r_\alpha\Phi_w^P\), since \(Y_x\subset Y_w\).
Therefore, if $\gg \in \Phi_x^P$, $a_{\gg}$ is a linear combination of $a_{\alpha}$ and $a_{\gb}$ for some $\gb \in \Phi_w^P$, so $d$ divides $a_{\gg}$.
This proves \eqref{equation: generic stabilizer in X_w}.  

Now suppose that $\lambda$ is minuscule.  To show that \eqref{equation: equivalent stabs} is an equality, we must show that if $d | a_x$ for all $x \leq_P w$, then
$d|a_{\gb}$ for all $\gb \in \Phi_w^P$.  As above, let $\ga>0$ be such that \(w=r_\alpha x\), where $x,w \in W^P$, and $\ell(x) = \ell(w) -1$.  We have $a_x = a_w - (w \lambda \cdot \alpha^{\vee}) a_{\alpha}$.  Now, $w \lambda \cdot \alpha^{\vee} = \lambda \cdot w^{-1} \alpha^{\vee} =-1$ since $w^{-1} \alpha \in \Phi(\fu^-)$.  Therefore, since
$d$ divides $a_w$ and $a_x$, it divides $a_{\alpha}$.  
$\Phi_w^P = r_{\alpha} \Phi_x^P \cup \{ \alpha \}$ (since the left side of this equation contains the right side,
and $|\Phi_w^P| = | \Phi_x^P | + 1)$.  By induction, 
if $\nu \in \Phi_x^P$ then $d | a_{\nu}$; since $d| a_{\alpha}$, we see that $d | a_{r_{\alpha} \nu}$.  We conclude that $d|a_{\gb}$ for all $\gb \in \Phi_w^P$, as desired.
\end{proof}

\begin{remark}
Weighted Grassmannians are a case where the weight $\lambda$ can be taken to be minuscule, and then the previous proposition implies that \eqref{equation: equivalent stabs} is an equality.
\end{remark}

\begin{example}
Let \(G=GL_2\).  Identifying $\Lambda$ with $\Z^2$ as usual, let \(\lambda=(2,0)\), \(\chi=(2,3)\).  The positive root is \(\alpha=(1,-1)\), and the nontrivial Weyl
group element is $w=r_{\alpha}$.  We have
\(\Phi_w=\set{\alpha}\) and \(w\lambda=(0,2)\).  Also,
$a_{\alpha}=-1$, $a_w=6$, and $a_e=4$.
We have \(\gcd(a_w,a_{\Phi_w})=1\) and \(\gcd(a_w,a_e)=2\), so \eqref{equation: equivalent stabs} need not be an equality.
\end{example}

Let \(w_0\in W\) be the longest element of \(W\), which is also the longest element of \(W^P\).
Define positive integers $q_w$ for $w \in W^P$ by the formula
\begin{equation} \label{e.qwdef}
    q_w = \dfrac{\gcd(a_w,a_{\Phi_{w}^P})}{\gcd(a_{w_0},a_{\Phi_{w_0}^P})}.
\end{equation}

\begin{corollary}\label{corollary: pullback}
The isomorphisms $\pi^*: H^*_T(X) \to H^*_H(Z)$ and $\pi_0^*\,\colon H_{T_0}^\ast(Y) \to H^*_H(Z)$ satisfy
\begin{equation}\label{equation: q_w}
\pi^\ast(\delta_{X_w}^{ T})= q_w \delta_{Z_w}^H.
\end{equation}
and
\begin{equation}\label{equation: triv}
  \pi_0^\ast(\delta_{Y_w}^{ T_0})= \delta_{Z_w}^H.  
\end{equation}

\begin{proof}
The equality \eqref{equation: q_w} follows by
combining Proposition \ref{prop:stabs} and \eqref{equation: pullback}.
The equality \eqref{equation: triv} follows from \eqref{equation: pullback} since
$S_0$ acts with constant stabilizers on $Z$.
\end{proof}
\end{corollary}

\subsection{Fixed points, invariant curves and weighted roots} \label{ss:weighted}
In this section we describe the $T$-fixed points and $T$-invariant
curves on $X$.   To describe the weights of these curves, we introduce the notion of weighted roots
at an element $w$ of $W^P$.

For $w \in W^P$, let $p_w = S \backslash (wP \times^P \bC_{\lambda}^\times)$.
Note that as an $H$-variety, $wP \times^P\bC_\lambda^\times$
is isomorphic to $\bC_{w \lambda}^\times$.

\begin{proposition} \label{prop:fixedpoints}
    We have $X^T = \{ p_w \mid w \in W^P \}$.
\end{proposition}

\begin{proof}
    The inverse image in $Z$ of a $T$-fixed point on $X$ is a $1$-dimensional $H$-orbit, so it suffices to show that the 
    $1$-dimensional $H$-orbits on $Z$ are the $wP \times^P\bC_\lambda^\times$.
    Any $H$-orbit is contained in an open subset $U_w$ of $Z$ for some 
    $w \in W^P$.  The subset $U_w$ is $H$-equivariantly isomorphic to  ${}^wU^-\times\bC^\times_{w\lambda}$, and under this isomorphism,
    $wP \times^P\bC_\lambda^\times$ corresponds to the $1$-dimensional $H$-orbit $\{ 1 \}\times \bC_\lambda^\times$.
    Since $\lambda$ is not in the span of the roots, the
    stabilizer in $H$ of $(u,z) \in {}^wU^-\times\bC^\times_{w\lambda}$
    will have codimension at least $2$ if $u \neq 1$, so
    $\dim H \cdot (u,z) \geq 2$.  Hence $\{ 1 \}\times \bC_\lambda^\times$ is the only $1$-dimensional
    $H$-orbit in ${}^wU^-\times\bC^\times_{w\lambda}$.  The result follows. 
\end{proof}

\begin{lemma} \label{lemma: fixedpointcorr}
  Suppose $M$ is a $H$-invariant subvariety of $Z$.  Let
  $N_0 = S_0 \backslash M \subset Y$, and $N = S \backslash M \subset X$.  Then $wP \in N_0$ if and only if $p_w \in N$.
\end{lemma}

\begin{proof}
    Arguing as in Proposition \ref{prop:fixedpoints} shows that 
    $(G/P)^{T_0}$ consists of the points $S_0 \backslash (wP \times^P \bC_{\lambda}^\times)$.  Thus,  $wP \in N_0$ $\Leftrightarrow$ $wP \times^P \bC_{\lambda}^\times$
    is contained in $M$ 
    $\Leftrightarrow$ $p_w \in N$.
\end{proof}

Let $w \in W$.  Define a map $\H^* \to \T^*$, $\mu \mapsto \overline{\mu}(w)$, by the formula 
\begin{equation} \label{e.weighted1}
\overline{\mu}(w) = \mu - \dfrac{a_\mu}{a_w} w \lambda.
\end{equation}
(note that $\overline{\mu}(w)$ is in $\T^*$ since it is orthogonal to $\chi$).
Since $W_P$ fixes $\lambda$,
$\overline{\mu}(w)$ only depends on the coset $wW_P$.
Given $\beta \in \Phi$, we define 
the weighted
root corresponding to $\beta$ at $w$ by the formula
\begin{equation} \label{e.weighted2}
\beta(w) = \dfrac{a_w}{\gcd(a_w,a_\beta)} \overline{\beta}(w) = \frac{1}{\gcd(a_w,a_\beta)} (a_w \beta - a_{\beta} w \lambda).
\end{equation}
Since $a_w$ is positive by hypothesis, $\dfrac{a_w}{\gcd(a_w,a_\beta)}>0$, so $\beta(w)$ and $\overline{\beta}(w)$ are positive scalar
multiples of each other.  Although the weighted roots $\beta(w)$ are geometrically natural, we will frequently use the $\overline{\beta}(w)$
for convenience.
Although the map $\beta \mapsto \beta(w)$ is not linear, if $\ga = - \beta$, then $\ga(w) = - \beta(w)$.  If $\beta$ is a negative root, then
$c = \frac{a_{\beta}}{a_w}$ is nonnegative, since by hypothesis, $a_w >0$, and
because $\chi$ is antidominant, $a_{\beta} \geq 0$.  Hence, 
$\beta = \overline{\beta}(w) + c \cdot w \lambda$ where $c \geq 0$

We record the following simple lemma for reference.

\begin{lemma} \label{lemma: linear}
The maps $\mu \mapsto a_{\mu}$ and $\mu \mapsto \overline{\mu}(w)$ (for fixed $w \in W$) are linear.
\end{lemma}

\begin{proof}
The first statement holds because $a_{\mu} = \mu \cdot \chi$, and the second follows from this and the formula for $\overline{\mu}(w)$.
\end{proof}

\begin{lemma} \label{lemma: weightedrefl}
Let $\ga \in \Phi$.  Then $\ga(r_{\alpha}w) = \ga(w)$. 
Hence if $\beta = -\alpha$, then
    $\beta(r_{\alpha}w) = - \alpha(w) $.
\end{lemma}

\begin{proof}
Let $c = w \lambda \cdot \alpha^{\vee}$, so $r_{\alpha} w \lambda = w \lambda - c \alpha$.  Let
$d = \gcd(a_w, a_{\alpha}) = \gcd(a_{r_{\ga}w}, a_{\alpha})$.
Then $\ga(w) = \frac{1}{d}(a_w \ga - a_{\alpha} w \lambda$), and
$$
\ga(s_{\ga} w)  = \frac{1}{d}(a_{r_{\ga} w} \ga - a_{\ga}  r_{\ga} w \lambda)  = \frac{1}{d}((a_w - c a_{\ga}) \ga - a_{\ga} (w \lambda - c \ga) )= \ga(w),
$$
proving the first statement.  The second statement follows immediately.
\end{proof}

The following proposition will be used to show
that the weighted roots are the weights of $T$-invariant curves
in $X$.  In fact, this proposition was the motivation for the 
definition of weighted roots.

\begin{proposition} \label{proposition: weighted}
Let $\alpha,\beta \in \Lambda$ be such that $ a_{\alpha} $ 
and $a_{\beta}$ are nonzero.  Let
\begin{equation}\label{equation: gamma}
\gamma=\dfrac{a_{\beta} \alpha- a_{\alpha} \beta}{\gcd(a_{\alpha},a_{\beta})}.
\end{equation}
Then as $H $-spaces, 
$$
S\backslash(\bC_\alpha\times\bC^\times_\beta) \isom 
\begin{cases} \bC_{\gamma} & \text{if  }a_{\gb}>0 \\
               \bC_{-\gamma} & \text{if  }a_{\gb}<0. 
\end{cases}
$$
In particular, if $\beta = w \lambda$, then $\gamma = \alpha(w)$.
\end{proposition}

\begin{proof}
We prove the result for the case $a_{\gb}>0$; the case $a_{\gb}<0$ is similar.
The group \( S \isom \bC^\times\) acts on \(\bC_{\alpha}\times\bC_{\beta}^\times\) with weights 
$a = a_{\alpha} $ and $b = a_{\beta}$.
Since \(\bC^\times\) is a reductive group acting properly on an affine variety, the universal geometric quotient exists and is an affine variety given by Spec of the ring of invariant functions (see\ \cite[Ch.~1, \S 2]{FogartyKirwanMumford1994}).  We identify $\bC[\bC_{\alpha}\times\bC_{\beta}^\times]$
with $\bC[x,y,y^{-1}]$, where $S$ acts on $\bC[x,y,y^{-1}]$ by
\begin{equation}\label{equation: action on functions}
(sf)(x,y)=f(s^{-a}x,s^{-b}y).
\end{equation}
Suppose the monomial $f(x,y)=x^ky^{-\ell}$ is $S$-invariant; then $ak = b \ell$.  First suppose
$a >0$.  Since $k \geq 0$, both 
both $ak$ and $b \ell$ must equal $c \cdot \lcm(a,b)$ for some nonnegative integer $c$.
Since $\lcm(a,b) = \frac{ab}{\gcd(a,b)}$, we see that 
\begin{equation} \label{equation: kl}
k = c \frac{b}{\gcd(a,b)}, \hspace{.5in} \ell = c \frac{a}{\gcd(a,b)}.
\end{equation}
On the other hand, if $a<0$, then both $ak$ and $b \ell$ must equal $-c \cdot \lcm(a,b)$ for some 
nonnegative integer $c$.  Since $\lcm(a,b) = - \frac{ab}{\gcd(a,b)}$, we see that
$k$ and $\ell$ are again given by \eqref{equation: kl}.  This implies that
$\bC[x,y,y^{-1}]^{S} = \bC[z]$, where $z= x^{b\gcd(a,b)^{-1}}y^{-a\gcd(a,b)^{-1}}$
Since $z$ is an $H$-weight vector
of weight $-\gamma$,  we see that as $H$-spaces,
$$
\bC[\bC_{\alpha}\times\bC_{\beta}^\times]^S = \Spec \C[z] \cong \C_{\gamma}.
$$
The last statement of the lemma follows from the definition of
$\alpha(w)$.
\end{proof}

Suppose $V$ is a $2$-dimensional
representation of $H$ with with weights $\mu_1$ and $\mu_2$.
In this case, we say $H$ acts on $\bP^1 = \bP(V)$ with weight $\alpha = \mu_1 - \mu_2$.  
With this definition, an action on $\bP^1$ of weight $\alpha$ is
also an action of weight $-\alpha$.

\begin{proposition} \label{prop:invariantcurves}
    The irreducible closed $T$-invariant curves in $X$ are the curves $C = S \backslash (\overline{U_{\alpha}w P\times^P \bC_\lambda})$ for
     \(\alpha \in w \Phi(\fu^-)\).  
   These 
   are the closures of the $1$-dimensional
    $T$-orbits on $X$.  Each such curve is 
    $H$-equivariantly isomorphic to $\bP^1$, where $H$ acts on $\bP^1$
    with weight $\alpha(w)$.  The $T$-fixed points
    on $C$ are $p_w$ and $p_{r_{\alpha} w}$.
\end{proposition}

\begin{proof}
Let $C$ be an irreducible closed $T$-invariant curve in $X$.
Let $M = \pi^{-1}(C)$ and $C_0 = S_0 \backslash M$.
    Then $M$ is an irreducible $2$-dimensional $H$-stable $2$-dimensional subvariety of $Z$.  For some $w \in W^P$, the intersection
    $U_w \cap M$ is a nonempty open subset of $M$, and hence
    a closed $2$-dimensional $H$-invariant subvariety of $U_w$.  
    Arguing as in the proof of Proposition \ref{prop:fixedpoints} shows
    that $U_w \cap M = U_{\alpha} w P \times^P \bC_{\lambda}^\times$
    for some $\alpha \in w \Phi(\fu^-)$
    Hence
    $\pi_0(U_w \cap M) = U_{\alpha} wP$, so
    $C_0 = \overline{U_{\alpha} wP}$.  Therefore, by a result of
    Carrell and Peterson (see \cite{Carrell94}), the $T_0$-fixed points in $C_0$ are
    $wP$ and $r_{\alpha} wP$.  By Lemma \ref{lemma: fixedpointcorr}, 
    the $T$-fixed points in $C$ are $p_w$ and $p_{r_{\alpha} w}$.  
    Since $M = \pi_0^{-1}(C_0)$, it is covered by the open sets
    $U_{\alpha} w P \times^P \bC_{\lambda}^\times$ and
    $U_{\beta} r_{\alpha} w P \times^P \bC_{\lambda}^\times$, where $\beta = - \alpha$.  These
    sets are $H$-equivariantly isomorphic to 
    $U_{\alpha} \times \bC_{w\lambda}^\times$ and $U_{\beta} \times \bC_{w\lambda}^\times$,
    respectively.  Proposition \ref{proposition: weighted} implies that the quotients
    of these sets by $S$ are isomorphic to $\C_{\alpha(w)}$ and
    $\C_{\beta(r_{\alpha} w)} \cong \C_{-\alpha(w)}$, respectively,
    where the last isomorphism follows from Lemma 
    \ref{lemma: weightedrefl}.  These open sets glue along their intersection
    to yield a $\bP^1$ on which $H$ acts with weight $\alpha(w)$.
\end{proof} 

\subsection{More about weighted roots}
Let $\beta_1, \ldots, \beta_n$ denote the set of negative simple
roots.   This section contains some results about the $\overline{\beta_i}(w)$, which are positive
scalar multiples of the weighted roots $\beta_i(w)$ (see \eqref{e.weighted2}).

The following lemma is a version of a well-known fact; for the convenience of the
reader we include a proof.  As usual, the weight $\lambda$ is assumed dominant.

\begin{lemma} \label{l: w-lambda}
Let $w \geq v$ be elements of $W$.  Then $w \lambda - v\lambda = \sum e_i \beta_i$, where
all $e_i \geq 0$.
\end{lemma}

\begin{proof}
It suffices to prove the lemma in the case $w = v r_{\ga} > v$, where $\ga$ is a positive root.  
Write $w \lambda - v{\lambda} = \sum e_i \beta_i$.
We have
$$
w \lambda = v r_{\ga} \lambda = v ( \lambda - (\lambda \cdot \ga^{\vee}) \ga) = v \lambda -  (\lambda \cdot \ga^{\vee}) v \ga.
$$
Therefore 
$w \lambda - v \lambda = -  (\lambda \cdot \ga^{\vee}) v \ga = \sum e_i \beta_i$.  
Since $w > v$, we have $v \ga > 0$, and since $\lambda$ is dominant, $ (\lambda \cdot \ga^{\vee}) \geq 0$.
Hence each $e_i \geq 0$.
\end{proof}

The following lemma provides a way to express a weighted root at $v$ in terms of weighted roots at $w$.

\begin{lemma} \label{lemma: positive}
  Let $w, v \in W^{P}$, and suppose that $w \lambda - v \lambda =   \sum_i e_i \beta_i$, where the $\beta_i$ are the negative
simple roots.  Then
 \begin{equation} \label{e:positive2}
 \overline{w \lambda}(v) = \frac{a_w}{a_v} \sum_i e_i \overline{\beta}_i (w),
  \end{equation}
and for $\beta \in \H^*$, 
 \begin{equation} \label{e:positive3}
 \overline{\beta}(v) =  \overline{\beta}(w) + \frac{a_{\beta}}{a_v} \sum_i e_i \overline{\beta}_i (w).
  \end{equation}
Hence, if $\beta$ is a negative root and $w \geq v$, then $\overline{\beta}(v)$ is a nonnegative
  linear combination of $\overline{\beta}_i (w)$.
\end{lemma}

\begin{proof}
A straightforward computation shows that
 \begin{equation} \label{e:positive1}
  a_v \, \overline{w \lambda}(v) = - a_w \, \overline{v \lambda}(w).  
 \end{equation}
For \eqref{e:positive2}, by linearity of the map $\mu \mapsto \overline{\mu}(w)$,
we have
$$
 \overline{w \lambda}(w) - \overline{v \lambda}(w) = \sum_i e_i \overline{\beta}_i (w).
$$
Equation \eqref{e:positive2} follows from this equation, \eqref{e:positive1}, and the equality $\overline{w \lambda}(w) = 0$. We have
\begin{align*}
\overline{\beta}(v) &= \beta - \frac{a_{\beta}}{a_v} v \lambda =  \overline{\beta}(w) + \frac{a_{\beta}}{a_w} w \lambda - \frac{a_{\beta}}{a_v} v \lambda \\
& = \overline{\beta}(w) + \frac{a_{\beta}}{a_w} \overline{w \lambda} (v) = \overline{\beta}(w) + \frac{a_{\beta}}{a_w} \frac{a_w}{a_v} \sum_i e_i \overline{\beta}_i (w) \\
& =  \overline{\beta}(w) + \frac{a_{\beta}}{a_v} \sum_i e_i \overline{\beta}_i (w),
\end{align*}
proving \eqref{e:positive3}.  Finally, if $\beta$ is a negative root then $\beta = \sum f_i \beta_i$ where $f_i \geq 0$; also,
$a_v >0$ by hypothesis, and
$a_{\beta} \geq 0$ as $\chi$ is antidominant. If
$w \geq v$ then $e_i \geq 0$.  Thus, in the expression
\begin{equation} \label{e:positive4}
\overline{\beta}(v) = \sum_i (f_i + \frac{a_{\beta}}{a_v} e_i) \overline{\beta}_i (w),
  \end{equation}
  all coefficients on the right hand side are non-negative.
\end{proof}

\subsection{GKM description of cohomology}\label{subsection: GKM}
Let $j: X^T = \{ p_w \} \hookrightarrow X $ be the inclusion,
and $j^*: H^*_T(X) \to H^*_T(X^T) = \oplus_{w \in W^P} H^*_T(p_w)$ the pullback.
We identify $H^*_T(p_w)$ with $H^*_T$.  The next proposition gives the GKM description
of $H^*_T(X)$, generalizing the descriptions given in \cite{AbeMatsumura2015} for the case of weighted Grassmannians and in \cite{AzamNazirQureshi2020}  for type \(A\) weighted flag varieties.

\begin{proposition} \label{prop:GKM}
    The pullback $j^*: H^*_T(X) \to H^*_T(X^T) = \oplus_{w \in W^P} H^*_T(p_w)$
    is injective, and the image of $j^*$ consists of tuples $(f_w)_{w \in W^P}$
    such that for each $\alpha$, $f_w - f_{r_\alpha w}$ is divisible by
    the weighted root $\alpha(w)$.
\end{proposition}

\begin{proof}
    Since $X$ has vanishing odd cohomology, it is equivariantly formal
    (see \cite{GKM1998}).  In light of the description of 
    $T$-invariant curves and their weights
    given in Proposition \ref{prop:invariantcurves}, the proposition
    is an immediate consequence of \cite[Theorem 1.2.2]{GKM1998}.
\end{proof}

\begin{remark} \label{r:curve-GKM}
In contrast to the proof given here,
the proofs of the GKM description
in \cite{AbeMatsumura2015} and \cite{AzamNazirQureshi2020}
(for the weighted Grassmannian and type $A$)
do not proceed by identifying the weights of $T$-invariant curves
in $X$.  Instead, they use the GKM description of $H^*_{T_0}(Y)$,
and the relation between the pullback
$j^*: H^*_T(X) \to H^*_T(X^T)$
and the non-weighted pullback $i^*: H^*_{T_0}(Y) \to H^*_{T_0}(Y^{T_0})$.
(We discuss the relationship between these pullbacks in the
 next section.)   Although \cite{AbeMatsumura2015} and \cite{AzamNazirQureshi2020} do not explicitly
 relate their GKM descriptions to invariant curves,
 the GKM descriptions of the cohomology, together with the fact that there are finitely many $1$-dimensional $T$-orbits,
 are enough (by \cite[Theorem 1.2.2]{GKM1998}) to describe the weights of the $T$-invariant curves up to scaling, 
   
  \end{remark}

\subsection{Comparison of restrictions to fixed points}\label{subsection: comparison of restrictions}
In this section we relate the weighted and non-weighted restrictions to fixed points.  In particular,
we prove Proposition \ref{p:fixed point isom}, which expresses the weighted restriction in terms of the non-weighted restriction.
The key diagram \eqref{equation: fixed points} was first considered in \cite[(4.6)]{AbeMatsumura2015}; see also \cite{AzamNazirQureshi2020} for \(G=GL_n\times\bC^\times\).

The sets of $T$-fixed points on $X$ and of $T_0$-fixed points of \(Y\)
can each be identified with $W^{P}$.
The maps
\begin{equation}
\begin{tikzcd}
{}
 &X\arrow[dl, bend right=20] 
 &W^P\arrow[l,swap,"j"]\\
pt
 &
 &\\
{}
 &Z\arrow[dl, bend right=20]\arrow[ul, bend left=20]\arrow[uu,"\pi"]\arrow[dd,swap,"\pi_0"]
 &\displaystyle\bigcup_{w\in W^P}\bC_{w\lambda}^\times\arrow[l,swap,"k"]\arrow[dd]\arrow[uu]\\
pt
 &
 &\\
{}
 &Y\arrow[ul, bend left=20]
 &W^P\arrow[l,swap,"i"]
\end{tikzcd}
\end{equation}
give the homomorphisms
\begin{equation}\label{equation: fixed points}
\begin{tikzcd}
{}
 &H^\ast_T(X)\arrow[r, "j^\ast"]\arrow[dd,swap,"\pi^\ast"]
 &H_T^\ast(W^P)\arrow[dd,swap, "\varphi"]\\
H_T^\ast\arrow[ur, bend left=20]\arrow[dr, bend right=20]
 &
 &\\
{}
 &H_H^\ast(Z)\arrow[r, "k^\ast"]
 &\displaystyle\bigoplus_{w\in W^P}H_H^\ast(\bC_{w\lambda}^\times)\\
H_{T_0}^\ast\arrow[ur, bend left=20]\arrow[dr, bend right=20]
 &
 &\\
{}
 &H_{T_0}^\ast(Y)\arrow[r, "i^\ast"]\arrow[uu,"\pi_0^\ast"]
 &H_{T_0}^\ast(W^P)\arrow[uu, "\psi"]
\end{tikzcd}
\end{equation}

We can write $\varphi = \oplus \varphi_w$
and $\psi = \oplus \psi_w$, where the sums are over $w \in W^P$. 
Let $H_{w \lambda}$ denote the kernel of the character $w \lambda$ of $H$, with
identity component $H_{w \lambda}^0$.  The inclusion $H_{w \lambda}^0 \hookrightarrow
H_{w \lambda}$ induces an isomorphism $H^*_{H_{w\lambda}} \to H^*_{H_{w\lambda}^0}$
which we use to identify these rings.
Since \(\bC_{w\lambda}^\times=H/H_{w\lambda}\) as \(H\)-spaces, 
$H_H^\ast(\bC_{w\lambda}^\times) \isom H^*_{H_{w\lambda}}$.
Write 
$$\varphi_w\,\colon H_T^\ast  \isom H_T^\ast(p_w) \to H_H^\ast(\bC_{w\lambda}^\times) \isom H^*_{H_{w\lambda}}$$
for the pullback induced via \eqref{equation: pimap} by the projection
from $\bC_{w\lambda}^\times$ to a point.  The map $\varphi_w$ is equal to the composition
\begin{equation} \label{e.compositioncoh}
  H_T^\ast \to H^*_H \to H^*_{H_{w\lambda}}  
\end{equation}
induced by the following maps of tori:
\begin{equation} \label{e.compositiontori}
H_{w\lambda}^0 \hookrightarrow H\to T.
\end{equation}
These maps of tori induce maps of symmetric algebras:
\begin{equation} \label{e.compositioncoh2}
S(\T^\ast) \to S(\H^\ast) \to S(\H_{w \lambda}^\ast),
\end{equation}
which coincide with \eqref{e.compositioncoh} under the identification of
the torus-equivariant cohomology of a point with the symmetric
algebra on the dual of the Lie algebra.  Similarly, replacing $T$ with $T_0$,
we see that $\psi_w$ can be identified as the composition
$H_{T_0}^\ast \to H^*_H \to H^*_{H_{w\lambda}}$ which coincides with the
corresponding composition of maps of symmetric algebras:
\begin{equation} \label{e.compositioncoh3}
S(\T_0^\ast) \to S(\H^\ast) \to S(\H_{w \lambda}^\ast),    
\end{equation} 

We can view $S(\T^*)$ and $S(\T_0^*)$ as the subsets of $S(\H^*)$ consisting
of polynomials which vanish on $\mathbb{S}$ and $\mathbb{S}_0$, respectively.  Under this
identification, we can describe the isomorphism \( \varphi_w^{-1}\circ \psi_w(\mu_0): S(\T_0^*) \to S(\T^*) \) 
as follows.

\begin{proposition} \label{p:fixed point isom}
Let \(w\in W^P\) and $\mu_0 \in \T^*_0$, and let
\(\mu = \varphi_w^{-1}\circ \psi_w(\mu_0) \).  Then $\mu = \overline{\mu_0}(w)$.
\end{proposition}

\begin{proof}
Since $\mu$ and $\mu_0$ agree when restricted to $\H_{w \lambda}$,
we have $\mu = \mu_0 + c \cdot w \lambda$ for some constant $c$.
Since $\mu$ vanishes on $\mathbb{S}$, we have 
$$
0 = \mu \cdot \chi = \mu_0 \cdot \chi + c (w \lambda \cdot \chi )
= a_{\mu_0} + c a_w.
$$
Hence $c = - a_{\mu_0}/a_w$.
\end{proof}

\begin{remark} \label{r:Schubert-fixed}
The element $i_x^*(\gd_{Y_w}^{T_0})$ can be calculated, using a formula due to Anderson-Jantzen-Soergel and Billey (see \cite{AndersenJantzenSoergel1994}, 
\cite{Billey99}), or a different formula due to Kumar \cite{Kumar1996}).  These formulas, together with
Proposition \ref{p:fixed point isom}, mean that we can calculate $j_x^*{\gd_{X_w}^T}$.  Indeed,
the commutativity of the diagram \eqref{equation: fixed points}, together with the definition of the weighted Schubert classes,
implies that $j_x^*{\gd_{X_w}^T} =   q_w \, \varphi_x^{-1}\circ \psi_x (i_x^*(\gd_{Y_w}^{T_0}))$.
\end{remark}

\section{Weighted Schubert calculus}\label{section: equivariant cohomology}
The basic problem of weighted Schubert calculus is to expand a product
in $H^*_T(X)$ as an $H^*_T$-linear combination of elements of the $H^*_T$-basis $\delta^T_{X_w}$.
This section contains our main results about weighted Schubert calculus, including a weighted
Chevalley formula and our main positivity theorem (Theorem~\ref{theorem: main theorem}).  We also
describe a Borel presentation of the equivariant cohomology, which is used in the proof of Proposition \ref{prop:weightedChevalley}.

\subsection{Structure constants and relation to non-weighted Schubert calculus} \label{ss.structure}
The relation between the weighted and non-weighted Schubert calculus can be described as follows.
Recall that $\pi_0^* \gd^{T_0}_{Y_w} = \gd^H_{Z_w}$ and $\pi^* \gd^T_{X_w} = q_w \gd^H_{Z_w}$
(Corollary \ref{corollary: pullback}).  The ring $H^*_H(Z)$ can be considered
as a module over either of the rings $H^*_{T_0}$ or $H^*_T$ (as well as $H^*_H$).  
On the other hand, $H^*_T(X)$ and $H^*_{T_0}(Y)$ acquire $H^*_H$-module structures
via the isomorphisms $\pi^*$ and $\pi_0^*$.  

Expanding products in $H^*_H(Z)$ in terms
of the basis $\gd^H_{Z_w}$,  with coefficients in $H^*_{T_0}$, is the problem of non-weighted
Schubert calculus.  Expanding products in $H^*_H(Z)$ in terms
of the basis $q_w \gd^H_{Z_w}$, with
coefficients in $H^*_T$, is the problem of weighted Schubert calculus.  The significant difference is that
the coefficients are in different rings: $H^*_{T_0}$ in the non-weighted case; $H^*_T$ in the weighted case.

There are various ways of presenting formulas for the structure constants.
Depending on which basis one uses, factors of $q_w$ (which can be computed by \eqref{e.qwdef})
may or may not occur in the formulas, but
these are easily dealt with.  Precisely, 
if we write
$$
 \gd^T_{X_u}  \gd^T_{X_v} = \sum_w c_{uv}^w  \gd^T_{X_w}
$$
and
$$
 \gd^H_{Z_u}  \gd^H_{Z_v} = \sum_w d_{uv}^w  \gd^H_{Z_w},
$$
where $c_{uv}^w$ and $d_{uv}^w$ are in $H^*_T$, we see that
\begin{equation} \label{e.structurechange}
c_{uv}^w = \frac{q_u q_v}{q_w} d_{uv}^w.
\end{equation}
Thus, we can easily obtain structure constants in one basis from structure constants in the other basis.
In particular, we can obtain formulas in the basis $\{ \gd^H_{Z_w} \}$ from formulas in 
the basis $\{ \gd^T_{X_w} \}$ by setting all the $q_u$ equal to $1$.
For convenience, we will often use the basis $\{ \gd^H_{Z_w} \}$.  
Similarly, we have introduced elements $\overline{\beta}(w)$ \eqref{e.weighted1}, which are related to the weighted roots $\beta(w)$ by a positive
scalar factor (see
\eqref{e.weighted2}).  This equation can be used to translate formulas from the $\overline{\beta}(w)$ to the $\beta(w)$, and conversely.  
Although the weighted roots are geometrically natural, we often use the $\overline{\beta}(w)$ for simplicity.

Finally, the formula of \eqref{e:positive3} can be used to translate formulas given in terms of $\overline{\beta}(v)$ into $\overline{\beta}(w)$
for $v, w \in W$.  In particular, taking $w = w_0$, we can obtain formulas in the parameters used in \cite{AbeMatsumura2015} for the weighted Grassmannian.

\subsection{Borel presentation} \label{subsection: Borel}
Since $H^*_T(X) \isom H^*_H(Z)$, the next proposition gives a Borel presentation of the
$T$-equivariant cohomology of a weighted flag variety.  Although the proof is simple, the result does not previously seem to have appeared.
Recall that $\H_{\lambda} = \ker \lambda \subset \H$; the dual space is $\H_{\lambda}^\ast$.

\begin{proposition} \label{proposition: Borel}
We have
\begin{equation}\label{equation: weighted Borel presentation}
H_H^\ast(Z)\isom S(\H^\ast)\underset{S(\H^\ast)^W}{\otimes}S(\H_{\lambda}^\ast)^{W_P}.
\end{equation}
\end{proposition}

\begin{proof}
 Recall that $Q = L_{\lambda}U \subset P$, where $L_{\lambda}$ is the kernel of the character of
 $\lambda$ of $L$.  Since $Z \isom G/Q$, applying \cite[Ch.~15, Prop.~6.5]{AndersonFulton2024} yields
$$
     H^*_H(Z) \isom H^*_H(G/Q) \isom H^*_H \otimes_{H^*_G} H^*_Q \isom S(\H^\ast)\underset{S(\H^\ast)^W}{\otimes}S(\H_{\lambda}^\ast)^{W_P},
$$
as desired.
\end{proof}

In terms of the Borel presentation, the left action of $r \in H^*_H = S(\H^*)$ on $H^*_H(Z)$ is given
by multiplication by $r \otimes 1$.  The right action of $S(\H^\ast_{\lambda})^{W_P}$ 
on $H^*_H(Z)$ can be
described as follows.  A representation $V$ of $H_{\lambda}$ may be viewed as an $H_{\lambda}$-equivariant
vector bundle over a point.  The $H_{\lambda}$-equivariant Chern class $c_i^{H_{\lambda}}(V)$ is an
element of $S(\H^\ast_{\lambda})$.  If $V$ is the restriction to $H_{\lambda}$ of a representation
of $L_{\lambda}$, then $c_i^H(V)$ is in $S(\H_{\lambda}^\ast)^{W_P}$.  The
element $1 \otimes c_i^H(V) \in H^*_H(Z)$ is the $i$-th equivariant Chern class of the
$H$-equivariant line bundle $G \times^Q V$, where as usual the $L_{\lambda}$-module structure
on $V$ is extended to a $Q = L_{\lambda}U$-module structure by defining $U$ to act trivially.

Given $\mu \in \H^*$, let $\tilde{\mu}$ denote the image of $\mu$ under the map
$\H^* \to \H_{\lambda}^*$.  Suppose $\mu \in \Lambda$.  Then $\tilde{\mu}$ is $W_P$-invariant
if and only if $\mu$ is.  In this case, we obtain $1$-dimensional representations 
$\bC_{\tilde{\mu}}$ of $L_{\lambda}$ and $\bC_{\mu}$ of $L$.
As usual, we extend these representations to $Q$ and $P$, respectively, by defining $U$ to act trivially.
   The line bundle $G \times^Q \bC_{\tilde{\mu}}$ on
    $Z = G/Q$ is the pullback via $\pi_0^*$ of the line bundle
    $G \times^P \bC_{\mu}$.  In terms of the Borel presentation,
    $c_1^H(G \times^Q \bC_{\tilde{\mu}})= 1 \otimes \tilde{\mu}$.
    Note that
    $\tilde{\lambda} = 0$, reflecting the fact that the pullback of $G \times^P \bC_{\lambda}$
    is $H$-equivariantly trivial on $Z$.

\subsection{Module structure and a Chevalley formula} \label{section: weighted chevalley formula}
As is apparent from the Borel presentation, $H^*_T(X) \cong H^*_H(Z)$ has the structure of a $S(\H^*) \otimes S(\H_{\gl}^*)$-module.  In this section we describe the module structure in terms of the weighted Schubert basis
(Proposition \ref{proposition: weighted chevalley formula I}).  As a consequence, we obtain a Chevalley formula for the multiplication of a weighted Schubert class in terms of a weighted Schubert divisor (Proposition \ref{prop:weightedChevalley}).   For weighted flag varieties of cominuscule type, the formula can be simplified; see
Corollary \ref{c:cominuscule}, which in the case of the weighted Grassmannian recovers the formula of \cite[Prop.~5.2]{AbeMatsumura2015}
The proofs of these propositions are deduced from results for non-weighted flag varieties, as is Lemma \ref{lemma: lambdamultiply}, which
also plays a key role
in the proof of our positivity theorem.

Via the pullback isomorphisms $\pi^*: H^*_T(X) \to H^*_H(Z)$ and $\pi_0^*: H^*_{T_0}(Y) \to H^*_H(Z)$, the ring
$H^*_H(Z)$ acquires the structure of a module for $H^*_T$ and for $H^*_{T_0}$.  
Lemma \ref{lemma: lambdamultiply} illustrates
that although the Schubert classes
form a basis of $H^*_H(Z)$ as a module over $H^*_T$ or over $H^*_{T_0}$, they
are not linearly independent over $H^*_H$.

\begin{lemma} \label{lemma: lambdamultiply}
   If \(v\in W^P\), then
    \begin{equation}\label{equation: another unnamed and referenced}
(v\lambda) \,\delta_{Z_v}^H=\sum_{\substack{w= v r_{\gamma}   \\ w {\lessdot}_P v}}(\lambda\cdot\gamma^\vee)\,\delta_{Z_w}^H,
\end{equation}
where the sum is over positive roots $\gg$ such that $v r_{\gg} \lessdot_P v$.
\end{lemma}

\begin{proof}
For $\mu \in \H^*$, we have, using $H$-equivariant classes on $Y$ (cf.~Remark \ref{r: changegroups}):
\begin{equation}\label{equation: non-weighted Chevalley formula}
c_1^H(G\times^P\bC_\mu)\delta_{Y_v}^H=v\mu \, \delta_{Y_v}^H-\sum_{\substack{w= v r_{\gamma}   \\ w {\lessdot}_P v}}(\mu\cdot\gamma^\vee)\delta_{Y_w}^H,
\end{equation}
where again the sum is over positive $\gg$.  See, for example, \cite{Brion1997} or \cite{AndersonFulton2024}.  (In accord with our convention that $\fb$ contains the positive root spaces, the sign in front of the sum is negative: the line bundle $G\times^P\bC_\mu$ is positive if $\mu$ is antidominant.)  Setting $\mu = \lambda$ and
pulling back along \(\pi_0\,\colon Z\to Y\) gives \eqref{equation: another unnamed and referenced}, since \(c_1(G\times^Q\bC_{\tilde\lambda})=0\), where \(\tilde\lambda\) is the restriction of \(\lambda\) to \(Q\).
\end{proof}

The ring $H^*_T(X)$ is a free $H^*_T$-module with a basis given by
the $\delta_{X_w}^T$.  For $\mu \in \H^* $, the next proposition gives expansions for
$\mu \cdot \delta_{X_v}^T$ and $(1 \otimes \mu) \cdot \delta_{X_v}^T$ in this basis,
with coefficients in $H^*_T = S(\T^*)$.  Recall
that for $\eta \in \H^*$ and $w \in W$, the element
$\overline{\eta}(w) = \eta - \frac{a_{\eta}}{a_w} w \lambda$
defined in  \eqref{e.weighted1} lies in the subspace $\T^*$ of $\H^*$.

\begin{proposition} \label{proposition: weighted chevalley formula I}
For every \(\mu\in \H^\ast\) and for every \(v\in W^P\), we have
\begin{equation} \label{equation: weighted chevalley formula I}
\mu\,\delta_{X_v}^T  = \overline{\mu}(v) \delta_{X_v}^T+
\dfrac{a_\mu}{a_v} \sum_{\substack{w= v r_{\gamma}   \\ w {\lessdot}_P v}} \dfrac{q_v}{q_w}(\lambda\cdot \gamma^\vee)\,\delta_{X_w}^T ,
\end{equation}
and
\begin{equation}\label{equation: weighted chevalley formula Ia}
(1\otimes\tilde\mu)\delta_{X_v}^T =\overline{(v \mu)}(v) \delta_{X_v}^T+\sum_{\substack{w= v r_{\gamma}   \\ w {\lessdot}_P v}} \frac{q_v}{q_w} \left((\dfrac{a_{v\mu}}{a_v}\lambda - \mu )\cdot \gamma^\vee\right)\delta_{X_w}^T.
\end{equation}
In both formulas, the sum is over positive roots $\gg$ such that $v r_{\gamma}  {\lessdot}_P v$.
\end{proposition}

\begin{proof}
By Lemma \ref{lemma: lambdamultiply}
and the relation $\pi^\ast\delta_{X_w}^T = q_w \delta_{Z_w}^H$, we have
\begin{equation}\label{equation: key relation}
\begin{split}
\mu\,\delta_{X_v}^T&=(\mu-\dfrac{a_\mu}{a_v}v\lambda)\delta_{X_v}^T+\dfrac{a_\mu}{a_v}v\lambda\,\delta_{X_v}^T\\
&= \overline{\mu}(v) \delta_{X_v}^T+
\dfrac{a_\mu}{a_v} \sum_{\substack{w=vr_\gamma \\  w {\lessdot}_P v }}\dfrac{q_v}{q_w}(\lambda\cdot\gamma^\vee)\,\delta_{X_w}^T,
\end{split}
\end{equation}
proving \eqref{equation: weighted chevalley formula I}.  

We now prove 
\eqref{equation: weighted chevalley formula Ia}.  By the non-weighted Chevalley formula \eqref{equation: non-weighted Chevalley formula}, we have
\begin{equation}
\begin{split}
(1 \otimes \tilde{\mu}) \gd_{Z_v}^H &= v \mu \gd_{Z_v}^H - \sum_{\substack{w= v r_{\gamma}   \\ w {\lessdot}_P v}}(\mu\cdot\gamma^\vee)\delta_{Z_w}^H \\
& = \left(v\mu-\dfrac{a_{v\mu}}{a_v}v\lambda\right) \gd_{Z_v}^H + \dfrac{a_{v\mu}}{a_v}v\lambda \, \gd_{Z_v}^H - \sum_{\substack{w= v r_{\gamma}   \\ w {\lessdot}_P v}}(\mu\cdot\gamma^\vee)\delta_{Z_w}^H.
\end{split}
\end{equation}
Applying Lemma \ref{lemma: lambdamultiply} to the middle term and simplifying, we obtain
$$
(1 \otimes \tilde{\mu}) \gd_{Z_v}^H = \overline{(v \mu)}(v)  \gd_{Z_v}^H + \sum_{\substack{w= v r_{\gamma} }}\left((\dfrac{a_{v\mu}}{a_v}\lambda - \mu )\cdot \gamma^\vee\right) \delta_{Z_w}^H.
$$
Equation \eqref{equation: weighted chevalley formula Ia} follows from this formula by again using $\pi^\ast\delta_{X_w}^T = q_w \delta_{Z_w}^H$ (and similarly with $v$ in place of $w$).
\end{proof}

Note that if $\mu = \lambda$, then both sides of \eqref{equation: weighted chevalley formula Ia} are zero.

The next proposition gives a formula for the product of a weighted
Schubert divisor with a weighted Schubert class.  In this formula and elsewhere, we can take
$\omega_{\ga}$ to be any element of $\H^*$ satisfying 
\begin{equation} \label{e:omega} 
\omega_{\ga} \cdot \beta^{\vee} = \delta_{\ga, \gb}
\end{equation}
for positive simple roots $\ga, \gb$.  This condition does not uniquely determine $\omega_{\ga}$ since
$G$ is not semisimple.  For example, if $G = \C^* \times G_1$, 
then $\omega_{\ga}$ could be any element of the form $x_0 + \omega_{1,\alpha}$, where
$\omega_{1, \alpha}$ is defined in terms of $G_1$, and $x_0$ is a character of the
first factor $\C^*$.  (This occurs in the examples in
Sections \ref{ss:Grassmannian} and \ref{ss:Lagrangian}.)  However,
if $\omega_{\ga}'$ were another
weight satisfying \eqref{e:omega}, then $\omega_{\ga} - \omega_{\ga}'$ is $W$-invariant, so
$w_0 \omega_{\ga} - v \omega_{\ga} = w_0 \omega'_{\ga} - v \omega'_{\ga}$.  Thus, the formula 
of Proposition \ref{prop:weightedChevalley} is unchanged.

\begin{proposition} \label{prop:weightedChevalley}
Let \(\alpha\) be a simple root with $r_{\alpha} \not\in W_P$, and let \(u_\alpha=w_0r_\alpha\in W^P\).
For every \(v\in W^P\), we have
\begin{equation}\label{equation: weighted chevalley formula II}
\delta_{X_{u_\alpha}}^T\,\delta_{X_v}^T=
q_{u_\alpha} \overline{(w_0\omega_\alpha-v\omega_\alpha)}(v) \delta_{X_v}^T+
\sum_{\substack{w= v r_{\gamma}   \\ w {\lessdot}_P v}} \dfrac{q_{u_\alpha}q_v}{q_w}\left((\dfrac{a_{w_0\omega_\alpha-v\omega_\alpha}}{a_v} \lambda
+ \omega_\alpha)\cdot \gamma^\vee\right)\,\delta_{X_w}^T,
\end{equation}
where the sum is over positive roots $\gg$ such that $v r_{\gamma}  {\lessdot}_P v$.
\end{proposition}

\begin{proof}
Since $\delta_{Z_{u_\alpha}}^H = \pi_{0,H}^* \delta_{Y_{u_\alpha}}^{H}$, applying $\pi_{0,H}^*$ to the formula
of \cite[\S16 Lemma 2.6]{AndersonFulton2024} yields
$\delta_{Z_{u_\alpha}}^H=(w_0\omega_\alpha\otimes 1)-(1\otimes\tilde\omega_\alpha)$,
where \(\tilde\omega_\alpha\) is the restriction to \(\H_{\lambda}^\ast\) of $\omega_{\alpha}$.  Note that the right hand side 
of this formula does not depend on the choice of $\omega_{\alpha}$ satisfying \eqref{e:omega}.
Therefore, since $\pi^*\delta_{X_{u_\alpha}}^T = q_{u_\alpha} \delta_{Z_{u_\alpha}}^H$, we have 
$$
\delta_{X_{u_\alpha}}^T=  q_{u_\ga } ( (w_0\omega_\alpha\otimes 1)-(1\otimes\tilde\omega_\alpha) ).
$$
Hence
\begin{equation}
\begin{split}
\delta_{X_{u_\alpha}}^T\,\delta_{X_v}^T &= q_{u_\ga } ( (w_0\omega_\alpha\otimes 1)-(1\otimes\tilde\omega_\alpha) ) \delta_{X_v}^T\\
& = q_{u_\ga } \overline{w_0 \omega_\alpha}(v) \delta_{X_v}^T + \dfrac{a_{w_0\omega_\alpha}}{a_v} 
\sum_{\substack{w= v r_{\gamma}   \\ w {\lessdot}_P v}} \frac{q_{u_\ga } q_v}{q_w} (\lambda \cdot \gamma^{\vee}) \gd_{X_w}^T \\
& \hspace{.2in} - q_{u_\ga } \overline{v \omega_\alpha}(v)  \delta_{X_v}^T
- \sum_{\substack{w= v r_{\gamma}   \\ w {\lessdot}_P v}} \frac{q_{u_\ga } q_v}{q_w} ( \dfrac{a_{v \omega_\alpha}}{a_v} \lambda - \omega_\alpha) \cdot \gamma^{\vee} ) \gd_{X_w}^T.
\end{split}
\end{equation}
This implies \eqref{equation: weighted chevalley formula II}. 
\end{proof}

\begin{remark} \label{r:Chev-pos}
In the special case $u = u_{\alpha}$, Proposition \ref{prop:weightedChevalley}
implies directly that
the $c_{uv}^w$ are nonnegative at $v$.  Equation \eqref{equation: weighted chevalley formula II} 
implies that the structure constants $c_{u_{\alpha}v}^w$ can be nonzero only for 
$w = v r_{\gamma}  {\lessdot}_P v$ (in which case the structure constant is a scalar), or for $w = v$. 
The nonnegativity of the scalar $c_{u_{\alpha}v}^{v r_{\gamma}}$ can be deduced as follows.  We have
$w_0\omega_\alpha-v\omega_\alpha = \sum b_i \beta_i$, where the $\beta_i$ are the negative
simple roots and $b_i \geq 0$.  Hence, since $\chi$ is antidominant,
$a_{w_0\omega_\alpha-v\omega_\alpha} \geq 0$.  Also, $a_v >0$ by hypothesis, and
since $\gamma^{\vee}$ is a positive coroot, both $\lambda \cdot \gamma^{\vee}$ and $\omega_{\ga} \cdot \gamma^{\vee}$ are
non-negative.  These observations and \eqref{equation: weighted chevalley formula II} imply that $c_{u_{\alpha}v}^{v r_{\gamma}} \geq 0$.
The $c_{u_{\alpha} v}^v$ are nonnegative at $v$, since $\overline{(w_0\omega_\alpha-v\omega_\alpha)}(v) = \sum b_i \overline{\beta_i}(v)$.
\end{remark}

Sections \ref{ss:Grassmannian} (on the weighted Grassmannian, considered by \cite{AbeMatsumura2015}) 
and \ref{ss:Lagrangian} consider examples of weighted flag
varieties for which $\lambda = \omega_{\alpha}$ (that is, $\lambda$ satisfies the condition \eqref{e:omega}).
In this case,
multiplication by a weighted Schubert divisor can be put in the following form.   

\begin{corollary} \label{c:cominuscule}
Keep the hypotheses of Proposition \ref{prop:weightedChevalley}, and assume $\lambda = \omega_{\ga}$.
Then
$$
\delta_{X_{u_\alpha}}^T\,\delta_{X_v}^T  =
q_{u_{\ga}} \overline{w_0 \omega_{\alpha}}(v) \delta_{X_v}^T+
\sum_{\substack{w= v r_{\gamma}   \\ w {\lessdot}_P v}} \dfrac{q_{u_\alpha}q_v}{q_w} \dfrac{a_{w_0}}{a_v} (\omega_\alpha \cdot \gamma^\vee )\,\delta_{X_w}^T .
$$
\end{corollary}

\begin{proof}
Our hypothesis on $\lambda$ implies that
$$
\overline{(w_0\omega_\alpha-v\omega_\alpha)}(v) = \overline{(w_0 \lambda - v \lambda)}(v) = \overline{w_0 \lambda}(v) = \overline{w_0\omega_\alpha}(v).
$$
Also, $\lambda \cdot \gamma^{\vee} = \omega_{\alpha} \cdot \gamma^{\vee}$, and
$$
\dfrac{a_{w_0 \omega_{\ga} - v \omega_{\ga}}}{a_v} = \dfrac{a_{w_0 \lambda - v \lambda}}{a_v} = \dfrac{a_{w_0} - a_v}{a_v} =  \dfrac{a_{w_0}}{a_v} - 1.
$$
The corollary follows from these observations
and Proposition \ref{proposition: weighted chevalley formula I}.  
\end{proof}

In the case of the weighted Grassmannian, this corollary appears as \cite[Prop.~5.2]{AbeMatsumura2015} (although some translation is required
to convert the formula from that paper into the formula of the above corollary).  Note that the Chevalley formula given for type $A$ in
\cite[Theorem 5.1]{AzamNazirQureshi2020} is not correct (it does not specialize to the formula of \cite{AbeMatsumura2015}).

\begin{remark} \label{r:cominusculew0}
In \cite{AbeMatsumura2015}, the parameters used are $\overline{\beta}(w_0)$ for negative simple roots $\beta$.  In terms of these
parameters, the formula of Corollary \ref{c:cominuscule} can be rewritten as follows.  
Define nonnegative integers
$e_i$ by the equation $w_0 \lambda - v \lambda = \sum_i e_i \beta_i$.
For simplicity write $\overline{\beta}_i = \overline{\beta}_i(w_0)$.  We have $\lambda = \omega_{\alpha}$, and Lemma \ref{lemma: positive} implies that 
$\overline{w_0 \lambda}(v) = \frac{a_{w_0}}{a_v} \sum e_i \overline{\beta}_i$.
Hence
$$
\delta_{X_{u_\alpha}}^T\,\delta_{X_v}^T  = q_{u_{\ga}} \frac{a_{w_0}}{a_v} \Big( (\sum_i e_i \overline{\beta}_i) \delta_{X_v}^T+ \sum_{\substack{w= v r_{\gamma}   \\ w {\lessdot}_P v}} \dfrac{q_v}{q_w} (\omega_\alpha \cdot \gamma^\vee )\,\delta_{X_w}^T \Big).
$$
\end{remark}

\subsection{Positivity}
In this section we prove our general positivity theorem.
We first record a lemma about the (non-weighted) equivariant structure
constants.   The first statement follows from the type of support argument in \cite[Section 19.1]{AndersonFulton2024}), and it
seems likely that the second statement is known as well, but for lack of reference, 
we provide a proof.

\begin{lemma} \label{l:nonweighted}
Suppose $\delta_{Y_u}^{T_0}\delta_{Y_v}^{T_0}=\sum_{w\in W^P}b_{uv}^w\delta_{Y_w}^{T_0}$, with
$b_{uv}^w \in H^*_{T_0}$.  If $b_{uv}^w \neq 0$, then $w \leq u$ and $w \leq v$.  Conversely,
if $w$ is a maximal element among the elements in $W^P$ that are less than or equal to both $u$ and $v$,
then $b_{uv}^w \neq 0$.
\end{lemma}

\begin{proof}
Let $x \in W^P$.  Then $i_x^* \gd^{T_0}_{Y_w}$ is nonzero if and only if $x \leq w$.  Indeed, if $x \not\leq w$, then $x P_0 \not\in Y_w$
so $i_x^* \gd^{T_0}_{Y_w} = 0$; conversely, if $x \leq w$, then $x P_0 \in Y_w$ and there is an explicit formula for the pullback 
$i_x^* \gd^{T_0}_{Y_w}$ showing that it is nonzero. (The formula is due to Anderson-Jantzen-Soergel \cite{AndersenJantzenSoergel1994} and Billey
\cite{Billey99}; for a discussion of conventions see \cite[Appendix B]{GrahamKreiman2015}.)  Hence, $i_x^* (\delta_{Y_u}^{T_0}\delta_{Y_v}^{T_0})$
is nonzero if and only if $x \leq u$ and $x \leq v$.  

Suppose that $b_{uv}^w \neq 0$.  We wish to show that $w \leq u$ and $w \leq v$.  We may assume that
$w \in W^P$ is a maximal element among the $y \in W^P$ with $b_{uv}^y \neq 0$.  In this case,
$$
i_w^* ( \delta_{Y_u}^{T_0}\delta_{Y_v}^{T_0} ) = i_w^* (\sum_{y\in W^P}b_{uv}^y\delta_{Y_w}^{T_0}) = b_{uv}^w i_w^* \delta_{Y_w}^{T_0}.
$$
Since the right hand side is nonzero, so is the left hand side, so $w \leq u$ and $w \leq v$.  

Conversely,   
if $w$ is maximal among the elements in $W^P$ that are less than or equal to both $u$ and $v$, 
then 
$$
i_w^* (\sum_{y\in W^P}b_{uv}^y\delta_{Y_w}^{T_0}) = b_{uv}^w i_w^* \delta_{Y_w}^{T_0}  = i_w^* (\delta_{Y_u}^{T_0}\delta_{Y_v}^{T_0}) \neq 0,
$$
where the first equality follows from the preceding paragraph.
We conclude that $b_{uv}^w \neq 0$.
\end{proof}

\begin{proof}[Proof of Theorem~\ref{theorem: main theorem}]
By the discussion in Section \ref{ss.structure}, it suffices to consider structure constants in terms of
the elements $\delta_{Z_u}^{H}$ and to work with the $\overline{\beta}_i(x)$ since these are positive
multiples of the weighted roots $\beta_i(x)$.  Therefore, it suffices to show that
\begin{equation} \label{e:Z-constants}
\delta_{Z_u}^{H}\delta_{Z_v}^{H}=\sum_{w\in W^P}d_{uv}^w\delta_{Z_w}^{H},
\end{equation}
such that each \(d_{uv}^w\) is a nonnegative linear combination of  $\overline{\nu}_1(x_1) \cdots \overline{\nu}_k(x_k)$,
where the $\nu_i$ are distinct negative roots, and $x_i \in \cS(u,v;w)$.  

By positivity for the non-weighted structure constants, we can write
\begin{equation}
\delta_{Z_u}^{H}\delta_{Z_v}^{H}=\sum_{w\in W^P} \sum_{I}b(u,v,w;I)\beta_I \delta_{Z_w}^{H},
\end{equation}
where \(I\) is a set of negative roots (not necessarily simple), \(\beta_I\) is the product of the roots in $I$ (in particular,
square-free), and \(b(u,v,w;I)\) is a nonnegative integer.  See \cite{Graham2001} and
\cite[Cor.~19.4.7]{AndersonFulton2024}.
By Lemma \ref{l:nonweighted}, the only nonzero terms in the sum occur for $w $ such that $w \leq u$ and $w \leq v$.

To prove the theorem, it suffices to prove the following assertion: For all $w \in W^P$, we have 
$\beta_I \delta_{Z_w}^{H} = \sum_{y \leq_P w} f_{y,I} \delta_{Z_y}^{H}$, where
each $f_{y,I}$ is a nonnegative linear combination of monomials
$\overline{\nu}_1(x_1)  \cdots \overline{\nu}_k(x_k)$, where the $\nu_1, \ldots, \nu_k$ are
distinct elements of $I$, and $x_i \in \cS(u,v;y)$.  We prove the assertion by induction on $| I |$.
If $| I | = 1$, then $\beta_I = \beta$ for some negative root $\beta$.  Since $\beta = \overline{\beta}(w) + \frac{a_{\beta}}{a_w} w \lambda$,
applying Lemma \ref{lemma: lambdamultiply} yields
$$
\beta \delta_{Z_w}^{H} = \overline{\beta}(w) \delta_{Z_w}^{H} + \frac{a_{\beta}}{a_w} \sum_{\substack{y= w r_{\gamma}   \\ y {\lessdot}_P w}}(\lambda\cdot\gamma^\vee)\,\delta_{Z_y}^H.
$$
which is of the desired form.  Suppose now that $I = I' \sqcup \{ \beta \}$ (disjoint union), and suppose that the assertion holds for $I'$.  
It suffices to show that
\begin{equation} \label{e:f-equation}
f_{y,I} = \overline{\beta}(y) f_{y, I'} + \sum_{\substack{z = y r_{\gamma}   \\ z {\gtrdot}_P y}}  \frac{a_{\gb}}{a_z}  (\lambda \cdot \gamma^{\vee}) f_{z, I'}.
\end{equation}
The reason is that since $\beta \not\in I'$, each term in the product $\overline{\beta}(y) f_{y, I'}$ is of the desired form, and
for $z$ in the sum, $\cS(u,v;z) \subset \cS(u,v;y)$.
We have
\begin{align*}
\gb_I \delta_{Z_w}^{H} & = \sum_y f_{y, I'}  \beta \, \delta_{Z_y}^H = \sum_y f_{y, I'} \Big( \overline{\beta}(y) \delta_{Z_y}^H + \frac{a_{\gb}}{a_y} 
\sum_{\substack{z= y r_{\gamma}   \\ z {\lessdot}_P y}}  (\lambda \cdot \gamma^{\vee}) \delta_{Z_z}^H \Big) \\
& = \sum_y f_{y, I'} \overline{\beta}(y) \delta_{Z_y}^H + \sum_y  \frac{a_{\gb}}{a_y} \sum_{\substack{z= y r_{\gamma}   \\ z {\lessdot}_P y}}  (\lambda \cdot \gamma^{\vee})  f_{y, I'}  \delta_{Z_z}^H.
\end{align*}
If we reindex the second sum by switching the roles of $y$ and $z$, we find
\begin{align*}
\gb_I \delta_{Z_w}^{H} & = \sum_y \overline{\beta}(y) f_{y, I'} \delta_{Z_y}^H + \sum_z  \frac{a_{\gb}}{a_z} \sum_{\substack{y= z r_{\gamma}   \\ y {\lessdot}_P z}}  (\lambda \cdot \gamma^{\vee})  f_{z, I'}  \delta_{Z_y}^H \\
&= \sum_y \Big( \overline{\beta}(y) f_{y, I'} + \sum_{\substack{z= y r_{\gamma}   \\ z {\gtrdot}_P y}}  \frac{a_{\gb}}{a_z}  (\lambda \cdot \gamma^{\vee})  f_{z, I'} \Big) \delta_{Z_y}^H.
\end{align*}
This proves \eqref{e:f-equation}; Theorem \ref{theorem: main theorem} follows.
\end{proof}

\begin{remark} \label{r:non-squarefree}
Example \ref{e:nonexpress} shows that for a fixed $x$, it may not be possible to obtain a square-free expansion of the $d_{uv}^w$ in terms of square-free monomials
in negative roots at $x$.  For example, this is the case for the coefficient of $\gd_{Z_2}^H$ in $(\gd_{Z_2}^H)^2$, when expanded at $v_3$.  The corresponding
non-weighted
expansion is square-free because $\gb_2+ \gb_3$ is a root.  However, the weighted expansion at $v_3$ is not square-free:
$ \frac{a_4}{a_2}  \overline{\gb}_2(v_3) +  \overline{\gb}_3(v_3) $ is not a multiple of a weighted root, so the expansion of the coefficient in terms of monomials
needs to be $ \frac{a_3 a_4}{a_1 a_2}  \overline{\gb}_2(v_3)^2+   \frac{a_4}{a_2}  \overline{\gb}_2(v_3)  \overline{\gb}_3(v_3) $.
\end{remark}

\begin{remark} \label{r:weighted-occur}
Example \ref{e:nonexpress} shows that $d_{uv}^w$ need not be nonnegative at $w$.
 The reason can be seen in the last displayed line of the proof: $f_{y,I}'$ contributes to the coefficient of $\delta_{Z_z}^H$, but $z >y$, so the weighted negative
 roots at $z$ are not
generally nonnegative linear combinations of weighted roots at $y$.   
\end{remark}

\begin{remark} \label{r:antidominant}
As noted in the introduction, Abe and Matsumura \cite[Remark 5.8]{AbeMatsumura2015} observed that $\chi$ is antidominant in the positivity statement can be removed by
using a different Schubert basis.  Indeed, if $\chi$ is fixed, one
can realize the variety $Z$ with respect to a choice of Borel subgroup $B' = w B w^{-1}$ (for $w \in W$)
with respect to which 
$\chi$ is antidominant.  The groups $P$ and $Q$ would be replaced by $P' = w P w^{-1}$ and $Q' = w Q w^{-1}$
and the character $\lambda$ by $\lambda' = w \lambda$; if $\lambda$ is dominant with respect to $B$, then 
$\lambda'$ is dominant with respect to $B'$.  Then $Z = G/Q \isom G/Q'$, and the positivity theorem applies with respect to the
realization of $Z$ as $G/Q' = G \times^{P'} \bC_{\lambda'}$.
\end{remark}

As noted in the introduction, the method of proof of Theorem \ref{theorem: main theorem} yields the following positivity result
about classes of subvarieties in weighted flag varieties.  In this result, in contrast to Theorem \ref{theorem: main theorem}, 
the positivity is in terms of the positive weighted roots, and we need to assume that $\chi$ is dominant.  The assumption that the $a_w$ are positive remains in force.

\begin{theorem}\label{t:expansion}
Suppose $M$ is a $T$-invariant subvariety of the weighted flag variety $X$.  Then
$$
\gd^T_M = \sum c_w \gd^T_{X_w},
$$
where each $c_w$ is a nonnegative linear combination of products of the form
$\mu_1(x_1) \cdots \mu_k(x_k)$.  Here the $\mu_i$ are distinct positive
roots, and $x_i \in W^P$.
\end{theorem}

\begin{proof}
Let $\tilde{M} = \pi^{-1}(M)$.  As in the proof of Theorem~\ref{theorem: main theorem}, it suffices to show that
$$
\gd_{\tilde{M}}^H = \sum d_w \gd^H_{Z_w},
$$
where each $d_w$  is a  nonnegative linear combination of monomials of the form
as $\overline{\mu}_1(x_1) \cdots \overline{\mu}_k(x_k)$, where the $\mu_1, \ldots, \mu_k$ are
distinct positive roots and $x_i \in W^P$.

 By \cite[Theorem 3.2]{Graham2001} or \cite[Theorem 19.3.1]{AndersonFulton2024},
we can write
$$
\gd_{\tilde{M}}^H = \sum_{I} a(w;I) \alpha_I \delta_{Z_w}^{H},
$$
where $I$ is a set of positive roots, $\ga_I$ is the product of the elements of $I$, and $a(w,I)$ is a nonnegative integer.
(Square-freeness is not stated in either reference, but it follows from the proof of \cite[Theorem 3.2]{Graham2001}, 
cf.~the discussion preceding \cite[Cor.~19.4.7]{AndersonFulton2024}.)
It suffices to prove that for all $w \in W^P$, we have 
$\ga_I \delta_{Z_w}^{H} = \sum_{y \leq_P w} f_{y,I} \delta_{Z_y}^{H}$, where
each $f_{y,I}$ is a  nonnegative linear combination of monomials of the form
$\overline{\mu}_1(x_1)  \cdots \overline{\mu}_k(x_k)$, where the $\mu_1, \ldots, \mu_k$ are
distinct elements of $I$, and $x_i \in W^P$.
This is proved in the same way as the analogous assertion in the proof of Theorem~\ref{theorem: main theorem},
using the equation
$$
\alpha \delta_{Z_w}^{H} = \overline{\alpha}(w) \delta_{Z_w}^{H} + \frac{a_{\alpha}}{a_w} \sum_{\substack{y= w r_{\gamma}   \\ y {\lessdot}_P w}}(\lambda\cdot\gamma^\vee)\,\delta_{Z_y}^H.
$$
Since $\chi$ is dominant, each $a_{\ga} \geq 0$, and by hypothesis, each $a_w >0$.  The remainder of the proof is similar to the argument
given in Theorem~\ref{theorem: main theorem}; we omit further details.
\end{proof}

\section{Examples}\label{section: examples}

\subsection{Weighted projective space} \label{ss.WPS}
We describe from our perspective weighted projective space (cf.\ \cite{Kawasaki1973}, \cite[Example 5.11]{AbeMatsumura2015}), realized as a quotient
of the group $GL_{m+1}$.  Let \(G=GL_{m+1}\) and let \(H\) denote the Cartan subgroup of diagonal matrices.  Let $\{ y_0, \ldots, y_m \}$ be the standard
basis of $\fh$, and $\{ x_0, \ldots, x_m \}$ the dual basis of $\fh^*$.  
The weight lattice $\Lambda$ is isomorphic to $\Z^{m+1}$; we sometimes denote the element $\sum c_i x_i$ by $(c_0, \ldots, c_m)$.
The set of dominant weights is
$$
\Lambda^+=\set{\lambda = (c_0, \ldots, c_m) \in \Lambda \mid c_0 \geq\cdots\geq c_m} \}.
$$

Let \(\lambda= x_0 = (1,0,\ldots,0)\), and let \(V=\bC^{m+1}\) denote the standard representation of $G$, which has highest weight $\lambda$.
Observe that $\lambda = \omega_{\alpha}$, where $\alpha$ is the simple root $x_0-x_1$.
Identify the highest weight space of $V$ with $\C_{\gl}$, and let $P$ be the subgroup of $G$ preserving this space.
There is an isomorphism
\begin{equation}\label{equation: mu WPS}
\mu\,\colon Z = G \times^P \C^\times_{\gl} \to V^\times,\quad [g,v]\mapsto gv,
\end{equation}
where \(V^\times=V\smallsetminus\set{0}\).
In what follows, we identify \(Z\) with \(V^\times\) via \eqref{equation: mu WPS}.

We view the Weyl group $W \cong S_{m+1}$ as the set of permutations of the set $\set{0,\ldots,m} $.  Then $W_P$ is identified with the subgroup
$\{ 0 \} \times S_m$.  Define $v_k \in W$ to be the element given in $1$-line notation as \(v_k=(k, m, m-1, \ldots, k+1, ,k-1,\ldots, 0)\).  The
set of maximal coset representatives is
$W^P = \{ v_0, v_1, \ldots, v_m \}$, which we can identify with the set $\set{0,\ldots,m} $ via $k \mapsto v_k$.  The longest element of
$W$ is $w_0 = v_m$.
Let $\chi=(a_0,\ldots,a_m) $ be an antidominant cocharacter, so \(a_0\leq a_1\leq\cdots\leq a_m\).
We have $a_{v_k} = a_k$, and the condition  \eqref{equation: a_w} implies that \(a_k>0\) for \(0\leq k\leq m\).
We denote the corresponding weighted flag variety $X = S\backslash Z$ (which in this case is weighted projective space)
by $\bP(a_0,\ldots,a_m)$.  Write $Z_k = Z_{v_k}$ and $X_k = X_{v_k}$.

The subspace $\H_{\lambda}$ of $\H$ is defined by the equation $x_0 = 0$.  The 
restriction map \( S(\H^\ast)^W \to S(\H_{\lambda}^\ast)^{W_\lambda}\) is given by $e_i(x_0, \ldots, x_m) \mapsto e_i(x_1, \ldots, x_m)$.
This map is surjective, so the following lemma applies.

\begin{lemma}\label{lemma: surjection}
Suppose the restriction map \(S(\H^\ast)^W\to S(\H_{\lambda}^\ast)^{W_\lambda}\) is surjective, with kernel $I$, so
\begin{equation}\label{equation: surjection}
S(\H_{\lambda}^\ast)^{W_\lambda} \isom S(\H^\ast)^W/I.
\end{equation}
Then
\begin{equation}\label{equation: surjective case}
H_T^\ast(X)\isom S(\H^\ast)/IS(\H^\ast).
\end{equation}
\end{lemma}

\begin{proof}
This follows from the Borel presentation of $H_H^\ast(Z) \cong H_T^\ast(X)$ (Proposition \ref{proposition: Borel}
\end{proof}

The lemma implies that for weighted projective space, we have
\begin{equation}
H_T^\ast(\bP(a_0,\ldots,a_m))\isom\F[x_0,\ldots,x_m]/ \langle x_0\cdots x_m \rangle.
\end{equation}

For each $k \in \{0, 1, \ldots m-1 \}$, we have
\begin{equation} \label{e.zk}
\delta_{Z_k}^H=x_{k+1} x_{k+2} \cdots x_m,
\end{equation}
and \(\delta_{Z_m}^H=1\).
To see this, observe that if $V_k$ is the subspace of $V$ defined by
$V_k = \{ (c_0, \ldots, c_k, 0, \ldots, 0) \mid c_i \in \C \} \subseteq V$, then $Z_k = V_k^{\times}$.  This implies that $\delta_{Z_k}^H$ is
the $H$-equivariant top Chern class of the representation $V/V_k$, which equals $x_{k+1}\cdots x_m$.

By Corollary \ref{corollary: pullback}, the weighted Schubert classes are given by the formula
$\pi^\ast \delta_{X_k}^T = q_{v_k} \delta_{Z_k}^H$, where $q_{v_k}$ is defined as in  \eqref{e.qwdef}.  We now calculate $q_{v_k}$.
Observe that 
$$
\Phi_{v_k}^P = \Phi^+ \cap w \Phi(\fu^-) = \{ x_0 - x_k, x_1 - x_k, \ldots, x_{k-1} - x_k \}.
$$
Therefore
$$
\gcd(a_{v_k},a_{\Phi_{v_k}^P}) = \gcd(a_k, a_0 - a_k, a_1- a_k, \ldots, a_{k-1} - a_k) = \gcd(a_0, \ldots, a_k).
$$
Since $w_0 = v_m$, we have $\gcd(w_0, a_{\Phi_{w_0}^P})  = \gcd(a_0,\ldots,a_m) = \gcd(\chi)$.  Hence
\begin{equation} \label{e.q-WPS}
q_{v_k} =  \frac{\gcd(a_{v_k},a_{\Phi_{v_k}^P})}{\gcd(w_0, a_{\Phi_{w_0}^P}) } =  \dfrac{\gcd(a_0,\ldots,a_k)}{\gcd(\chi)}.
\end{equation}
Therefore, the weighted Schubert classes are given by
\begin{equation}\label{equation: weighted schubert classes in WPS}
\delta_{X_k}^T=\dfrac{\gcd(a_0,\ldots,a_k)}{\gcd(\chi)}x_{k+1}\cdots x_m.
\end{equation}

We now calculate $\delta_{Z_{u_\alpha}}^H\,\delta_{Z_{v_k}}^H$ using Corollary \ref{c:cominuscule}.
We have $\lambda = \omega_{\alpha}$ where
$\alpha = x_0 - x_1$, so $u_{\alpha} = v_{m-1}$.    Also, $w_0 \omega_{\ga} - v_k \omega_{\ga} = \gamma_k$,
where $\gamma_k$ is the negative root $x_m - x_k$.  
The only root $\gg$ occuring on the right hand side of the formulas in Corollary \ref{c:cominuscule}
is $\gg = x_0 - x_{m-k+1}$,
and $v_k r_{\gg} = v_{k-1}$.  We have $\omega_{\alpha} \cdot \gamma^\vee = 1$, and $\frac{a_{w_0}}{a_{v_k}} = \frac{a_m}{a_k}$.  Therefore, 
Corollary \ref{c:cominuscule} and Remark \ref{r:cominusculew0} yield
\begin{align}
 \delta_{Z_{m-1}}^H \delta_{Z_k}^H & = \overline{\gamma}_k(v_k) \delta_{X_k}^T +   \frac{a_m}{a_k} \delta_{Z_{k-1}}^H  \label{e:Chev-WPS}\\
 & = \frac{a_m}{a_k} \big( \overline{\gamma}_k(w_0) \delta_{X_k}^T +   \delta_{Z_{k-1}}^H \big).
\end{align}

We remark that although our formula for $ \delta_{Z_{m-1}}^H \delta_{Z_k}^H $ was derived using Corollary \ref{c:cominuscule},
the product can be calculated directly, using
the formula $\delta_{Z_k}^H=x_{k+1} x_{k+2} \cdots x_m$.  Indeed,
\begin{align*}
 \delta_{Z_{m-1}}^H \delta_{Z_k}^H  & = x_m \cdot x_{k+1} \cdots x_m = (x_m -  \frac{a_m}{a_k} x_k +  \frac{a_m}{a_k} x_k) x_{k+1} \cdots x_m \\
 & = (x_m -  \frac{a_m}{a_k} x_k) x_{k+1} \cdots x_m +  \frac{a_m}{a_k} x_k x_{k+1} \cdots x_m \\
 & = \overline{\gamma}_k(v_k) \delta_{Z_k}^H +   \frac{a_m}{a_k} \delta_{Z_{k-1}}^H ,
\end{align*}
recovering our first formula for the product.  The second formula follows from this, since calculating from the definitions
shows that $\overline{\gamma}_k (v_k) =  \frac{a_m}{a_k} \overline{\gamma}_k(w_0)$.

\begin{remark}
Weighted projective space can be realized as a weighted flag variety for other groups $G$ besides $GL(m+1)$, for example,
for $G$ equal to $\C^* \times SL(m+1)$ or $\C^* \times GL(m+1)$.  Computations involving different realizations, while related, are not exactly
the same, since the tori involved are different.  For example, in this section we took $G = GL_{m+1}$; in this case, $T_0$ is a maximal torus of $PGL_{m+1}$.  This torus is a quotient
of  the torus $T_0$ for
$\C^* \times SL_{m+1}$, which
is a maximal torus of $SL_{m+1}$.  Therefore, some care is needed when comparing computations using different realizations. 
Note that weighted projective space, realized for the group $\C^* \times GL(m)$,
 is the special case $d=1$ of the weighted Grassmannians considered in the next section.
 \end{remark}

\begin{example} \label{e:nonexpress}
Corollary \ref{corollary: w0} implies that $c_{uv}^w$ can be expressed as monomials in negative weighted roots at $x$ with nonnegative
coefficients, provided that $x \geq u$ and $x \geq v$.  The example of weighted $X = {\mathbb P}^4$ shows that in general, for arbitrary antidominant $\chi$, it is not possible to retain this positivity while
using negative weighted roots at $x \leq w$.  
Below, we calculate $(\delta^H_{Z_2})^2$.  We comment briefly on the calculation.  The weighted structure constants at any $x$ in $W^P$ can
be calculated from the nonweighted structure constants, using the method of the proof of Theorem \ref{theorem: main theorem}.  Using this method, we start
with coefficients in $H^*_{T_0}$ and convert to coefficients in $H^*_T$, which adds a layer of complication. To avoid this, we use the 
weighted Chevalley formula \eqref{e:Chev-WPS} to deduce a formula for $(\delta^H_{Z_2})^2$ with coefficients
in $H^*_T$.  We obtain:
\begin{equation} \label{e:delta2square2}
(\delta^H_{Z_2})^2 = \frac{a_3}{a_4} (\overline{\gamma}_2(v_2) - \overline{\gamma}_3(v_3)) \overline{\gamma}_2(v_2)  \, \delta^H_{Z_2}
+ \frac{a_3}{a_2} ( \overline{\gamma}_1(v_1) + \overline{\gamma}_2(v_2) - \overline{\gamma}_3(v_3) )\delta^H_{Z_1} + \frac{a_3 a_4}{a_1 a_2} \delta^H_{Z_0}.
\end{equation}
In the tables below, we record expressions for the
coefficients of $\delta^H_{Z_2}$ and $\delta^H_{Z_1}$ in terms of the negative simple roots at each $x \in W^P$.  These expressions were
obtained from \eqref{e:delta2square2} by using \eqref{e:positive4}, which for any $v, w \in W$, allows us to 
express a weighted root at $v$ in terms of weighted roots at any $w$.  
There is no table for the coefficient of $\delta^H_{Z_0}$ since this coefficient is a constant (in fact, it is the coefficient in the non-equivariant product).

\begingroup

\renewcommand{\arraystretch}{1.5}
\begin{center}
\begin{tabular}{c|c}
$W^P$ &   Coefficient of $ \delta^H_{Z_2}$ in $(\delta^H_{Z_2})^2 $ \\
\hline
unweighted & $ \beta_2(\beta_2 + \beta_3)$ \\
$v_4$ & $ \frac{a_3}{a_2} \Big( \frac{a_4}{a_2} \overline{\gb}_2(v_4) +   (\frac{a_4}{a_2} - \frac{a_4}{a_3})\overline{\gb}_3(v_4)  \Big) \Big(\overline{\gb}_2(v_4) + \overline{\gb}_3(v_4) \Big)  $  \\
$v_3$ & $ \frac{a_3}{a_1} \overline{\gb}_2(v_3) \Big(  \frac{a_4}{a_2}  \overline{\gb}_2(v_3) +  \overline{\gb}_3(v_3) \Big) $ \\
$v_2$ & $\overline{\gb}_2(v_2)  \Big( \overline{\gb}_2(v_2) + \overline{\gb}_3(v_2)  \Big) $  \\
$v_1$ & $   \Big( (1 - \frac{a_3}{a_2})  \overline{\gb}_1(v_1) + \overline{\gb}_2(v_1) \Big)   \Big( (1 - \frac{a_4}{a_2})  \overline{\gb}_1(v_1) + \overline{\gb}_2(v_1) + \overline{\gb}_3(v_1)  \Big)  $ 
 \\
 $v_0$ & $  \Big( (1 - \frac{a_3}{a_2})  (  \overline{\gb}_0(v_0)  +  \overline{\gb}_1(v_0) )  + \overline{\gb}_2(v_0)  \Big)   \Big((1 - \frac{a_4}{a_2}) (  \overline{\gb}_0(v_0)  +  \overline{\gb}_1(v_0) ) + \overline{\gb}_2(v_0) + \overline{\gb}_3(v_0)  \Big)    $
\end{tabular}
\end{center}

\medskip

\renewcommand{\arraystretch}{1.5}
\begin{center}
\begin{tabular}{c|c}
$W^P$ &  Coefficient of $ \delta^H_{Z_1}$ in $(\delta^H_{Z_2})^2 $ \\
\hline
unweighted & $ \beta_1+ 2 \beta_2 + \beta_3$ \\
$v_4$ & $ \frac{a_3}{a_2} \Big( \frac{a_4}{a_1} \overline{\gb}_1(v_4) + (\frac{a_4}{a_1} + \frac{a_4}{a_2}) \overline{\gb}_2(v_4) + (\frac{a_4}{a_1} + \frac{a_4}{a_2} -  \frac{a_4}{a_3} ) \overline{\gb}_3(v_4) \Big) $ \\

$v_3$ & $ \frac{a_3}{a_2} \Big( \frac{a_4}{a_1} \overline{\gb}_1(v_3) + (\frac{a_4}{a_1} + \frac{a_4}{a_2}) \overline{\gb}_2(v_3) + \overline{\gb}_3(v_3) \Big) $ \\

$v_2$ & $  \frac{a_3}{a_2} \Big( \frac{a_4}{a_1} \overline{\gb}_1(v_2) + (\frac{a_4}{a_3} + 1) \overline{\gb}_2(v_2) + \overline{\gb}_3(v_2) \Big) $ \\

$v_1$  & $ \frac{a_3}{a_2} \Big(  (\frac{a_4}{a_3} - \frac{a_4}{a_2} +1)  \overline{\gb}_1(v_1) + (\frac{a_4}{a_3} +1)\overline{\gb}_2(v_1) + \overline{\gb}_3(v_1) \Big) $ \\

$v_0$  &  $ \frac{a_3}{a_2} \Big( (1 - \frac{a_4}{a_1} - \frac{a_4}{a_2} + \frac{a_4}{a_3})  \overline{\gb}_0(v_0) + ( 1 -  \frac{a_4}{a_2} + \frac{a_4}{a_3} ) \overline{\gb}_1(v_0)
+ (1 +  \frac{a_4}{a_3} ) \overline{\gb}_2 (v_0) +  \frac{a_4}{a_3}  \overline{\gb}_3(v_0) \Big) $ \\
\end{tabular}
\end{center}

We observe that for $x \geq v_2$, the coefficients of $ \delta^H_{Z_2}$ and $ \delta^H_{Z_1}$ are nonnegative at $x$.  The coefficient of
$ \delta^H_{Z_2}$ is nonnegative at $v_1$ only if $\chi$ (which is still assumed antidominant) is highly singular---in fact, the entries of $\chi$ must satisfy $a_2 = a_3 = a_4$.  
In this case, $a_{\beta_2} = a_{\beta_3} = 0$, so $\overline{\gb}_i(v_1) = \gb_i$ for $i = 2,3$, and we see that the weighted coefficient is the same as the non-weighted
coefficient.  In contrast, the coefficient of $ \delta^H_{Z_1}$ is nonnegative at $v_1$ provided that $\frac{a_4}{a_3} - \frac{a_4}{a_2} +1 \geq 0$.  This is equivalent to
the condition $\frac{1}{a_2} - \frac{1}{a_3} \leq \frac{1}{a_4}$. 
 We see that there are antidominant regular $\chi$ such that the coefficient of $ \delta^H_{Z_1}$ is nonnegative at $v_1$.

\end{example}

\endgroup

\subsection{Weighted Grassmannians} \label{ss:Grassmannian}
In this section, we discuss weighted Grassmannians from our point of view, and
provide a dictionary to translate between our notation (GL) and that in Abe-Matsumura \cite{AbeMatsumura2015} (AM).
In particular, we interpret the positivity statement in \cite{AbeMatsumura2015} precisely in terms of our parameters.

Let $\ge_1, \ldots, \ge_{m}$ denote the standard basis of $\C^m$.
Set \(G=\bC^\times\times GL_m\), so \(n=m-1\) and \(H=\bC^{\times(m+1)}\).  
Let $x_0, \ldots, x_m$ denote the standard basis of the character group $\Lambda \cong \Z^{m+1}$ of $H$;
as in the previous section, we sometimes denote the element $\sum c_i x_i$ by $(c_0, \ldots, c_m)$.
The Weyl group $W \cong S_m$ acts on $\Lambda$; if $w =(w_1 \ldots w_m)$, then
$w x_0 = x_0$, and $w x_i = x_{w_i}$ for $i = 1, \ldots, m$.

Let $\lambda=(1,\ldots,1,0,\ldots,0) = x_0 + x_1 + \cdots + x_k$.  
Let $P = \C^* \times P_0$, where $P_0$ is the block upper triangular subgroup of $GL_m$
with diagonal blocks of sizes $k$ and $m-k$.  We have \(W_P = W_\lambda=S_k\times S_{m-k}\), and $W^P$ is the set of permutations
$w = (w_1, \ldots, w_m)$ such that $w_1 > w_2 > \cdots > w_k$ and $w_{k+1} > w_{k+2} > \cdots > w_m$.
Let $\bracenom{m}{k}$ denote the collection of $k$-element subsets of $\{ 1, \ldots, m \}$.
We identify $W^P$ with $\bracenom{m}{k}$ by the map $I: w \mapsto \{ w_1, \ldots, w_k \}$.
The covering relations in $W^P$ are as follows.  Let $\gg = x_p - x_q$ where
$p \in \{ 1, \ldots, k \}$ and $q \in \{k+1, \ldots, m \}$.  Then $w = v r_{\gg} \lessdot_P v$ if and only 
if $v_p = v_{q} + 1$.

The group $G$ acts transitively on the
Grassmannian $Gr(k,m)$ of $k$-planes in $\C^m$, and $P$ is the stabilizer of the plane spanned by the first
$k$ coordinate vectors $\ge_1, \ldots, \ge_k$, so $G/P \cong Gr(k,m)$.  
Under this isomorphism, the point $wP$ corresponds to the plane
spanned by $\ge_{w_1}, \ldots, \ge_{w_k}$.  For later use, we remark
that if $P' = w_0 P w_0^{-1}$, then $P'$ is the block lower triangular subgroup of $GL_m$ with
blocks of sizes $m-k$ and $k$.  Since $P'$ is the stabilizer of the plane spanned by
the last $k$ coordinate vectors, we can identify $Gr(k,m)$ with $G/P'$.  The map
$G/P \to G/P'$, $gP \mapsto g w_0 P'$, is an isomorphism compatible with the
identifications of $G/P$ and $G/P'$ with $Gr(k,m)$.

Let \(\chi=(a_0,\ldots,a_m)\), where for every \(w\in W^P\), we have
\begin{equation}
a_w=a_0+a_{w(1)}+\cdots +a_{w(k)}>0,
\end{equation}
so \eqref{equation: a_w} is satisfied.
Then \(\chi\) is antidominant if and only if
\begin{equation}\label{equation: chi is antidom}
a_1\leq a_2\leq \cdots\leq a_m.
\end{equation}
The weighted Grassmannian
is $w Gr(k,m) = S \backslash Z$.  If we take $\chi_0 = (1, 0, \ldots, 0)$, then $S_0 \backslash Z \cong G/P \cong G_0 /P_0 = Gr(k,m)$.

In this setting, the Chevalley formula takes a simple form.  The only Schubert
divisor corresponds to $u_{\ga}$, with 
$\ga = \ga_k = x_k - x_{k+1}$.  We can take 
$\omega_{\ga} = \lambda = x_0 + x_1 + \cdots + x_k$.
Then for any $v \in W$, we have $v \lambda =  x_0 + x_{v_1} + \cdots + x_{v_k}$.
Similarly, $a_v = a_0 +  a_{v_1} + \cdots + a_{v_k}$.  Taking $\gg = x_p - x_q$
where $p \leq k$ and $q \geq k + 1$, we have $\lambda \cdot \gamma^{\vee} = \omega_{\ga} \cdot \gamma^{\vee} = 1$.
Corollary \ref{c:cominuscule} applies, so
\begin{equation} \label{e.ChevalleyGr}
\gd^H_{Z_{u_{\ga}}} \gd^H_{Z_v} = \overline{w_0 \omega_{\alpha}}(v) \gd^H_{Z_v}  + \frac{a_{w_0}}{a_v} \sum_{w \lessdot_P v}  \gd^H_{Z_w}.
\end{equation}
We remark that since there is only one Schubert divisor, Lemma \ref{lemma: lambdamultiply} implies that
$\gd^H_{Z_{u_{\ga}}} = (w_0 \lambda)  \gd^H_{Z_{w_0}} = w_0 \lambda$.
 
We now explain the relation between our notation (GL) and the notation of \cite{AbeMatsumura2015} (AM).  Some of the notation of \cite{AbeMatsumura2015} 
is drawn from \cite{KnutsonTao2003}, where more explanation can be found.   Abe and Matsumura parametrize Schubert
varieties in $Gr(k,m)$ by $\gl = \{ \gl_1, \ldots, \gl_k \} \in \bracenom{m}{k}$.  (Their use of the symbol $\gl$ bears no relation to ours; also, they
work with $Gr(k,n)$, but we have changed their $n$ to $m$ to be consistent with this section.)  
Corresponding to $\lambda \in \bracenom{m}{k}$ there is a plane, which we denote here by $F_{\gl}$, spanned by the coordinate vectors
$\ge_{\gl_1}, \ldots, \ge_{\gl_k}$.  The $B$-orbit $B \cdot F_{\gl} = \Omega_{\gl}^0 \subset Gr(k,m)$ is a Schubert cell whose closure is the
Schubert variety $\Omega_{\gl}$.  In our notation, the point $wP$ in $G/P \cong Gr(k,m)$ corresponds to the plane spanned
by $\ge_{w_1}, \ldots, \ge_{w_k}$.  Therefore, our $Y_w$ coincides with $\Omega_{I(w)}$.
Note that although we view $Y_w = \overline{B \cdot w P} \subset G/P$, their terminology is consistent with the identification
 $\Omega_{I(w)} = \overline{B \cdot I(w) w_0 P'}  \subset G/P'$, where $P' = w_0 P w_0^{-1}$ as above.  Thus, 
they denote by $\mbox{id}$ the set $\{ n-k+1, \ldots, n \} = I(w_0)$,
and by $w_0$ the set $\{1, \ldots, k \} = I(k,\ldots,1,m,\ldots,k+1)$.   

We identify the larger torus $H$ with $\C^* \times T_0$, while their torus is (in effect) $T_0 \times \C^*$.
We have used $\chi = (a_0, \ldots, a_n)$ to denote the cocharacter defining the torus $S$, whereas they are taking the quotient by
a torus $wD$ corresponding to a cocharacter $(w_1, \ldots, w_n, a)$ (their $w$ does not denote an element of $W$).  In our notation, their
cocharacter corresponds to $a_0 = a$ and $a_i = w_i$ for $i = 1, \ldots, n$.

This correspondence can be continued.  In Table 1, we record a list of notations from this paper, along with the
corresponding notation from \cite{AbeMatsumura2015}.  In this table, $\lambda = I(w)$.  In this table, the $d_{uv}^w$ are as in
\eqref{e:Z-constants}, i.e., the structure constants with regard to the basis $\delta^H_{Z_w}$.
Observe that the weighted root $\mathrm{w}u_\alpha$ of \cite{AbeMatsumura2015} corresponds to our
$\overline{\alpha}(w_0)$, so our positivity statement recovers theirs.  Similarly, the Pieri-Chevalley formula \eqref{e.ChevalleyGr}
can be seen to recover their Proposition 5.2.

\begin{table}[ht] 
  \centering
  \caption{A dictionary to \cite{AbeMatsumura2015}}
  \begin{tabular}{|r|l|}
    \hline
    GL & AM \\
    \hline
\(H\) & \(K\) \\
\(\chi=(a_0,\ldots,a_m)\)       & \(\rho=(a,w_1,\ldots,w_m)\)\\   
\(S\)      & \(\mathrm{wD}\)\\   
\(T\) & \(T_w\)\\
\(Z\)  & \(\mathrm{aPl}(d,m)^\times\)\\ 
\(X\) & \(\mathrm{wGr}(d,m)\)\\
\(w\in W^P\)   & \(\lambda\in \plucker{m}{d}\)\\   
\(w_0\in W^P\)      & \(\mathrm{id}\in\plucker{m}{d}\)\\   
\(a_w\)     & \(w_\lambda\) \\  
\(Y_{w_0}=G/P\)      & \(\Omega_{\mathrm{id}}=G/P'\)\\   
\(Y_w=\overline{BwP}/P\)   & \(\Omega_\lambda=\overline{B \cdot I(w) w_0 P'}/P'\)\\   
\(P/P\)  & \(\Omega_{(1,\ldots,1,0,\ldots,0)}\)\\   
\(Z_w\)      & \(\mathrm{a}\Omega_\lambda\)\\   
\(X_w\)     & \(\mathrm{w}\Omega_\lambda\)\\    
\(\delta_{Z_w}^H\)      & \(\tilde S_\lambda=\mathrm{a}\tilde S_\lambda=\mathrm{w}\tilde S_\lambda\)\\   
\(\delta_{Z_{u_{\ga}}}^H  = w_0 \lambda \)   & \(\mathrm{a}\tilde S_{\mathrm{div}}\)\\ 
\(d_{uv}^w\)   & \(\mathrm{w}\tilde c_{\lambda\mu}^\nu\)\\   
\(S(\H^\ast)\)      & \(\bQ[z,y_1,\ldots,y_m]\)\\   
\(S(\T^\ast)\)      & \(\bQ[y_1^w,\ldots,y_m^w]\)\\   
\(\alpha=x_i-x_j\)      & \(u_\alpha=y_i-y_j\)\\   
\(\overline\alpha(w_0)\)     & \(\mathrm{w}u_\alpha\)\\   
\(\overline\alpha_i(w_0)\) & \(\mathrm{w}u_i\) \\   
    \hline
  \end{tabular}
\end{table}
We conclude this section by looking more closely at the case where
 \(m=4\) and \(k=2\).
Identifying $W^P$ with $\bracenom{4}{2}$, we can depict \((W^P,\leq)\) as
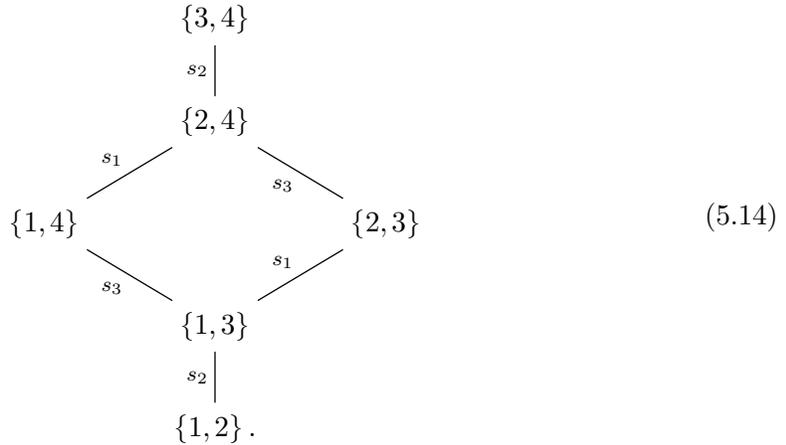
\begin{equation}
\begin{tikzcd}
&\set{3,4}\arrow[swap,d, dash,"s_2"]&\\
&\set{2,4}\arrow[swap,dl, dash,"s_1"]\arrow[swap,dr, dash,"s_3"]&\\
\set{1,4}\arrow[swap,dr, dash,"s_3"]&&\set{2,3}\arrow[swap,dl, dash,"s_1"]\\
&\set{1,3}\arrow[swap,d, dash,"s_2"]&\\
&\set{1,2}.&
\end{tikzcd}
\end{equation}
The maximal element is \(w_0=\set{3,4}\).
As above, let \(\set{x_0,\ldots,x_4}\) denote the standard basis of \(\H^\ast\), and let \(y_i=x_i\vert\ker\lambda \in \H_{\lambda}^\ast\).
We have \(\lambda=x_0+x_1+x_2\), and
$S(\H_{\gl}^\ast) = \F[y_0, \ldots, y_4]/\langle y_0+y_1+y_2 \rangle$.  Thus, \(y_0=-(y_1+y_2)\), and
$S(\H_{\gl}^\ast)^{W_P} \isom \F[y_1+y_2, y_3+y_4]$.  We have $S(\H^\ast) = \F[x_0, \ldots, x_4]$ and
$S(\H^\ast)^W = \F[x_0, e_1, \ldots, e_4]$, where $e_i$ is the $i$-th elementary symmetric function in $x_1, \ldots, x_4$.
The map $\H_{\gl} /W^{\gl} \to \H/W$ is not injective (for example, the points \(W_\lambda(0,0,0,-1,1)\) and \(W_\lambda(0,-1,1,0,0)\) are distinct in \(W_\lambda\backslash\H_{\lambda}\), but map to the same
element in $\H/W$).
Hence $S(\H^\ast)^W \to S(\H_{\gl}^\ast)^{W_P}$ is not surjective, so
Lemma~\eqref{lemma: surjection} does not apply.

The table below gives formulas for the classes $\gd^H_{Z_w}$.  For simplicity, we write simply $Z_{ \{a,b \} }$ for the class $\gd^H_{Z_{ \{ a,b \} }}$.

\begin{center}
\begin{tabular}{c|c}
Class in \(H_H^\ast(Z)\) & Formula in \(S(\H^\ast)\otimes_{S(\H^\ast)^W}S(\H_{\lambda}^\ast)^{W_\lambda}\)\\
\hline
$Z_{\set{3,4}}$&$1$\\
$Z_{\set{2,4}}$&$w_{\set{3,4}}\lambda=x_0+x_3+x_4$\\
$Z_{\set{1,4}}$&$(x_0+x_2)(x_0+x_3+x_4)+x_3x_4-y_1y_2$\\
$Z_{\set{2,3}}$&$(x_4-y_1)(x_4-y_2)$\\
$Z_{\set{1,3}}$&$(x_0+x_2+x_3)(x_4-y_1)(x_4-y_2)$\\
$Z_{\set{1,2}}$&$(x_3-y_1)(x_4-y_1)(x_3-y_2)(x_4-y_2)$\\
\end{tabular}
\end{center}

\subsection{A Weighted Lagrangian} \label{ss:Lagrangian}
Let \(G=\bC^\times\times\Sp(4)\), where $Sp(4)$ is the subgroup of $GL(2)$ preserving the symplectic form
$\omega$ on $\C^4$ corresponding to the matrix $\begin{bmatrix} 0 & I \\ -I & 0 \end{bmatrix}$; here $I$ is the $2 \times 2$ identity matrix.
Let \(H = \C^* \times T_0\), where $T_0$ is the set of diagonal matrices in $Sp(4)$. In the notation of Section \ref{ss: definitions}, we have \(m=n=2\).
Let $x_0, x_1, x_2$ denote the standard basis of $\H^*$, and write $(c_0, c_1, c_2)$ for $\sum c_i x_i$.
The root system $\Phi$ is of type $C_2$ and is described, for example, in \cite{Humphreys1978}.
We take as simple roots \(\alpha_1=(0,1,-1)\) and \(\alpha_2=(0,0,2)\), with corresponding coroots \(\alpha_1^\vee=(0,1,-1)\) and \(\alpha_2^\vee=(0,0,1)\).
The set of dominant characters is \(\Lambda^+=\set{\lambda=(c_0,c_1,c_2)\in\Lambda\mid c_1\geq c_2 \geq 0}\).
The Weyl group $W$ is identified with the set of signed permutations on $2$ elements.
Let $P = \C^* \times P_0$, where $P_0 \subset Sp(4)$ is the stabilizer of the Lagrangian $2$-plane in $\C^4$ spanned by the first $2$ coordinate
vectors.  Then $Y = G/P = LG(2,4)$, the Grassmannian of Lagrangian $2$-planes in $\C^4$.  We have $W_P = \{1, r_{\ga_1} \} \cong S_2$.
Following the notation of \cite{BoeGraham2003} for
signed permutations, we write $\bar{a}$ for $-a$.  We have $W^P = \{w_0, w_1, w_2, w_3 \}$, where $w_0 = (\bar{1}, \bar{2})$, $w_1 = (\bar{1},2)$, $w_2 = (\bar{2},1)$,
$w_3 = (2,1)$.  If we write $Y_k = Y_{w_k}$, then the Schubert variety $Y_k$ has codimension $k$ in $Y$.
Let \(\lambda=(1,1,1) = \omega_{\alpha_2} \in\Lambda^+\).  If $w \in W$, then $w x_0 = x_0$, and $w x_i = x_{w(i)}$ for $i = 1, 2$.  Let $\chi = (a_0, a_1, a_2)$; then 
$a_w = a_0 + a_{w(1)} + a_{w(2)}$; we assume $\chi$ is antidominant.  In these formulas, our convention is that $x_{\bar{k}} = - x_k$ and $a_{\bar{k}} = - a_k$.

Let $y_i$ denote the restriction of $x_i$ to $\H_{\gl}$, so $y_0+y_1+y_2 = 0$, and $S(\H_\lambda^\ast) = \F[y_0, y_1, y_2]/ \langle y_0+y_1+y_2 \rangle
\cong \F[y_1,y_2]$.  Under this identification, the restriction homomorphism
$S(\H^\ast) \to S(\H_\lambda^\ast)$ satisfies $x_0 \mapsto -(y_1+y_2)$, $x_i \mapsto y_i$ for $i \geq 1$.
The induced homomorphism \(S(\H^\ast)^W\to S(\H_\lambda^\ast)^{W_\lambda}\) is the corresponding map \(\F[x_0,x_1^2+x_2^2,x_1^2x_2^2]\to\F[y_1+y_2,y_1y_2]\).
This map is surjective, and one can show that the kernel is generated by $\prod w_i \gl$.  
It follows that
\begin{equation}\label{equation: weighted lagrangian}
H_H^\ast(Z)=\dfrac{\F[x_0,x_1,x_2]}{ \langle w_0\lambda \cdot w_1\lambda \cdot w_2\lambda \cdot w_3\lambda \rangle}.
\end{equation}
The homomorphism \(k^\ast\,\colon H_H^\ast(Z)\to\bigoplus_{w\in W^P}S(\H_{w\lambda}^\ast)\) from \eqref{equation: fixed points} is the natural map from \eqref{equation: weighted lagrangian} to
\begin{equation}
\dfrac{\F[x_0,x_1,x_2]}{ \langle w_0\lambda \rangle }\oplus\dfrac{\F[x_0,x_1,x_2]}{ \langle w_1\lambda \rangle }\oplus\dfrac{\F[x_0,x_1,x_2]}{ \langle w_2\lambda \rangle }\oplus\dfrac{\F[x_0,x_1,x_2]}{ \langle w_3\lambda \rangle }.
\end{equation}

We claim that 
\begin{equation} \label{e.Z3}
\delta^H_{Z_3}=\frac{1}{2}w_2\lambda \ w_1\lambda\ w_0 \lambda.  
\end{equation}
Since $k^*$ is injective, it suffices to show that the left hand side and the right
hand side have the same image under $k^*_v$ for all $v \in W^P$.  If $v \neq w_3$, then $k^*_v$ takes both sides of \eqref{e.Z3} to $0$, so it suffices to show that
$k^*_{w_3} \delta^H_{Z_3} = \frac{1}{2}w_2\lambda \ w_1\lambda\ w_0 \lambda$ in $\F[x_0,x_1,x_2] /  \langle w_3\lambda \rangle = \F[x_0,x_1,x_2] /  \langle \lambda \rangle $.
Observe that $\psi_{w_3}: H^*_{T_0} \to H^*_{\H_{w_3 \gl}}$ takes $i_{w_3}^* \gd^{T_0}_{Y_3}$ to $k_{w_3}^* \gd^{H}_{Z_3}$.
Since $Y_3$ is the $B$-fixed point in $Y = G/P$, $i_{w_3}^* \gd^{T_0}_{Y_3}$ is the product of the weights of $T_{eP} Y \cong \fu^-$, so
$$
i_{w_3}^* \gd^{T_0}_{Y_3} = - \ga_2 \cdot (\ga_1 + \ga_2) (2 \ga_1 + \ga_2) = -4 (x_1 + x_2) x_1 x_2 \in H^*_{T_0} = \F[x_1, x_2].
$$
The map $\psi_{w_3}: H^*_{T_0} \to H^*_{\H_{w_3 \gl}}$ takes this element to its image in $\F[x_0,x_1,x_2] /  \langle \lambda \rangle $.  
Hence, $k_{w_3}^* \gd^{H}_{Z_3} = -4 (x_1 + x_2) x_1 x_2 \in \F[x_0,x_1,x_2] /  \langle \lambda \rangle $.
In $\F[x_0,x_1,x_2] /  \langle \lambda \rangle $, we have the relation $x_0 = -(x_1+x_2)$, so in this ring,
$ w_2 \gl = x_0 + x_1 - x_2 = -2x_2$, $w_1 \gl = x_0 - x_1 + x_2 = - 2 x_1$, and $w_0 \gl = x_0 - x_1 - x_2 = -2 (x_1 + x_2)$.
The claim follows from these calculations.

Formulas for the remaining Schubert classes can be obtained using the formula \(\partial_{\alpha} \delta_{Z_w}=\delta_{Z_{r_{\alpha} w}}\),
where $\alpha$ is a simple root such that \(w<r_{\alpha} w\).
Here $\partial_{\alpha}$ denotes the divided difference operator acting on $H^*_H$, defined by the formula
$ \partial_{\alpha} f=\frac{r_{\ga} f-f}{\alpha}$ for $f \in H^*_H$.  The resulting formulas are recorded in the table below.

\begin{center}
\begin{tabular}{c|c}
Class in \(H_H^\ast(Z)\)&Formula in \(\F[x_0,x_1,x_2]/(\prod w_k\lambda)\)\\
\hline
\(\delta_{Z_0}\)&\(1\)\\
\(\delta_{Z_1}\)&\(w_0\lambda\)\\
\(\delta_{Z_2}\)&\(\frac{1}{2}w_1\lambda\ w_0\lambda\)\\
\(\delta_{Z_3}\)&\(\frac{1}{2}w_2\lambda\ w_1\lambda\ w_0\lambda\)\\
\end{tabular}
\end{center}

Since $\lambda = x_0 + \omega_{\alpha_2}$, we can compute products with with the divisor class $\gd^H_{Z_1} = \gd^H_{Z_{u_{\ga_2}}}$ in terms
of negative simple roots at $w_0$ using
the version of the Chevalley formula given in Remark \ref{r:cominusculew0}.   We briefly
describe the inputs to the calculation.  We have 
$$
a_{w_0} = a_0 - a_1 - a_2, \hspace{.2in} a_{w_1} = a_0 - a_1 + a_2 , \hspace{.2in} a_{w_2} = a_0 + a_1 - a_2 , \hspace{.2in} a_{w_3} = a_0 + a_1 + a_2 .
$$
Let $c_i = \frac{a_{w_0}}{a_{w_i}}$.  Because $w \mapsto a_w$ is an increasing positive function on $W^P$, we have $0< c_0 \leq c_1 \leq c_2 \leq c_3$.
The nonweighted case corresponds to $a_0 = 1$, $a_1 = a_2 = 0$, and then each $c_i = 1$.

Write $\beta_i = - \alpha_i$
for the negative simple roots, and $\overline{\beta}_i = \beta_i(w_0)$.  In computing $\gd^H_{Z_1} \gd^H_{Z_i}$ for $i = 0, 1, 2$, there is only one root $\gamma = \gamma_i$ which occurs in the Chevalley formula; this root
satisfies $w_i r_{\gamma_i} = w_{i+1}$.  We have
$\gamma_0 = 2 x_2 = \ga_2$; $\gamma_1 = x_1+x_2 = \ga_1 + \ga_2$; and $\gamma_2 = 2 x_1 = 2 \ga_1 + \ga_2$.
The corresponding coroots are $\gamma_0^{\vee} = x_2 = \alpha_2^{\vee}$, $\gamma_1^{\vee} = x_1 + x_2 = \alpha_1^{\vee} + 2 \alpha_2^{\vee}$,
and $\gamma_2^{\vee} = x_1 = \alpha_1^{\vee} + \alpha_2^{\vee}$.  Hence $\omega_{\ga_2} \cdot \gamma_0^{\vee} = \omega_{\ga_2} \cdot \gamma_2^{\vee} = 1$
and $\omega_{\ga_2} \cdot \gamma_1^{\vee} = 2$.
Using these inputs, the Chevalley
formula yields the following table of products (we omit the product with $\gd^H_{Z_0}$, which is the identity element).
\begin{align*}
\gd^H_{Z_1} \gd^H_{Z_1} & =  c_1 \Big( \overline{\gb}_2 \gd^H_{Z_1} +  2 \gd^H_{Z_2} \Big) \\
\gd^H_{Z_1} \gd^H_{Z_2} & =  c_2 \Big( (2 \overline{\gb}_1 + \overline{\gb}_2) \gd^H_{Z_2} +  
 \gd^H_{Z_3} \Big) \\
\gd^H_{Z_1} \gd^H_{Z_3} & = 2 c_3 ( \overline{\gb}_1 + \overline{\gb}_2)  \gd^H_{Z_3}.
\end{align*}
Using this table, one can compute the other products $(\gd^H_{Z_2})^2$, $\gd^H_{Z_2} \gd^H_{Z_3}$, and $(\gd^H_{Z_3})^2$.
For example,
$$
(\gd^H_{Z_2})^2 = \frac{c_2}{2} (2 \overline{\gb}_1 + \overline{\gb}_2) [ \frac{2 c_2}{c_1} \gb_1 + (\frac{c_2}{c_1} - 1) \overline{\gb}_2 ] \gd^H_{Z_2}
+ \frac{c_2}{2 c_1} [2(c_2 + c_3) \overline{\gb_1} + (c_2 + 2 c_3 - c_1) \overline{\gb}_2 ]  \gd^H_{Z_3}.
$$
The non-weighted products are recovered by taking each $c_i = 1$ and $\overline{\gb}_i = \gb_i$.

The formulas for the products $\gd^T_{X_i} \gd^T_{X_j}$ can be obtained from the corresponding $H$-equivariant product formulas by inserting appropriate
factors involving $q_{w_i}$.  We will not write down the general formulas, but as an example, we provide
the following table recording the $q_{w_i}$ for a few $\chi$:
the entry $(a_0,a_1, a_2)$ is the cocharacter $\chi$, and below are the corresponding $q_{w_i}$.
\begin{center}
\begin{tabular}{c|ccc}
$w\in W^P$ & $(8,-1,-1)$ &$(3,-1,0)$ & $(11,-4,3)$ \\
\hline
$w_0$ & 1 & 1 & 1 \\
$w_1$ & 1 & 1 & 1 \\
$w_2$ & 1 & 2  & 2 \\
$w_3$ & 3 & 2 &  4 \\
\end{tabular}
\end{center}



\begin{bibdiv}
\begin{biblist}

\bib{AbeMatsumura2015}{article}{
      author={Abe, Hiraku},
      author={Matsumura, Tomoo},
       title={Equivariant cohomology of weighted {G}rassmannians and weighted
  {S}chubert classes},
        date={2015},
        ISSN={1073-7928},
     journal={Int. Math. Res. Not. IMRN},
      number={9},
       pages={2499\ndash 2524},
         url={https://doi.org/10.1093/imrn/rnu003},
}

\bib{AndersenJantzenSoergel1994}{article}{
      author={Andersen, H.~H.},
      author={Jantzen, J.~C.},
      author={Soergel, W.},
       title={Representations of quantum groups at a {$p$}th root of unity and
  of semisimple groups in characteristic {$p$}: independence of {$p$}},
        date={1994},
        ISSN={0303-1179},
     journal={Ast\'erisque},
      number={220},
       pages={321},
}

\bib{AndersonFulton2024}{book}{
      author={Anderson, David},
      author={Fulton, William},
       title={Equivariant cohomology in algebraic geometry},
      series={Cambridge Studies in Advanced Mathematics},
   publisher={Cambridge University Press, Cambridge},
        date={2024},
      volume={210},
        ISBN={978-1-00-934998-7},
}

\bib{BartoloMartinMoralesOrtigasGalindo2014}{article}{
      author={Artal~Bartolo, Enrique},
      author={Mart\'{\i}n-Morales, Jorge},
      author={Ortigas-Galindo, Jorge},
       title={Cartier and {W}eil divisors on varieties with quotient
  singularities},
        date={2014},
        ISSN={0129-167X,1793-6519},
     journal={Internat. J. Math.},
      volume={25},
      number={11},
       pages={1450100, 20},
         url={https://doi.org/10.1142/S0129167X14501006},
}

\bib{AtiyahMacdonald1969}{book}{
      author={Atiyah, M.~F.},
      author={Macdonald, I.~G.},
       title={Introduction to commutative algebra},
   publisher={Addison-Wesley Publishing Co., Reading, Mass.-London-Don Mills,
  Ont.},
        date={1969},
}

\bib{AzamNazirQureshi2020}{article}{
      author={Azam, Haniya},
      author={Nazir, Shaheen},
      author={Qureshi, Muhammad~Imran},
       title={The equivariant cohomology of weighted flag orbifolds},
        date={2020},
        ISSN={0025-5874},
     journal={Math. Z.},
      volume={294},
      number={3-4},
       pages={881\ndash 900},
         url={https://doi.org/10.1007/s00209-019-02285-x},
}

\bib{Billey99}{article}{
      author={Billey, Sara~C.},
       title={Kostant polynomials and the cohomology ring for {$G/B$}},
        date={1999},
        ISSN={0012-7094},
     journal={Duke Math. J.},
      volume={96},
      number={1},
       pages={205\ndash 224},
}

\bib{BoeGraham2003}{article}{
      author={Boe, Brian~D.},
      author={Graham, William},
       title={A lookup conjecture for rational smoothness},
        date={2003},
        ISSN={0002-9327},
     journal={Amer. J. Math.},
      volume={125},
      number={2},
       pages={317\ndash 356},
  url={http://muse.jhu.edu/journals/american_journal_of_mathematics/v125/125.2boe.pdf},
}

\bib{BorhoMacPherson1983}{incollection}{
      author={Borho, Walter},
      author={MacPherson, Robert},
       title={Partial resolutions of nilpotent varieties},
        date={1983},
   booktitle={Analysis and topology on singular spaces, {II}, {III} ({L}uminy,
  1981)},
      series={Ast\'{e}risque},
      volume={101},
   publisher={Soc. Math. France, Paris},
       pages={23\ndash 74},
}

\bib{Brion1997}{article}{
      author={Brion, M.},
       title={Equivariant {C}how groups for torus actions},
        date={1997},
        ISSN={1083-4362},
     journal={Transform. Groups},
      volume={2},
      number={3},
       pages={225\ndash 267},
         url={https://doi.org/10.1007/BF01234659},
}

\bib{Carrell94}{incollection}{
      author={Carrell, James~B.},
       title={The {B}ruhat graph of a {C}oxeter group, a conjecture of
  {D}eodhar, and rational smoothness of {S}chubert varieties},
        date={1994},
   booktitle={Algebraic groups and their generalizations: classical methods
  ({U}niversity {P}ark, {PA}, 1991)},
      series={Proc. Sympos. Pure Math.},
      volume={56},
   publisher={Amer. Math. Soc., Providence, RI},
       pages={53\ndash 61},
}

\bib{Carrell1995}{incollection}{
      author={Carrell, James~B.},
       title={On the smooth points of a {S}chubert variety},
        date={1995},
   booktitle={Representations of groups ({B}anff, {AB}, 1994)},
      series={CMS Conf. Proc.},
      volume={16},
   publisher={Amer. Math. Soc., Providence, RI},
       pages={15\ndash 33},
}

\bib{CortiReid2002}{incollection}{
      author={Corti, Alessio},
      author={Reid, Miles},
       title={Weighted {G}rassmannians},
        date={2002},
   booktitle={Algebraic geometry},
   publisher={de Gruyter, Berlin},
       pages={141\ndash 163},
}

\bib{CoxLittleSchenck2011}{book}{
      author={Cox, David~A.},
      author={Little, John~B.},
      author={Schenck, Henry~K.},
       title={Toric varieties},
      series={Graduate Studies in Mathematics},
   publisher={American Mathematical Society, Providence, RI},
        date={2011},
      volume={124},
        ISBN={978-0-8218-4819-7},
         url={https://doi.org/10.1090/gsm/124},
}

\bib{Deodhar1977}{article}{
      author={Deodhar, Vinay~V.},
       title={Some characterizations of {B}ruhat ordering on a {C}oxeter group
  and determination of the relative {M}\"{o}bius function},
        date={1977},
        ISSN={0020-9910},
     journal={Invent. Math.},
      volume={39},
      number={2},
       pages={187\ndash 198},
         url={https://doi.org/10.1007/BF01390109},
}

\bib{DuistermaatHeckman1982}{article}{
      author={Duistermaat, J.~J.},
      author={Heckman, G.~J.},
       title={On the variation in the cohomology of the symplectic form of the
  reduced phase space},
        date={1982},
        ISSN={0020-9910},
     journal={Invent. Math.},
      volume={69},
      number={2},
       pages={259\ndash 268},
         url={https://doi.org/10.1007/BF01399506},
}

\bib{EdidinGraham1998}{article}{
      author={Edidin, Dan},
      author={Graham, William},
       title={Equivariant intersection theory},
        date={1998},
        ISSN={0020-9910},
     journal={Invent. Math.},
      volume={131},
      number={3},
       pages={595\ndash 634},
         url={https://doi.org/10.1007/s002220050214},
}

\bib{Fulton1997}{book}{
      author={Fulton, William},
       title={Young tableaux},
      series={London Mathematical Society Student Texts},
   publisher={Cambridge University Press, Cambridge},
        date={1997},
      volume={35},
        ISBN={0-521-56144-2; 0-521-56724-6},
        note={With applications to representation theory and geometry},
}

\bib{Fulton1998}{book}{
      author={Fulton, William},
       title={Intersection theory},
     edition={Second},
      series={Ergebnisse der Mathematik und ihrer Grenzgebiete. 3. Folge. A
  Series of Modern Surveys in Mathematics [Results in Mathematics and Related
  Areas. 3rd Series. A Series of Modern Surveys in Mathematics]},
   publisher={Springer-Verlag, Berlin},
        date={1998},
      volume={2},
        ISBN={3-540-62046-X; 0-387-98549-2},
         url={https://doi.org/10.1007/978-1-4612-1700-8},
}

\bib{GKM1998}{article}{
      author={Goresky, Mark},
      author={Kottwitz, Robert},
      author={MacPherson, Robert},
       title={Equivariant cohomology, {K}oszul duality, and the localization
  theorem},
        date={1998},
        ISSN={0020-9910},
     journal={Invent. Math.},
      volume={131},
      number={1},
       pages={25\ndash 83},
         url={https://doi.org/10.1007/s002220050197},
}

\bib{GoreskyMacPherson1983}{article}{
      author={Goresky, Mark},
      author={MacPherson, Robert},
       title={Intersection homology. {II}},
        date={1983},
        ISSN={0020-9910},
     journal={Invent. Math.},
      volume={72},
      number={1},
       pages={77\ndash 129},
         url={https://doi-org.argo.library.okstate.edu/10.1007/BF01389130},
}

\bib{Graham2001}{article}{
      author={Graham, William},
       title={Positivity in equivariant {S}chubert calculus},
        date={2001},
        ISSN={0012-7094},
     journal={Duke Math. J.},
      volume={109},
      number={3},
       pages={599\ndash 614},
         url={https://doi.org/10.1215/S0012-7094-01-10935-6},
}

\bib{GrahamKreiman2015}{article}{
      author={Graham, William},
      author={Kreiman, Victor},
       title={Excited {Y}oung diagrams, equivariant {$K$}-theory, and
  {S}chubert varieties},
        date={2015},
        ISSN={0002-9947},
     journal={Trans. Amer. Math. Soc.},
      volume={367},
      number={9},
       pages={6597\ndash 6645},
         url={https://doi.org/10.1090/S0002-9947-2015-06288-6},
}

\bib{Humphreys1978}{book}{
      author={Humphreys, James~E.},
       title={Introduction to {L}ie algebras and representation theory},
      series={Graduate Texts in Mathematics},
   publisher={Springer-Verlag, New York-Berlin},
        date={1978},
      volume={9},
        ISBN={0-387-90053-5},
        note={Second printing, revised},
}

\bib{Jantzen2003}{book}{
      author={Jantzen, Jens~Carsten},
       title={Representations of algebraic groups},
     edition={Second},
      series={Mathematical Surveys and Monographs},
   publisher={American Mathematical Society, Providence, RI},
        date={2003},
      volume={107},
        ISBN={0-8218-3527-0},
}

\bib{Jantzen2004}{incollection}{
      author={Jantzen, Jens~Carsten},
       title={Nilpotent orbits in representation theory},
        date={2004},
   booktitle={Lie theory},
      series={Progr. Math.},
      volume={228},
   publisher={Birkh\"{a}user Boston, Boston, MA},
       pages={1\ndash 211},
}

\bib{Kawasaki1973}{article}{
      author={Kawasaki, Tetsuro},
       title={Cohomology of twisted projective spaces and lens complexes},
        date={1973},
        ISSN={0025-5831},
     journal={Math. Ann.},
      volume={206},
       pages={243\ndash 248},
         url={https://doi.org/10.1007/BF01429212},
}

\bib{KeelMori1997}{article}{
      author={Keel, Se\'{a}n},
      author={Mori, Shigefumi},
       title={Quotients by groupoids},
        date={1997},
        ISSN={0003-486X},
     journal={Ann. of Math. (2)},
      volume={145},
      number={1},
       pages={193\ndash 213},
         url={https://doi.org/10.2307/2951828},
}

\bib{KnutsonTao2003}{article}{
      author={Knutson, Allen},
      author={Tao, Terence},
       title={Puzzles and (equivariant) cohomology of {G}rassmannians},
        date={2003},
        ISSN={0012-7094},
     journal={Duke Math. J.},
      volume={119},
      number={2},
       pages={221\ndash 260},
         url={https://doi.org/10.1215/S0012-7094-03-11922-5},
}

\bib{Kollar1997}{article}{
      author={Koll\'{a}r, J\'{a}nos},
       title={Quotient spaces modulo algebraic groups},
        date={1997},
        ISSN={0003-486X},
     journal={Ann. of Math. (2)},
      volume={145},
      number={1},
       pages={33\ndash 79},
         url={https://doi.org/10.2307/2951823},
}

\bib{KollarMori1998}{book}{
      author={Koll\'ar, J\'anos},
      author={Mori, Shigefumi},
       title={Birational geometry of algebraic varieties},
      series={Cambridge Tracts in Mathematics},
   publisher={Cambridge University Press, Cambridge},
        date={1998},
      volume={134},
        ISBN={0-521-63277-3},
  url={https://doi-org.argo.library.okstate.edu/10.1017/CBO9780511662560},
        note={With the collaboration of C. H. Clemens and A. Corti, Translated
  from the 1998 Japanese original},
}

\bib{Kumar1996}{article}{
      author={Kumar, Shrawan},
       title={The nil {H}ecke ring and singularity of {S}chubert varieties},
        date={1996},
        ISSN={0020-9910},
     journal={Invent. Math.},
      volume={123},
      number={3},
       pages={471\ndash 506},
         url={https://doi.org/10.1007/s002220050038},
}

\bib{Mukai2003}{book}{
      author={Mukai, Shigeru},
       title={An introduction to invariants and moduli},
     edition={Japanese},
      series={Cambridge Studies in Advanced Mathematics},
   publisher={Cambridge University Press, Cambridge},
        date={2003},
      volume={81},
        ISBN={0-521-80906-1},
}

\bib{FogartyKirwanMumford1994}{book}{
      author={Mumford, D.},
      author={Fogarty, J.},
      author={Kirwan, F.},
       title={Geometric invariant theory},
     edition={Third},
      series={Ergebnisse der Mathematik und ihrer Grenzgebiete (2) [Results in
  Mathematics and Related Areas (2)]},
   publisher={Springer-Verlag, Berlin},
        date={1994},
      volume={34},
        ISBN={3-540-56963-4},
         url={https://doi.org/10.1007/978-3-642-57916-5},
}

\end{biblist}
\end{bibdiv}
\end{document}